%% file: PD.tex
\documentclass[11pt,letterpaper]{amsart}
\usepackage[foot]{amsaddr}

\usepackage{mathtools}
\usepackage[T1]{fontenc}
\usepackage[latin9]{inputenc}
\usepackage{geometry}
\usepackage{wrapfig}
\usepackage{multicol}
\usepackage{graphicx}
\usepackage{soul}
\usepackage{xcolor}
\usepackage{amssymb}
\usepackage{placeins}
\usepackage{cite}
\usepackage{caption}
\usepackage{enumerate}
\usepackage{afterpage}
\usepackage{enumitem}
\usepackage{bmpsize}
\usepackage{hyperref}
\usepackage{tabu}
\usepackage{enumitem}   
\numberwithin{equation}{section}
\usepackage{stmaryrd}
\usepackage{tikz}
\usetikzlibrary{matrix,graphs,arrows,positioning,calc,decorations.markings,shapes.symbols}

\makeatletter
\renewcommand{\email}[2][]{%
  \ifx\emails\@empty\relax\else{\g@addto@macro\emails{,\space}}\fi%
  \@ifnotempty{#1}{\g@addto@macro\emails{\textrm{(#1)}\space}}%
  \g@addto@macro\emails{#2}%
}
\makeatother

\newtheorem{theorem}{Theorem}[section]
\newtheorem{lemma}[theorem]{Lemma}
\newtheorem{proposition}[theorem]{Proposition}
\newtheorem{conjecture}[theorem]{Conjecture}

{ \theoremstyle{definition}
\newtheorem{definition}[theorem]{Definition}}
{ \theoremstyle{remark}
\newtheorem{remark}[theorem]{Remark}}

\newcommand{\N}{\mathbb{N}}
\newcommand{\Z}{\mathbb{Z}}
\newcommand{\R}{\mathbb{R}}
\newcommand{\E}{\mathbb{E}}

\newcommand{\CL}{\mathcal{L}}

\newcommand{\weyl}{W^\circ}
\newcommand{\GT}{\mathbb{GT}}

\title{Characterization of Brownian Gibbsian line ensembles}
\date{\today}
\author[E. Dimitrov]{Evgeni Dimitrov}
\author[K. Matetski]{Konstantin Matetski}
\address[E. Dimitrov, K. Matetski]{Columbia University, Department of Mathematics, Office 517, 2990 Broadway, New York, NY 10027, USA}
\email[E. Dimitrov]{edimitro@math.columbia.edu}
\email[K. Matetski]{matetski@math.columbia.edu}

\begin{document}

\begin{abstract}
In this paper we show that a Brownian Gibbsian line ensemble is completely characterized by the finite-dimensional marginals of its top curve, i.e. the finite-dimensional sets of the its top curve form a separating class. A particular consequence of our result is that the Airy line ensemble is the unique Brownian Gibbsian line ensemble, whose top curve is the Airy$_2$ process.
\end{abstract}

\maketitle 

\tableofcontents

\input{Section1.tex}
\input{Section2.tex}
\input{Section3.tex}
\input{Section4.tex}
\input{Section5.tex}

\bibliographystyle{amsplain}
\bibliography{PD}

\end{document}

%% file: Section1.tex
\section{Introduction and main result}

%
\subsection{Gibbs measures} Many problems in probability theory and mathematical physics deal with random objects, whose distribution has a {\em Gibbs property}. The term ``Gibbs'' means different things in different contexts, and to illustrate what we mean by it and provide some motivation for our work, we consider a simple model of lozenge tilings of the hexagon. Consider three integers $A, B, C\geq 1$ and the $A \times B \times C$ hexagon drawn on the triangular lattice, see the left part of Figure \ref{S1_1}. By gluing two triangles along a common side, we obtain three types of tiles (also called {\em lozenges}), that are depicted in red, blue and green in Figure \ref{S1_1}. There are finitely many possible ways to tile any given hexagon and we can put the uniform measure on all such tilings. The resulting random tiling model satisfies the following Gibbs property: if we fix a tileable region $K$ in the hexagon, and fix the tiling outside of it then the conditional distribution of the tilings of $K$ is just the uniform measure on all possible tilings of $K$. See the right part of Figure \ref{S1_1}.
\begin{figure}[ht]
\begin{center}
  \includegraphics[scale = 0.40]{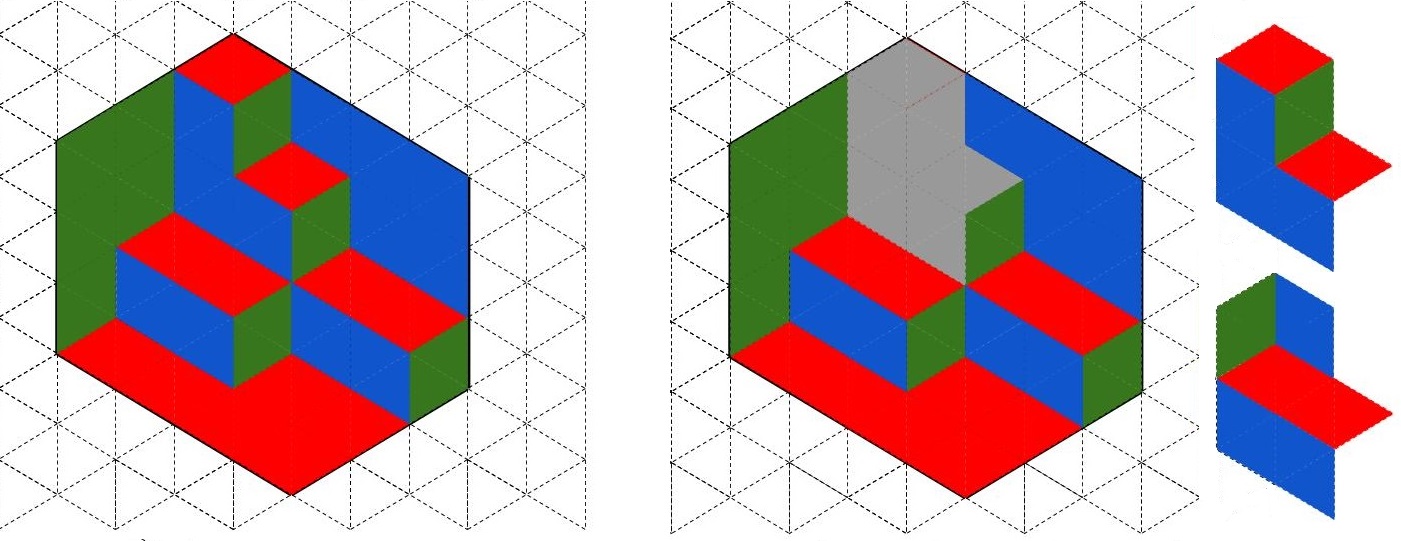}
  \vspace{-2mm}
 \captionsetup{width=0.95\linewidth}
  \caption{ The left part depicts a $3 \times 3 \times 4$ hexagon with a particular tiling. On the right side, a tileable region $K$ is depicted in grey. There are two possible ways to tile $K$, given the tiling outside of it (they are drawn on the very right of the picture). The Gibbs property says that conditioned on the tiling outside of $K$ each of these two tilings is equally likely. }
  \label{S1_1}
  \end{center}
\end{figure}

\begin{figure}[ht]
\begin{center}
  \includegraphics[scale = 0.50]{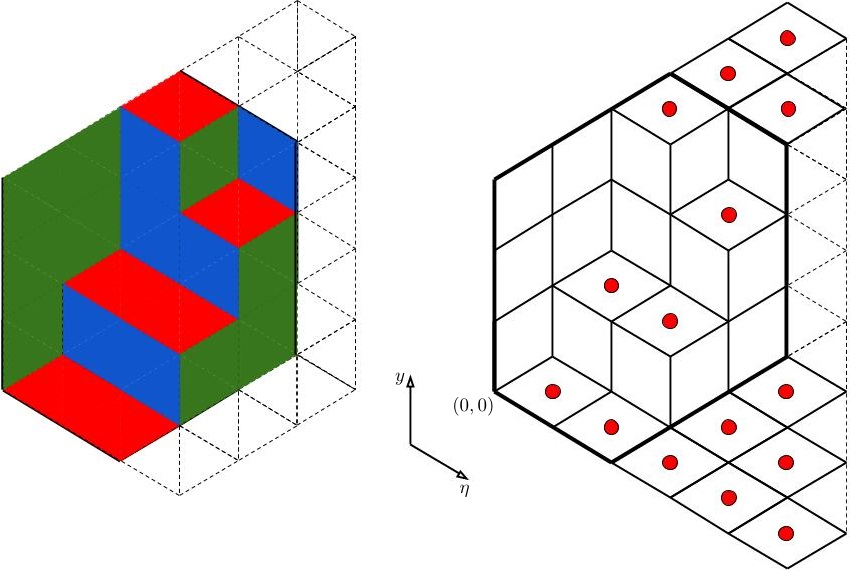}
  \vspace{-2mm}
 \captionsetup{width=0.95\linewidth}
  \caption{Given an element $(\lambda^1, \dots, \lambda^N)  \in \GT^{(B+C)}(\mu)$ with $\mu = (A^C, 0^B)$, we construct an array $y_i^j$ for $1 \leq i \leq j$ and $1 \leq j \leq B+C$ through $y_i^j = \lambda_i^j + j - i + 1/2$ and then plot the points $(i, y_i^j)$ on the triangular grid, denoted as red dots on the right side of the figure. The dots outside of the hexagon are fixed by the interlacing conditions and the positions of the dots inside are distinct for different elements of $\GT^{(B+C)}(\mu)$. The positions of the red dots inside the hexagon specify the locations of the red lozenges, which uniquely determine the tiling. }
  \label{S1_2}
  \end{center}
\end{figure}
An alternative way to represent the above hexagon tiling model is as a random triangular array of interlacing signatures. Specifically, let 
$$\Lambda_k = \{ \lambda \in \mathbb{Z}^k: \lambda_1 \geq \lambda_2 \geq \cdots \geq \lambda_k \}$$
denote the set of signatures of length $k$. Given $N \in \mathbb{N}$ we let 
$$\GT^{(N)} = \{ (\lambda^1,  \dots, \lambda^N): \lambda^k \in \Lambda_k \mbox{ for $k = 1, \dots, N$ and }  \lambda^1 \preceq \lambda^2 \preceq \cdots \preceq \lambda^N\},$$
denote the set of {\em Gelfand-Tsetlin patterns}. The notation $\mu \preceq \lambda$ means that the signatures $\lambda$ and $\mu$ {\em interlace}; i.e. we have $\lambda_1 \geq \mu_1 \geq \lambda_2 \geq \mu_2 \geq \dots$. Finally, given a signature $\mu \in \Lambda^N$, we let $\GT^{(N)}(\mu)$ denote the set of elements in $\GT^{(N)}$ such that $\lambda^N = \mu$. With this notation one can see that lozenge tilings of an $A \times B \times C$ hexagon are in one-to-one correspondence with elements in $\GT^{(B+C)}(\mu)$, where $\mu = (A^C, 0^B)$ (the signature with first $C$ entries equal to $A$ and last $B$ entries equal to $0$). The correspondence is depicted in Figure \ref{S1_2}.

With the above correspondence, we see that the uniform measure on the set of lozenge tilings of the $A \times B\times C$ hexagon is the same as the uniform measure on $\GT^{(B+C)}(\mu)$. In this new notation, the Gibbs property of the beginning of the section can be rephrased as follows. Given any $k \in \{1, \dots, B+C\}$, the conditional distribution of $(\lambda^1, \dots, \lambda^k)$, given $\lambda^k$ is precisely the uniform measure on $\GT^{(k)}({\lambda^k}).$ Both of the Gibbs properties described so far are equivalent to the statement that the lozenge tiling of the hexagon is uniform, and are thus equivalent to each other.\\

There is a natural continuous analogue of the above setting of interlacing triangular arrays, which essentially corresponds to replacing the state space $\mathbb{Z}$ with $\mathbb{R}$. Specifically, let 
$$\mathcal{W}_k = \{ \vec{x} \in \mathbb{R}^k: x_1 \geq x_2 \geq \cdots \geq x_k \},$$
denote the {\em Weyl chamber} in $\mathbb{R}^k$. For $\vec{x} \in \mathbb{R}^n$ and $\vec{y} \in \mathbb{R}^{n-1}$ we write $\vec{y} \preceq \vec{x}$ to mean that
$$x_1 \geq y_1 \geq x_2 \geq y_2 \geq \cdots \geq x_{n-1} \geq y_{n-1} \geq x_n.$$
Given, $N \in \mathbb{N}$ we define the {\em Gelfand-Tsetlin cone} $\GT^{(N)}$ to be 
$$\GT^{(N)} = \{ y \in \mathbb{R}^{N(N+1)/2}: y_i^{j+1} \geq y_i^j \geq y_{i+1}^{j+1}, \mbox{ } 1 \leq i \leq j \leq N-1\}.$$ 
Finally, given $\vec{x} \in \mathcal{W}_N$ we define the {\em Gelfand-Tsetlin polytope}
$$\GT^{(N)}(\vec{x}) = \{ y \in \mathbb{R}^{N(N+1)/2}: y^N_i = x_i \mbox{ for $i = 1, \dots, N$}\}.$$
In this context, we say that a probability measure $\mu$ on $\GT^{(N)}$ is Gibbs, or equivalently satisfies the {\em continuous Gibbs property} \cite{GorinAlt}, if the following condition is satisfied for a $\mu$-distributed random variable $Y$. Given any $k \in \{1, \dots, N\}$ the conditional distribution of $( Y_i^j: 1 \leq i \leq j \leq k)$, given $\vec{Y}^k = (Y_1^k, \dots, Y_k^k)$, is uniform on $\GT^{(k)}(\vec{Y}^k).$

Measures that have the continuous Gibbs property naturally appear in random matrix theory, generally in the context of orbital measures on the space of Hermitian matrices under the action of the unitary group, see e.g. \cite{Def10}. We forgo stating the most general result and illustrate a simple special case coming from the Gaussian Unitary Ensemble (GUE). Recall that the GUE of rank $N$ is the ensemble of random Hermitian matrices $X = \{X_{ij}\}_{i,j = 1}^N$ with probability density proportional to $\exp(-\text{Trace}(X^2)/2)$, with respect to Lebesgue measure. For $r = 1, \ldots, N$  we let $Y_1^r  \geq Y_2^r \geq \cdots \geq Y_r^r$ denote the eigenvalues of the top-left $r \times r$ corner $\{X_{ij}\}_{i,j = 1}^r$. The joint distribution of $(Y_i^j: 1 \leq i \leq j \leq N)$ is known as the {\em GUE-corners} process of rank $N$ (sometimes called the GUE-minors process), and it satisfies the continuous Gibbs property, see \cite{Bar01, Def10}. The fact that the GUE-corners process satisfies the continuous Gibbs property is not a lucky coincidence, but is actually a manifestation of the Gibbs property for tiling models. Indeed, the GUE-corners process is known to be a diffuse limit of random lozenge tiling models, see \cite{JN06}, \cite{Nor09} and \cite{OR}, and under this diffuse scaling the tiling Gibbs property naturally becomes the continuous Gibbs property.\\

There is a different way to interpret a lozenge tilings of the hexagon, which is closer to the topic of the present paper. Specifically, let us perform a simple affine transformation and draw segments connecting the mid-points of the left and right sides of each of the green and blue lozenges, see Figure \ref{Fig1}. In this way a random lozenge tiling corresponds a set of $A$ random curves connecting the left and right side of the hexagon. A natural way to interpret these curves is as trajectories of $A$ Bernoulli random walks, whose starting and ending points are equally spaced points on the two sides of the hexagon, and which have been conditioned to never intersect.  
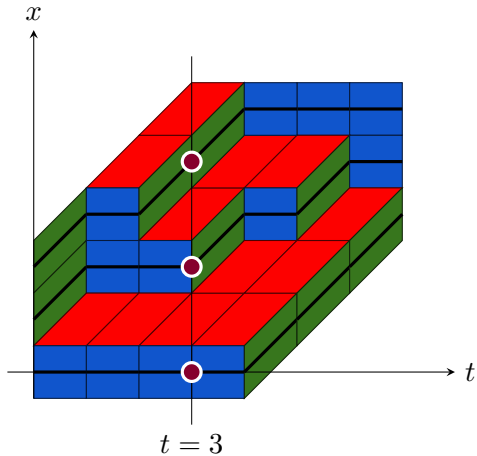
\begin{figure}[ht]
\definecolor{myBlue}{RGB}{16,85,204}
\definecolor{myRed}{RGB}{253,1,0}
\definecolor{myGreen}{RGB}{55,118,29}
\centering
\begin{tikzpicture}[>=stealth,
brxy/.style={fill=myRed, draw=myRed!30!black},
bryz/.style={fill=myGreen, draw=myGreen!30!black},
brxz/.style={fill=myBlue, draw=myBlue!30!black},
pth/.style={very thick, draw=black},
intpt/.style={circle, draw=white!100, fill=purple!70!black, very thick, inner sep=1pt, minimum size=2.5mm},
scale=0.7
]
\def\brxy(#1:#2:#3){
\filldraw[brxy]
(${(#1)+0}*(1,0) + {(#2)+0}*({sqrt(2)*cos(deg(pi/4))},{sqrt(2)*sin(deg(pi/4))}) + {(#3)-0.5}*(0,1)$) --
(${(#1)+1}*(1,0) + {(#2)+0}*({sqrt(2)*cos(deg(pi/4))},{sqrt(2)*sin(deg(pi/4))}) + {(#3)-0.5}*(0,1)$) --
(${(#1)+1}*(1,0) + {(#2)+1}*({sqrt(2)*cos(deg(pi/4))},{sqrt(2)*sin(deg(pi/4))}) + {(#3)-0.5}*(0,1)$) --
(${(#1)+0}*(1,0) + {(#2)+1}*({sqrt(2)*cos(deg(pi/4))},{sqrt(2)*sin(deg(pi/4))}) + {(#3)-0.5}*(0,1)$) -- cycle;
}
\def\bryz(#1:#2:#3){
\filldraw[bryz]
(${(#1)+0}*(1,0) + {(#2)+0}*({sqrt(2)*cos(deg(pi/4))},{sqrt(2)*sin(deg(pi/4))}) + {(#3)-0.5}*(0,1)$) --
(${(#1)+0}*(1,0) + {(#2)+1}*({sqrt(2)*cos(deg(pi/4))},{sqrt(2)*sin(deg(pi/4))}) + {(#3)-0.5}*(0,1)$) --
(${(#1)+0}*(1,0) + {(#2)+1}*({sqrt(2)*cos(deg(pi/4))},{sqrt(2)*sin(deg(pi/4))}) + {(#3)+0.5}*(0,1)$) --
(${(#1)+0}*(1,0) + {(#2)+0}*({sqrt(2)*cos(deg(pi/4))},{sqrt(2)*sin(deg(pi/4))}) + {(#3)+0.5}*(0,1)$) -- cycle;
\draw[pth]
(${(#1)+0}*(1,0) + {(#2)+0}*({sqrt(2)*cos(deg(pi/4))},{sqrt(2)*sin(deg(pi/4))}) + {(#3)}*(0,1)$) --
(${(#1)+0}*(1,0) + {(#2)+1}*({sqrt(2)*cos(deg(pi/4))},{sqrt(2)*sin(deg(pi/4))}) + {(#3)}*(0,1)$);
}
\def\brxz(#1:#2:#3){
\filldraw[brxz]
(${(#1)+0}*(1,0) + {(#2)+0}*({sqrt(2)*cos(deg(pi/4))},{sqrt(2)*sin(deg(pi/4))}) + {(#3)-0.5}*(0,1)$) --
(${(#1)+1}*(1,0) + {(#2)+0}*({sqrt(2)*cos(deg(pi/4))},{sqrt(2)*sin(deg(pi/4))}) + {(#3)-0.5}*(0,1)$) --
(${(#1)+1}*(1,0) + {(#2)+0}*({sqrt(2)*cos(deg(pi/4))},{sqrt(2)*sin(deg(pi/4))}) + {(#3)+0.5}*(0,1)$) --
(${(#1)+0}*(1,0) + {(#2)+0}*({sqrt(2)*cos(deg(pi/4))},{sqrt(2)*sin(deg(pi/4))}) + {(#3)+0.5}*(0,1)$) -- cycle;
\draw[pth]
(${(#1)+0}*(1,0) + {(#2)+0}*({sqrt(2)*cos(deg(pi/4))},{sqrt(2)*sin(deg(pi/4))}) + {(#3)}*(0,1)$) --
(${(#1)+1}*(1,0) + {(#2)+0}*({sqrt(2)*cos(deg(pi/4))},{sqrt(2)*sin(deg(pi/4))}) + {(#3)}*(0,1)$);
}
\brxy(0:1:3);\brxy(0:2:3);\brxz(2:3:2);\brxz(3:3:2);
\bryz(0:0:2);\brxz(0:1:2);\bryz(1:1:2);\brxy(1:1:2);\brxy(1:2:2);\brxy(2:2:2);\bryz(1:2:2);\brxz(2:2:3);
\bryz(0:0:1);\brxz(0:1:1);\brxz(1:1:1);\bryz(2:1:1);\brxz(2:2:1);\bryz(3:2:1);\brxz(3:3:1);
\brxy(0:0:1);\brxy(1:0:1);\brxy(2:0:1);\brxy(2:1:1);\brxy(3:0:1);\brxy(3:1:1);\brxy(3:2:1);
\brxz(0:0:0);\brxz(1:0:0);\brxz(2:0:0);\brxz(3:0:0);\bryz(4:0:0);\bryz(4:1:0);\bryz(4:2:0);
\draw[->] (-0.5,0) -- (8,0) node[right] {$t$};
\draw[->] (0,-0.5) -- (0,6.5) node[above] {$x$};
\draw[-,line width=0pt] (3,6) -- (3,-1) node[below] {$t=3$};
\node[intpt] at (3,0) {}; \node[intpt] at (3,2) {}; \node[intpt] at (3,4) {};
\end{tikzpicture}
\caption{Lozenge tiling of the hexagon and corresponding up-right path configuration. The dots represent the location of the random walks at time $t = 3$.} \label{Fig1}
\end{figure} 

Let us number the random paths from top to bottom by $L_1, L_2, \dots, L_A$, and denote the position of the $k$-th random walk at time $t$ by $L_k(t)$. Then the Gibbs property for the tiling model, can be seen to be equivalent to the following resampling invariance. Suppose that we sample $\{L_m\}_{m = 1}^A$ and fix two times $0 \leq s < t \leq B+C$ and $k_1, k_2 \in \{1, \dots, A\}$ with $k_1 \leq k_2$. We can erase the part of the paths $L_k$ between the points $(s, L_k(s))$ and $(t, L_k(t))$ for $k = k_1, \dots, k_2$ and sample independently $k_2 - k_1 + 1$ up-right paths between these points uniformly from the set of all such paths that do not intersect the lines $L_{k_1-1}$ and $L_{k_2 + 1}$ or each other with the convention that $L_{0} = \infty$ and $L_{A+1} = -\infty$. In this way we obtain a new random collection of paths $\{L'_m\}_{m = 1}^A$ and the essence of the Gibbs property is that the law of $\{L'_m\}_{m = 1}^A$ is the same as that of $\{L_m\}_{m = 1}^A$. 

There is a natural continuous analogue of the above random path formulation, which in this case corresponds to replacing the random walk trajectories with those of Brownian motions. In this continuous context, the random variables of interest take values in $C(\Sigma \times \Lambda)$ -- the space of continuous functions on $\Sigma \times \Lambda$, where $\Sigma = \{1, \dots, N\}$ with $N \in \mathbb{N}$ or $\Sigma = \mathbb{N}$ and $\Lambda \subset \mathbb{R}$ is an interval. We call $C(\Sigma \times \Lambda)$-valued random variables $\mathcal{L}$ {\em line ensembles} (indexed by $\Sigma$ on $\Lambda$). A formal definition of this object is given in Section \ref{Section2.1}, presently it will suffice for us to know that a line ensemble is a collection of at most countably many continuous functions on $\Lambda$, which we number using the index set $\Sigma$. For convenience we denote $\mathcal{L}_{i}(\omega)(x) =\mathcal{L}(\omega)(i,x)$ the $i$-th continuous function (or line) in the ensemble, and typically we drop the dependence on $\omega$ from the notation as one usually does for Brownian motion. The notion of a line ensemble is what replaces the collection of random walk trajectories from the previous paragraph, and we next explain the continuous analogue of the Gibbs property. The description we give is informal, and we postpone the precise formulation to Section \ref{Section2.1} as it requires more notation.

 We say that a probability measure $\mu$ on $C(\Sigma \times \Lambda)$ satisfies the {\em Brownian Gibbs property} if it has the following resampling invariance. Suppose we sample $\mathcal{L}$ according to $\mu$ and fix two times $s, t \in \Lambda$ with $s < t$ and a finite set $K = \{k_1, \dots, k_2\} \subset \Sigma$ with $k_1 \leq k_2$. We can erase the part of the lines $\mathcal{L}_k$ between the points $(s, \mathcal{L}_k(s))$ and $(t, \mathcal{L}_k(t))$ for $k = k_1, \dots, k_2$ and sample independently $k_2 - k_1 + 1$ random curves between these points according to the law of $k_2 - k_1 + 1$ Brownian bridges, which have been conditioned to not intersect the lines $\mathcal{L}_{k_1 - 1}$ and $\mathcal{L}_{k_2 + 1}$ or each other with the convention that $\mathcal{L}_0 = \infty$ and $\mathcal{L}_{k_2 + 1} = -\infty$ if $k_2 + 1 \not \in \Sigma$. In this way we obtain a new random line ensemble $\mathcal{L}'$, and the essence of the Brownian Gibbs property is that the law of $\mathcal{L}'$ is equal to $\mu$. 

While versions of the above definition have appeared earlier in the literature, the term ``Brownian Gibbs property'' was first coined in \cite{CorHamA}. One of the prototypical random models that enjoy the Brownian Gibbs property is {\em Dyson Brownian motion} \cite{Dys62} as has been shown in \cite{CorHamA}. As in the case of the GUE-corners process the latter can be seen as a consequence of the fact that Dyson Brownian motion can be obtained as a diffuse limit of the non-colliding Bernoulli walkers \cite{EK08}; under this limit transition the path resampling interpretation of the tiling Gibbs property naturally becomes the Brownian Gibbs property. \\

The present paper deals with Brownian Gibbsian line ensembles, i.e. line ensembles that satisfy the Brownian Gibbs property. Our main interest comes from the following basic question:\\ {\em How much information does one need in order to uniquely specify the law of a Brownian Gibbsian line ensemble?}

 To begin understanding the latter question let us go back to the case of random variables on $\mathbb{R}^{N(N+1)/2}$, which satisfy the continuous Gibbs property and recall the notion of a {\em separating class} \cite[p.9]{Billing}. Given a class of probability measures $\mathcal{P}$ on the same space $(S, \mathcal{S})$, we call a $\pi$-system of sets $\mathcal{A} \subset \mathcal{S}$ a {\em separating class for $\mathcal{P}$} if the following implication holds 
$$\mu_1, \mu_2 \in \mathcal{P} \mbox{ and }\mu_1(A) = \mu_2(A) \mbox{ for all $A \in \mathcal{A}$ } \implies \mu_1 = \mu_2.$$
If $\mathcal{P}$ denotes the set of all probability measures on $S = \mathbb{R}^{N(N+1)/2}$ and $Y$ is an $S$-valued random variable, it is well-known that the sets of the form
$$ \{ Y_i^j \leq x_i^j : 1 \leq i \leq j \leq N\}, \mbox{ for $x_i^j \in \mathbb{R}$ and $1 \leq i \leq j \leq N$},$$
form a separating class for $\mathcal{P}$, cf. \cite[Example 1.1, p.9]{Billing}. However, if $\mathcal{P}_{Gibbs}$ denotes the set of probability measures on $S = \mathbb{R}^{N(N+1)/2}$ that satisfy the continuous Gibbs property, one can readily see that the sets 
$$ \{ Y_i^N \leq x_i^N : 1 \leq i \leq N\}, \mbox{ for $x_i^N \in \mathbb{R}$ and $i = 1,\dots, N$},$$
form a separating class for $\mathcal{P}_{Gibbs}$. Indeed, since conditionally on $\vec{Y}^N = (Y_1^N, \dots, Y_N^N)$ the law of $( Y_i^j: 1 \leq i \leq j \leq N)$, is uniform $\GT^{(N)}(\vec{Y}^N),$ we see that two Gibbsian probability measures on $S$ are equal the moment their top rows have the same marginal distribution. The essential observation here is that the continuous Gibbs property reduces the amount of information one needs to specify the law of a random variable from order $N^2$ to order $N$; or from dimension $2$ to dimension $1$. 

The main result of the present paper is the analogue of the above statement for Brownian Gibbsian line ensembles and is the content of the following theorem.
\begin{theorem}\label{ThmMainS1}  Let $\Sigma = \{ 1,2, \dots,  N \}$ with $N \in \mathbb{N}$ or $\Sigma =\mathbb{N}$, and let $\Lambda \subset \mathbb{R}$ be an interval. Suppose that $\mathcal{L}^1$ and $\mathcal{L}^2$ are $\Sigma$-indexed line ensembles on $\Lambda$ that satisfy the Brownian Gibbs property with laws $\mathbb{P}_1$ and $\mathbb{P}_2$ respectively. Suppose further that for every $k\in \mathbb{N}$,  $t_1 < t_2 < \cdots < t_k$ with $t_i \in \Lambda$ and $x_1, \dots, x_k \in \mathbb{R}$ we have that 
\begin{equation*}
\mathbb{P}_1 \left( \mathcal{L}^1_1(t_1) \leq x_1, \dots,\mathcal{L}^1_1(t_k) \leq x_k  \right) =\mathbb{P}_2 \left( \mathcal{L}^2_1(t_1) \leq x_1, \dots,\mathcal{L}^2_1(t_k) \leq x_k  \right).
\end{equation*}
Then we have that $\mathbb{P}_1 = \mathbb{P}_2$.
\end{theorem}
\begin{remark}
In plain words, Theorem \ref{ThmMainS1} states that if two line ensembles both satisfy the Brownian Gibbs property and have the same finite-dimensional distributions of the top curve, then they have the same distribution as line ensembles. Equivalently, the finite-dimensional sets of the top curve form a separating class for the space of probability measures with the Brownian Gibbs property.
\end{remark}
\begin{remark}
Theorem \ref{ThmMainS1} is formulated slightly more generally after introducing some necessary notation as Theorem \ref{ThmMain} in the main text.
\end{remark}

The proof of Theorem \ref{ThmMainS1}, or rather its generalization Theorem \ref{ThmMain} in the text, is presented in Section \ref{Section4} and is the main novel contribution of the present paper. The argument is inductive and one roughly shows that if two Brownian Gibbsian line ensembles $\mathcal{L}^1$ and $\mathcal{L}^2$ have the same finite dimensional distributions when restricted to their top $k$ curves then the same is true for when they are restricted to their top $k+1$ curves. The difficulty lies in establishing the induction step, since we are assuming the statement of the theorem for the base case $k = 1$. In going from $k$ to $k+1$ the key idea of the proof is to use the available by induction equality of laws of $\{\mathcal{L}^1_i\}_{i = 1}^k$ and $\{\mathcal{L}^2_i\}_{i = 1}^k$ to construct a family of observables, which are measurable with respect to the top $k$ curves, but which probe the $k+1$-st one. Informally speaking, the law of $\{\mathcal{L}^v_i\}_{i = 1}^k$ is that of $k$ Brownian bridges conditioned on non-intersecting each other and staying above $\mathcal{L}^v_{k+1}$ for $v \in \{1,2\}$. Then the observables we construct measure the difference in behavior between $\{\mathcal{L}^v_i\}_{i = 1}^k$ and that of $k$ Brownian bridges conditioned on non-intersecting each other but being allowed to freely go below $\mathcal{L}^v_{k+1}$. The difference between these two ensembles is negligible when the curve $\mathcal{L}^v_{k+1}$ is below a certain level and non-negligible when it is above it, and a careful analysis of our observables show that they effectively approximate the joint cumulative distribution of the $k+1$-st curve. This allows us to conclude that the restrictions of $\mathcal{L}^1$ and $\mathcal{L}^2$ to their top $k+1$ curves also need to agree in the sense of finite dimensional distributions, which is enough to complete the induction step. The latter description of the main argument is of course quite reductive and the full argument, presented in Section \ref{Section4.1} for a special case and Section \ref{Section4.2} in full generality, relies on various technical statements and definitions that are given in Sections \ref{Section2} and \ref{Section3}. We remark that some of the results we establish in these two sections have appeared in earlier studies on Brownian Gibbsian line ensembles; however, we could not always find complete proofs of them. We have thus opted to fill in the gaps in the proofs of these statements in the literature and this work is the content of the (somewhat) technical Sections \ref{Section3} and \ref{Section5}.\\

We end this section with a brief discussion of some of the motivation behind our work. Our interest in Theorem \ref{ThmMainS1} is twofold. Firstly, Brownian Gibbsian line ensembles have become central objects in probability theory and understanding their structure is an important area of research. As mentioned earlier, Dyson Brownian motion is an example of these ensembles and is a key object in random matrix theory. Other important examples of models that satisfy the Brownian Gibbs property include {\em Brownian last passage percolation}, which has been extensively studied recently in \cite{Ham1, Ham2, Ham3, Ham4} and the {\em Airy line ensemble} (shifted by a parabola) \cite{Spohn, CorHamA}. The second, and more important, reason we believe Theorem \ref{ThmMainS1} to be important is that it can be used as a tool for proving KPZ universality for various models in integrable probability. We elaborate on these points below.

Regarding the first point, there has been some interest in classifying the set of random $\mathbb{N}$-indexed line ensembles that satisfy the Brownian Gibbs property. Specifically, one has the following open problem, which can be found as \cite[Conjecture 3.2]{CorHamA}, see also \cite[Conjecture 1.7]{CorwinSun}.
\begin{conjecture}\label{Conjecture1} The set of extremal Brownian Gibbs $\mathbb{N}$-indexed line ensembles $\mathcal{L}$, which have the property that $\mathcal{A}$ (given by $\mathcal{A}_i(t) = 2^{1/2} \mathcal{L}_i(t) + t^2$ for $i \in \mathbb{N}$) is horizontal shift-invariant, is given by $\{\mathcal{L}^{Airy} + y: y \in \mathbb{R} \}$, where $\mathcal{L}^{Airy}$ denotes the Airy line ensemble, cf. \cite{Spohn} and \cite[Theorem 3.1]{CorHamA}. 
\end{conjecture}
\begin{remark}
Let us explain the terms {\em horizontal shift-invariant} and {\em extremal} in Conjecture \ref{Conjecture1}. We say that an $\mathbb{N}$-indexed line ensemble $\mathcal{A}$ is horizontal shift-invariant if $\mathcal{A}(s + \cdot)$ is equal in distribution to $\mathcal{A}$ for each $s \in \mathbb{R}$. It is relatively easy to see that any convex combination of two laws that satisfy the Brownian Gibbs property also satisfies it. Therefore, the set of Brownian Gibbs measures naturally has the structure of a convex set in the space of measures on $\mathbb{N}$-indexed line ensembles on $\mathbb{R}$. A measure that satisfies the Brownian Gibbs property is then said to be extremal (or {\em ergodic}) if it cannot be written as a non-trivial convex combination of two other measures that satisfy the Brownian Gibbs property. 
\end{remark}
\begin{remark}
We mention that the analogue of Conjecture \ref{Conjecture1} in the context of the probability measures on triangular interlacing arrays we discussed earlier asks about the classification of ergodic measures on the set of infinite triangular arrays that satisfy the continuous Gibbs property. This classification result has been established in the remarkable paper \cite{OV} and has important implications about asymptotic representation theory.
\end{remark}
Beyond its intrinsic interest, Conjecture \ref{Conjecture1} is of considerable interest in light of its possible use as an invariance principle for deriving convergence of systems to the Airy line ensemble, see \cite[Section 2.3.3]{CorHamK} for a discussion of this approach in the context of the KPZ line ensemble. We also mention that in \cite{CorwinSun} the authors showed that $\{\mathcal{L}^{Airy} + y: y \in \mathbb{R} \}$ are ergodic with respect to the action of the translation group on $\mathbb{R}$, which provides further evidence for the validity of the above conjecture.

The relationship between Conjecture \ref{Conjecture1} and our Theorem \ref{ThmMainS1} is somewhat indirect, and in order to compare them we discuss how each classifies the Airy line ensemble. In this context, Conjecture \ref{Conjecture1} says that if $\mathcal{L}$ an extremal Brownian Gibbs $\mathbb{N}$-indexed line ensemble such that $\mathcal{A}$ (given by $\mathcal{A}_i(t) = 2^{1/2} \mathcal{L}_i(t) + t^2$ for $i \in \mathbb{N}$) is horizontal shift-invariant and $\mathbb{E} \left[\mathcal{L}_1(0) \right] = \mathbb{E} \left[\mathcal{L}^{Airy}_1(0) \right]$, then $\mathcal{L}$ has the same law as $\mathcal{L}^{Airy}$. On the other hand, Theorem \ref{ThmMainS1} states that if $\mathcal{L}$ is a Brownian Gibbs $\mathbb{N}$-indexed line ensemble, and $\mathcal{L}_1$ has the same finite dimensional distribution as $\mathcal{L}^{Airy}_1$ then $\mathcal{L}$ has the same law as $\mathcal{L}^{Airy}$. While the conclusions of the two results are the same, we emphasize that the assumptions are quite different. In the case of the conjecture, mostly qualitative information for the ensemble (such as ergodicity and horizontal shift-invariance) is required, and only a bit of quantitative information is needed (mostly to determine the vertical shift $y$). On the other hand, our theorem requires significant quantitative information, specifically the finite dimensional distribution of the top curve$\mathcal{L}_1$; however, it does not require any information about the remaining curves in the ensemble. So in a sense, Theorem \ref{ThmMainS1} requires a lot of quantitative information but only for $\mathcal{L}_1$, while Conjecture \ref{Conjecture1} requires only qualitative information but for the full ensemble. In particular, one result does not imply the other. While it is not clear if Theorem \ref{ThmMainS1} brings us any closer to proving Conjecture \ref{Conjecture1}, we do want to emphasize that the two problems are naturally related as they both characterize the Airy line ensemble in terms of reduced information about the ensemble. In addition, similarly to Conjecture \ref{Conjecture1} we also hope that Theorem \ref{ThmMainS1} can be as a tool for deriving convergence of systems to the Airy line ensemble as we explain next.

The Airy line ensemble, first introduced in \cite{Spohn} and later extensively studied in \cite{CorHamA}, is believed to be a universal scaling limit for various models that belong to the so-called {\em KPZ universality class}, see \cite{CU2} for an expository review of this class. In \cite{Spohn} the convergence to the Airy line ensemble (in the finite dimensional sense) was established for the polynuclear growth model and in \cite{CorHamA} it was shown for Dyson Brownian motion (in a stronger uniform sense). Very recently, \cite{DNV} established the uniform convergence of various classical integrable models to the Airy line ensemble including non-colliding Bernoulli walks and geometric, Poisson and Brownian last passage percolation. We refer to the introduction of \cite{DNV} for a more extensive discussion of the history, motivation behind and progress on the problem of establishing convergence to the Airy line ensemble.

The approach taken in \cite{DNV} relies on obtaining finite dimensional convergence to the Airy line ensemble as a prerequisite for obtaining uniform convergence. Theorem \ref{ThmMainS1} paves a different way to showing uniform convergence to the Airy line ensemble, where establishing finite dimensional convergence is required only for the top curve of the ensemble. We believe that the latter approach is more suitable for models, which naturally have the structure of a line ensemble and for which the finite dimensional marginals of the top curve are easier to access. The primary examples, we are interested in applying this approach to, come from the {\em Macdonald processes} \cite{BorCor}, and include the {\em Hall-Littlewood processes} \cite{ED, CD}, the {\em $q$-Whittaker processes} \cite{BCF}, the {\em log-gamma polymer} \cite{Sep12, COSZ}, the {\em semi-discrete polymer} \cite{OCY} and the {\em mixed polymer model} of \cite{BCFV}. Each of the models we listed naturally has the structure of a line ensemble with a Gibbs property, which can be found for the Hall-Littlewood process in \cite{CD} and for the log-gamma polymer in \cite{Wu19}. If we denote by $\{L^N_i\}_{i = 1}^\infty$ the discrete line ensemble associated to one of the above models and $\{\mathcal{L}_i^{Airy}\}_{i = 1}^\infty$ the Airy line ensemble, the proposed program for establishing the uniform convergence of $\{L^N_i \}_{i = 1}^\infty$  to $\{\mathcal{L}_i^{Airy}\}_{i = 1}^\infty$ goes through the following steps:
\begin{enumerate}
\item show that $L^N_1$ converges in the sense of finite dimensional distributions to $\mathcal{L}^{Airy}_1$, which is the Airy$_2$ process, as $N \rightarrow \infty$;
\item show that $\left\{L^N_i \right\}_{i = 1}^\infty$ form a tight sequence of line ensembles and that every subsequential limit enjoys the Brownian Gibbs property;
\item use the characterization of Theorem \ref{ThmMainS1} to prove that all subsequential limits are given by $\{\mathcal{L}_i^{Airy}\}_{i = 1}^\infty$.
\end{enumerate}
The difference between the above program and the approach of \cite{DNV} is that in the latter the necessity of showing that any subsequential limit satisfies the Brownian Gibbs property is omitted from (2), but one is required to show the finite dimensional convergence in (1), not just for the top line $L_1^N$ but for all of the lines. This approach is best suited for {\em determinantal point processes}, for which the finite dimensional formulas are readily available and their asymptotics fairly well-understood. A common feature of all of the above models coming from the Macdonald processes, is that they are no longer determinantal and formulas suitable for taking asymptotics are unknown for all of the lines. One reason we are optimistic that our proposed program has a better chance of establishing convergence to the Airy line ensemble for these models is that there are non-determinantal formulas that allow one to study one-point marginals of $L_1^N$, see e.g. \cite{ED,CD,BCFV,KQ,BorCor} and also there is some progress on understanding the multi-point asymptotics of $L_1^N$ for the case of the log-gamma polymer \cite{NZ}. Another reason we are optimistic about our proposed program is that its analogue for the triangular interlacing arrays was successfully implemented in \cite{ED2} to prove the convergence of a class of six-vertex models to the GUE-corners process.

Even beyond the above program, we believe that Theorem \ref{ThmMainS1} will be useful in reducing some of the work in showing convergence to the Airy line ensemble, and is an important result that furthers our understanding of Gibbsian line ensembles in general. \\

%
\subsection{Outline of the paper}

The structure of this paper is as follows. In Section~\ref{Section2} we make crucial definitions which are used throughout the paper. In particular, we  define avoiding Brownian line ensembles and introduce the standard and partial Brownian Gibbs properties. The main result of this paper is stated in this section as Theorem \ref{ThmMain}. In Section~\ref{Section3} we collect several properties of Brownian line ensembles and in Section~\ref{Section4} a proof of Theorem \ref{ThmMain} is provided. In Section~\ref{Section5} we prove several technical results, which include the construction of monotonically coupled Brownian line ensembles and a proof of the statement that non-intersecting Brownian bridges satisfy the Brownian Gibbs property.\\

\noindent {\bf Acknowledgments.} We would like to thank Alexei Borodin, Ivan Corwin and Vadim Gorin for useful comments on earlier drafts of this paper. We also thank Julien Dub{\'e}dat for suggesting some useful literature and Alisa Knizel for some of the figures. Both authors are partially supported by the Minerva Foundation Fellowship.

%% file: Section2.tex
\section{Definitions, main result and basic lemmas}\label{Section2}

In this section we introduce the basic definitions that are necessary for formulating our main result, given in Section \ref{Section2.2} below. In Section \ref{Section2.3} we state several lemmas used later in the paper.

%
%
\subsection{Line ensembles and the (partial) Brownian Gibbs property}\label{Section2.1}
In order to state our main results we need to introduce some notation as well as the notions of a {\em line ensemble} and the {\em (partial) Brownian Gibbs property}. Our exposition in this section closely follows that of \cite[Section 2]{CorHamA}. 

Given two integers $p \leq q$, we let $\llbracket p, q \rrbracket$ denote the set $\{p, p+1, \dots, q\}$. Given an interval $\Lambda \subset \mathbb{R}$ we endow it with the subspace topology of the usual topology on $\mathbb{R}$. We let $(C(\Lambda), \mathcal{C})$ denote the space of continuous functions $f: \Lambda \rightarrow \mathbb{R}$ with the topology of uniform convergence over compacts, see \cite[Chapter 7, Section 46]{Munkres}, and Borel $\sigma$-algebra $\mathcal{C}$. Given a set $\Sigma \subset \mathbb{Z}$ we endow it with the discrete topology and denote by $\Sigma \times \Lambda$ the set of all pairs $(i,x)$ with $i \in \Sigma$ and $x \in \Lambda$ with the product topology. We also denote by $\left(C (\Sigma \times \Lambda), \mathcal{C}_{\Sigma}\right)$ the space of continuous functions on $\Sigma \times \Lambda$ with the topology of uniform convergence over compact sets and Borel $\sigma$-algebra $\mathcal{C}_{\Sigma}$. Typically, we will take $\Sigma = \llbracket 1, N \rrbracket$ (we use the convention $\Sigma = \mathbb{N}$ if $N = \infty$) and then we write  $\left(C (\Sigma \times \Lambda), \mathcal{C}_{|\Sigma|}\right)$ in place of $\left(C (\Sigma \times \Lambda), \mathcal{C}_{\Sigma}\right)$.
The following defines the notion of a line ensemble.
\begin{definition}
Let $\Sigma \subset \mathbb{Z}$ and $\Lambda \subset \mathbb{R}$ be an interval. A {\em $\Sigma$-indexed line ensemble $\mathcal{L}$} is a random variable defined on a probability space $(\Omega, \mathcal{F}, \mathbb{P})$ that takes values in $\left(C (\Sigma \times \Lambda), \mathcal{C}_{\Sigma}\right)$. Intuitively, $\mathcal{L}$ is a collection of random continuous curves (sometimes referred to as {\em lines}), indexed by $\Sigma$,  each of which maps $\Lambda$ in $\mathbb{R}$. We will often slightly abuse notation and write $\mathcal{L}: \Sigma \times \Lambda \rightarrow \mathbb{R}$, even though it is not $\mathcal{L}$ which is such a function, but $\mathcal{L}(\omega)$ for every $\omega \in \Omega$. For $i \in \Sigma$ we write $\mathcal{L}_i(\omega) = (\mathcal{L}(\omega))(i, \cdot)$ for the curve of index $i$ and note that the latter is a map $\mathcal{L}_i: \Omega \rightarrow C(\Lambda)$, which is $(\mathcal{C}, \mathcal{F})-$measurable.

Given a sequence $\{ \mathcal{L}^n: n \in \mathbb{N} \}$ of random $\Sigma$-indexed line ensembles we say that $\mathcal{L}^n$ {\em converge weakly} to a line ensemble $\mathcal{L}$, and write $\mathcal{L}^n \implies \mathcal{L}$ if for any bounded continuous function $f: C (\Sigma \times \Lambda) \rightarrow \mathbb{R}$ we have that 
$$\lim_{n \rightarrow \infty} \mathbb{E} \left[ f(\mathcal{L}^n) \right] = \mathbb{E} \left[ f(\mathcal{L}) \right].$$
We call a line ensemble {\em non-intersecting} if $\mathbb{P}$-almost surely $\mathcal{L}_i(r) > \mathcal{L}_j(r)$  for all $i < j$ and $r \in \Lambda$.
\end{definition}

We next turn to formulating the Brownian Gibbs property -- we do this in Definition \ref{DefBGP} after introducing some relevant notation and results. If $W_t$ denotes a standard one-dimensional Brownian motion, then the process
$$\tilde{B}(t) =  W_t - t W_1, \hspace{5mm} 0 \leq t \leq 1,$$
is called a {\em Brownian bridge (from $\tilde{B}(0) = 0$ to $\tilde{B}(1) = 0 $) with diffusion parameter $1$.} For brevity we call the latter object a {\em standard Brownian bridge}.

Given $a, b,x,y \in \mathbb{R}$ with $a < b$ we define a random variable on $(C([a,b]), \mathcal{C})$ through
\begin{equation}\label{BBDef}
B(t) = (b-a)^{1/2} \cdot \tilde{B} \left( \frac{t - a}{b-a} \right) + \left(\frac{b-t}{b-a} \right) \cdot x + \left( \frac{t- a}{b-a}\right) \cdot y, 
\end{equation}
and refer to the law of this random variable as a {\em Brownian bridge (from $B(a) = x$ to $B(b) = y$) with diffusion parameter $1$.} Given $k \in \mathbb{N}$ and $\vec{x}, \vec{y} \in \mathbb{R}^k$ we let $\mathbb{P}^{a,b, \vec{x},\vec{y}}_{free}$ denote the law of $k$ independent Brownian bridges $\{B_i: [a,b] \rightarrow \mathbb{R} \}_{i = 1}^k$ from $B_i(a) = x_i$ to $B_i(b) = y_i$ all with diffusion parameter $1$.

We next state a couple of results about Brownian bridges from \cite{CorHamA} for future use.
\begin{lemma}\label{NoTouch} \cite[Corollary 2.9]{CorHamA}. Fix a continuous function $f: [0,1] \rightarrow \mathbb{R}$ such that $f(0) > 0$ and $f(1) > 0$. Let $B$ be a standard Brownian bridge and let $C = \{ B(t) > f(t) \mbox{ for some $t \in [0,1]$}\}$ (crossing) and $T = \{ B(t) = f(t) \mbox{ for some } t\in [0,1]\}$ (touching). Then $\mathbb{P}(T \cap C^c) = 0.$
\end{lemma}
\begin{lemma}\label{Spread} \cite[Corollary 2.10]{CorHamA}. Let $U$ be an open subset of $C([0,1])$, which contains a function $f$ such that $f(0) = f(1) = 0$. If $B:[0,1] \rightarrow \mathbb{R}$ is a standard Brownian bridge then $\mathbb{P}(B[0,1] \subset U) > 0$.
\end{lemma}

The following definition introduces the notion of an $(f,g)$-avoiding Brownian line ensemble, which in plain words can be understood as a random ensemble of $k$ independent Brownian bridges, conditioned on not-crossing each other and staying above the graph of $g$ and below the graph of $f$ for two continuous functions $f$ and $g$.
\begin{definition}\label{DefAvoidingLaw}
Let $k \in \mathbb{N}$ and $\weyl_k$ denote the open Weyl chamber in $\mathbb{R}^k$, i.e.
$$\weyl_k = \{ \vec{x} = (x_1, \dots, x_k) \in \mathbb{R}^k: x_1 > x_2 > \cdots > x_k \}$$
(in \cite{CorHamA} the notation $\mathbb{R}_{>}^k$ was used for this set).
Let $\vec{x}, \vec{y} \in \weyl_k$, $a,b \in \mathbb{R}$ with $a < b$, and $f: [a,b] \rightarrow (-\infty, \infty]$ and $g: [a,b] \rightarrow [-\infty, \infty)$ be two continuous functions. The latter condition means that either $f: [a,b] \rightarrow \mathbb{R}$ is continuous or $f = \infty$ everywhere, and similarly for $g$. We also assume that $f(t) > g(t)$ for all $t \in[a,b]$, $f(a) > x_1, f(b) > y_1$ and $g(a) < x_k, g(b) < y_k.$

With the above data we define the {\em $(f,g)$-avoiding Brownian line ensemble on the interval $[a,b]$ with entrance data $\vec{x}$ and exit data $\vec{y}$} to be the $\Sigma$-indexed line ensemble $\mathcal{Q}$ with $\Sigma = \llbracket 1, k\rrbracket$ on $\Lambda = [a,b]$ and with the law of $\mathcal{Q}$ equal to $\mathbb{P}^{a,b, \vec{x},\vec{y}}_{free}$ (the law of $k$ independent Brownian bridges $\{B_i: [a,b] \rightarrow \mathbb{R} \}_{i = 1}^k$ from $B_i(a) = x_i$ to $B_i(b) = y_i$) conditioned on the event 
$$E  = \left\{ f(r) > B_1(r) > B_2(r) > \cdots > B_k(r) > g(r) \mbox{ for all $r \in[a,b]$} \right\}.$$ 

Let us elaborate on the above formulation briefly.  Let $\left(\Omega, \mathcal{F}, \mathbb{P}\right)$ be a probability measure that supports $k$ independent Brownian bridges $\{B_i: [a,b] \rightarrow \mathbb{R} \}_{i = 1}^k$ from $B_i(a) = x_i$ to $B_i(b) = y_i$ all with diffusion parameter $1$. Notice that we can find $\tilde{u}_1, \dots, \tilde{u}_k \in C([0,1])$ and $\epsilon > 0$ (depending on $\vec{x}, \vec{y}, f, g, a, b$) such that $\tilde{u}_i(0) = \tilde{u}_i(1) = 0$ for $i = 1, \dots, k$ and such that if $\tilde{h}_1, \dots, \tilde{h}_k \in C([0,1])$ satisfy $\tilde{h}_i(0) = \tilde{h}_i(1) = 0$ for $i = 1, \dots, k$ and $\sup_{t \in [0,1]}|\tilde{u}_i(t) - \tilde{h}_i(t)| < \epsilon$ then the functions
$$h_i(t) = (b-a)^{1/2} \cdot \tilde{h}_i \left( \frac{t - a}{b-a} \right) + \left(\frac{b-t}{b-a} \right) \cdot x_i + \left( \frac{t- a}{b-a}\right) \cdot y_i,$$ 
satisfy $f(r) > h_1(r) > \cdots > h_k(r) > g(r)$. It follows from Lemma \ref{Spread} that 
$$\mathbb{P}(E) \geq \mathbb{P}\left(\max_{1 \leq i \leq k} \sup_{r \in [0,1]}|\tilde{B}_i(r) - \tilde{u}_i(r)| < \epsilon \right) = \prod_{i = 1}^k \mathbb{P} \left( \sup_{r \in [0,1]}|\tilde{B}_i(r) - \tilde{u}_i(r)| < \epsilon  \right)> 0,$$
 and so we can condition on the event $E$. 

To construct a realization of $\mathcal{Q}$ we proceed as follows. For $\omega \in E$ we define
$$\mathcal{Q}(\omega)(i,r) = B_i(r)(\omega) \mbox{ for $i = 1, \dots, k$ and $r \in [a,b]$}.$$
Observe that for $i \in \{1, \dots, k\}$ and an open set $U \in C([a,b])$ we have that 
$$\mathcal{Q}^{-1}(\{i\} \times U) = \{B_i \in U \} \cap E\; \in\; \mathcal{F},$$
and since the sets $\{i\} \times U$ form an open basis of $C(\llbracket 1, k \rrbracket \times [a,b])$ we conclude that $\mathcal{Q}$ is $\mathcal{F}$-measurable. This implies that the law $\mathcal{Q}$ is indeed well-defined and also it is non-intersecting almost surely.  Also, given measurable subsets $A_1, \dots, A_k$ of $C([a,b])$ we have that 
$$\mathbb{P}(\mathcal{Q}_i \in A_i \mbox{ for $i = 1, \dots, k$} ) = \frac{\mathbb{P}^{a,b, \vec{x},\vec{y}}_{free} \left( \{ B_i \in A_i \mbox{ for $i = 1, \dots, k$}\} \cap E \right) }{\mathbb{P}^{a,b, \vec{x},\vec{y}}_{free}(E)}.$$
We denote the probability distribution of $\mathcal{Q}$ as $\mathbb{P}_{avoid}^{a,b, \vec{x}, \vec{y}, f, g}$ and write $\mathbb{E}_{avoid}^{a,b, \vec{x}, \vec{y}, f, g}$ for the expectation with respect to this measure.
\end{definition}

The following definition introduces the notion of the Brownian Gibbs property from \cite{CorHamA}.
\begin{definition}\label{DefBGP}
Fix a set $\Sigma = \llbracket 1, N \rrbracket$ with $N \in \mathbb{N}$ or $N = \infty$ and an interval $\Lambda \subset \mathbb{R}$ and let $K = \{k_1, k_1 + 1, \dots, k_2 \} \subset \Sigma$ be finite and $a,b \in \Lambda$ with $a < b$. Set $f = \mathcal{L}_{k_1 - 1}$ and $g = \mathcal{L}_{k_2 + 1}$ with the convention that $f = \infty$ if $k_1 - 1 \not \in \Sigma$ and $g = -\infty$ if $k_2 +1 \not \in \Sigma$. Write $D_{K,a,b} = K \times (a,b)$ and $D_{K,a,b}^c = (\Sigma \times \Lambda) \setminus D_{K,a,b}$. A $\Sigma$-indexed line ensemble $\mathcal{L} : \Sigma \times \Lambda \rightarrow \mathbb{R}$ is said to have the {\em Brownian Gibbs property} if it is non-intersecting and 
$$\mbox{ Law}\left( \mathcal{L}|_{K \times [a,b]} \mbox{ conditional on } \mathcal{L}|_{D^c_{K,a,b}} \right)= \mbox{Law} \left( \mathcal{Q} \right),$$
where $\mathcal{Q}_i = \tilde{\mathcal{Q}}_{i - k_1 + 1}$ and $\tilde{\mathcal{Q}}$ is the $(f,g)$-avoiding Brownian line ensemble on $[a,b]$ with entrance data $(\mathcal{L}_{k_1}(a), \dots, \mathcal{L}_{k_2}(a))$ and exit data $(\mathcal{L}_{k_1}(b), \dots, \mathcal{L}_{k_2}(b))$ from Definition \ref{DefAvoidingLaw}. Note that $\tilde{Q}$ is introduced because, by definition, any such $(f,g)$-avoiding Brownian line ensemble is indexed from $1$ to $k_2 - k_1 + 1$ but we want $\mathcal{Q}$ to be indexed from $k_1$ to $k_2$.

A more precise way to express the Brownian Gibbs property is as follows. A $\Sigma$-indexed line ensemble $\mathcal{L}$ on $\Lambda$ satisfies the Brownian Gibbs property if and only if it is non-intersecting and for any finite $K = \{k_1, k_1 + 1, \dots, k_2 \} \subset \Sigma$ and $[a,b] \subset \Lambda$ and any bounded Borel-measurable function $F: C(K \times [a,b]) \rightarrow \mathbb{R}$ we have $\mathbb{P}$-almost surely
\begin{equation}\label{BGPTower}
\mathbb{E} \left[ F\left(\mathcal{L}|_{K \times [a,b]} \right)  {\big \vert} \mathcal{F}_{ext} (K \times (a,b))  \right] =\mathbb{E}_{avoid}^{a,b, \vec{x}, \vec{y}, f, g} \bigl[ F(\tilde{\mathcal{Q}}) \bigr],
\end{equation}
where
$$\mathcal{F}_{ext} (K \times (a,b)) = \sigma \left \{ \mathcal{L}_i(s): (i,s) \in D_{K,a,b}^c \right\}$$
is the $\sigma$-algebra generated by the variables in the brackets above, $ \mathcal{L}|_{K \times [a,b]}$ denotes the restriction of $\mathcal{L}$ to the set $K \times [a,b]$, $\vec{x} = (\mathcal{L}_{k_1}(a), \dots, \mathcal{L}_{k_2}(a))$, $\vec{y} = (\mathcal{L}_{k_1}(b), \dots, \mathcal{L}_{k_2}(b))$, $f = \mathcal{L}_{k_1 - 1}[a,b]$ (the restriction of $\mathcal{L}$ to the set $\{k_1 - 1 \} \times [a,b]$) with the convention that $f = \infty$ if $k_1 - 1 \not \in \Sigma$, and $g = \mathcal{L}_{k_2 +1}[a,b]$ with the convention that $g =-\infty$ if $k_2 +1 \not \in \Sigma$. 
\end{definition}
\begin{remark}\label{RemMeas} It is perhaps worth explaining why equation (\ref{BGPTower}) makes sense. Firstly, since $\Sigma \times \Lambda$ is locally compact, we know by \cite[Lemma 46.4]{Munkres} that $\mathcal{L} \rightarrow \mathcal{L}|_{K \times [a,b]}$ is a continuous map from $C(\Sigma \times \Lambda)$ to $C(K \times [a,b])$, so that the left side of (\ref{BGPTower}) is the conditional expectation of a bounded measurable function, and is thus well-defined. A more subtle question is why the right side of (\ref{BGPTower})  is $\mathcal{F}_{ext} (K \times (a,b))$-measurable. In fact we will show in Lemma \ref{LemmaMeasExp} that the right side is measurable with respect to the $\sigma$-algebra 
$$ \sigma \left\{ \mathcal{L}_i(s) : \mbox{  $i \in K$ and $s \in \{a,b\}$, or $i \in \{k_1 - 1, k_2 +1 \}$ and $s \in [a,b]$} \right\}.$$
\end{remark}

In the present paper it will be convenient for us to use the following modified version of the definition above, which we call the {\em partial Brownian Gibbs property}. We explain the difference between the two definitions, and why we prefer the second one in Remark \ref{RPBGP}.
\begin{definition}\label{DefPBGP}
Fix a set $\Sigma = \llbracket 1 , N \rrbracket$ with $N \in \mathbb{N}$ or $N  = \infty$ and an interval $\Lambda \subset \mathbb{R}$.  A $\Sigma$-indexed line ensemble $\mathcal{L}$ on $\Lambda$ is said to satisfy the {\em partial Brownian Gibbs property} if and only if it is non-intersecting and for any finite $K = \{k_1, k_1 + 1, \dots, k_2 \} \subset \Sigma$ with $k_2 \leq N - 1$ (if $\Sigma \neq \mathbb{N}$), $[a,b] \subset \Lambda$ and any bounded Borel-measurable function $F: C(K \times [a,b]) \rightarrow \mathbb{R}$ we have $\mathbb{P}$-almost surely
\begin{equation}\label{PBGPTower}
\mathbb{E} \left[ F(\mathcal{L}|_{K \times [a,b]}) {\big \vert} \mathcal{F}_{ext} (K \times (a,b))  \right] =\mathbb{E}_{avoid}^{a,b, \vec{x}, \vec{y}, f, g} \bigl[ F(\tilde{\mathcal{Q}}) \bigr],
\end{equation}
where we recall that $D_{K,a,b} = K \times (a,b)$ and $D_{K,a,b}^c = (\Sigma \times \Lambda) \setminus D_{K,a,b}$, and
$$\mathcal{F}_{ext} (K \times (a,b)) = \sigma \left \{ \mathcal{L}_i(s): (i,s) \in D_{K,a,b}^c \right\}$$
is the $\sigma$-algebra generated by the variables in the brackets above, $ \mathcal{L}|_{K \times [a,b]}$ denotes the restriction of $\mathcal{L}$ to the set $K \times [a,b]$, $\vec{x} = (\mathcal{L}_{k_1}(a), \dots, \mathcal{L}_{k_2}(a))$, $\vec{y} = (\mathcal{L}_{k_1}(b), \dots, \mathcal{L}_{k_2}(b))$, $f = \mathcal{L}_{k_1 - 1}[a,b]$ with the convention that $f = \infty$ if $k_1 - 1 \not \in \Sigma$, and $g = \mathcal{L}_{k_2 +1}[a,b]$.
\end{definition}
\begin{remark} 
Observe that if $N = 1$ then the conditions in Definition \ref{DefPBGP} become void. I.e., any line ensemble with one line satisfies the partial Brownian Gibbs property. Also we mention that (\ref{PBGPTower}) makes sense by the same reason that (\ref{BGPTower}) makes sense, see Remark \ref{RemMeas}.
\end{remark}
\begin{remark}\label{RPBGP}
Definition \ref{DefPBGP} is slightly different from the Brownian Gibbs property of Definition~\ref{DefBGP} as we explain here. Assuming that $\Sigma = \mathbb{N}$ the two definitions are equivalent. However, if $\Sigma = \{1, \dots, N\}$ with $1 \leq N < \infty$ then a line ensemble that satisfies the Brownian Gibbs property also satisfies the partial Brownian Gibbs property, but the reverse need not be true. Specifically, the Brownian Gibbs property allows for the possibility that $k_2 = N$ in Definition \ref{DefPBGP} and in this case the convention is that $g = -\infty$. A distinct advantage of working with the partial Brownian Gibbs property instead of the Brownian Gibbs property is that the former is stable under projections, while the latter is not. Specifically, if $1 \leq M \leq N$ and $\mathcal{L}$ is a $\llbracket 1, N\rrbracket$-indexed line ensemble on $\Lambda$ that satisfies the partial Brownian Gibbs property, and $\tilde{\mathcal{L}}$ is obtained from $\mathcal{L}$ by projecting on $(\mathcal{L}_1, \dots, \mathcal{L}_M)$ then the induced law on $\tilde{\mathcal{L}}$ also satisfies the partial Brownian Gibbs property as a $\llbracket 1, M \rrbracket$-indexed line ensemble on $\Lambda$. Later in the text, some of our arguments rely on an induction on $N$, for which having this projectional stability becomes important. This is why we choose to work with the partial Brownian Gibbs property instead of the Brownian Gibbs property.
\end{remark}

%
%
\subsection{Main result}\label{Section2.2} In this section we formulate the main result of the paper. We continue with the same notation as in Section \ref{Section2.1}.
\begin{theorem}\label{ThmMain}  Let $\Sigma = \llbracket 1, N \rrbracket$ with $N \in \mathbb{N}$ or $N =\infty$, and let $\Lambda \subset \mathbb{R}$ be an interval. Suppose that $\mathcal{L}^1$ and $\mathcal{L}^2$ are $\Sigma$-indexed line ensembles on $\Lambda$ that satisfy the partial Brownian Gibbs property with laws $\mathbb{P}_1$ and $\mathbb{P}_2$ respectively. Suppose further that for every $k\in \mathbb{N}$,  $t_1 < t_2 < \cdots < t_k$ with $t_i \in \Lambda$ and $x_1, \dots, x_k \in \mathbb{R}$ we have that 
\begin{equation*}
\mathbb{P}_1 \left( \mathcal{L}^1_1(t_1) \leq x_1, \dots,\mathcal{L}^1_1(t_k) \leq x_k  \right) =\mathbb{P}_2 \left( \mathcal{L}^2_1(t_1) \leq x_1, \dots,\mathcal{L}^2_1(t_k) \leq x_k  \right).
\end{equation*}
Then we have that $\mathbb{P}^1 = \mathbb{P}^2$.
\end{theorem}

In plain words, Theorem \ref{ThmMain} states that if two line ensembles both satisfy the partial Brownian Gibbs property and have the same finite-dimensional distributions of the top curve, then they have the same distribution as line ensembles. Equivalently, a Brownian Gibbsian line ensemble is completely characterized by the finite-dimensional distribution of its top curve. 

One of the assumptions in Theorem \ref{ThmMain} is that $\mathcal{L}^1$ and $\mathcal{L}^2$ have the same number of curves $N$, and a natural question is whether this condition can be relaxed. That is, can two Brownian Gibbsian line ensembles with a {\em different} number of curves have the same finite-dimensional distributions of the top curve. The answer to this question is negative and we isolate this statement in the following corollary.
\begin{theorem}\label{CorMain2}  Let $\Sigma_1 = \llbracket 1, N_1 \rrbracket$ with $N_1 \in \mathbb{N}$ and $\Sigma_2 = \llbracket 1, N_2 \rrbracket$ with $N_2 \in \mathbb{N}$ or $N_2 = \infty$ such that $N_2 > N_1$. In addition, let $\Lambda \subset \mathbb{R}$ be an interval. Suppose that $\mathcal{L}^i$ are $\Sigma_i$-indexed line ensembles on $\Lambda$ for $i = 1,2$ such that $\mathcal{L}^1$ satisfies the Brownian Gibbs property and $\mathcal{L}^2$ satisfies the partial Brownian Gibbs property with laws $\mathbb{P}_1$ and $\mathbb{P}_2$ respectively.
Then there exist $k\in \mathbb{N}$,  $t_1 < t_2 < \cdots < t_k$ with $t_i \in \Lambda$ and $x_1, \dots, x_k \in \mathbb{R}$ such that
\begin{equation*}
\mathbb{P}_1 \left( \mathcal{L}^1_1(t_1) \leq x_1, \dots,\mathcal{L}^1_1(t_k) \leq x_k  \right) \neq \mathbb{P}_2 \left( \mathcal{L}^2_1(t_1) \leq x_1, \dots,\mathcal{L}^2_1(t_k) \leq x_k  \right).
\end{equation*}
\end{theorem}
\begin{remark} It is important that $\mathcal{L}^1$ satisfies the usual rather than the partial Brownian Gibbs property in Corollary \ref{CorMain2}. Indeed, otherwise one could take $\mathcal{L}^2$ and project this line ensemble to its top $N_1$ curves. The resulting $\Sigma_1$-indexed line ensemble on $\Lambda$ will have the same top curve distribution as $\mathcal{L}^2$ and also satisfy the partial Brownian Gibbs property -- see Remark \ref{RPBGP}. In a sense $\mathcal{L}^1$ can be understood as a line ensemble with $N_1+1$ curves with the $(N_1+1)$-st curve sitting at $-\infty$, while $\mathcal{L}^2$ has a $(N_1+1)$-st curve that is finite-valued and the question that Corollary \ref{CorMain2} answers in the affirmative is whether we can distinguish between these two cases using only the top curve of the line ensemble. We are grateful to Vadim Gorin who suggested this question after reading a preliminary draft of the paper.
\end{remark}

%
%
\subsection{Basic lemmas}\label{Section2.3} 
In this section we present three lemmas, whose proof is postponed until Section~\ref{Section5}. Lemma \ref{AvoidIsGibbs} states that a line ensemble with distribution $\mathbb{P}_{avoid}^{a,b,\vec{x},\vec{y},\infty,-\infty}$ from Definition \ref{DefAvoidingLaw} satisfies the Brownian Gibbs property. Although this result looks natural, we were unable to find its proof in the literature, and so we provide it. 

\begin{lemma}\label{AvoidIsGibbs} Assume the same notation as in Definition \ref{DefAvoidingLaw}. If $\mathcal{Q}$ is a $\llbracket 1, k \rrbracket$-indexed line ensemble on $[a,b]$ with probability distribution $\mathbb{P}_{avoid}^{a,b,\vec{x},\vec{y},\infty,-\infty}$ then it satisfies the Brownian Gibbs property of Definition \ref{DefBGP}.
\end{lemma}

The following two lemmas provide couplings of two line ensembles of non-intersecting Brownian bridges on the same interval, which depend monotonically on their boundary data. Schematic depictions of the couplings are provided in Figure \ref{fig:MCL}.
\vspace{-3mm}
\begin{figure}[ht]
\begin{center}
  \includegraphics[scale = 0.8]{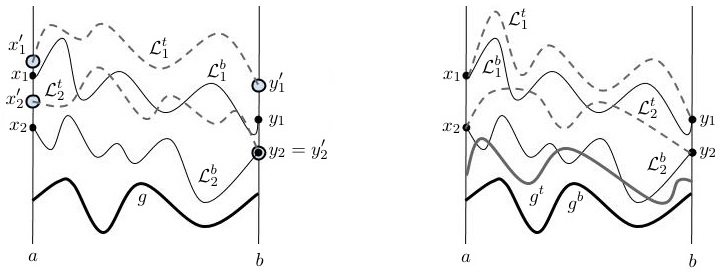}
  \vspace{-2mm}
  \caption{Two diagrammatic depictions of the monotone coupling Lemma \ref{MCLxy} (left part) and Lemma \ref{MCLfg} (right part).}
  \label{fig:MCL}
  \end{center}
\end{figure}

\begin{lemma}\label{MCLxy} Assume the same notation as in Definition \ref{DefAvoidingLaw}. Fix $k \in \mathbb{N}$, $a < b$ and a continuous function $g: [a,b] \rightarrow \mathbb{R} \cup \{- \infty \}$ and assume that $\vec{x}, \vec{y}, \vec{x}', \vec{y}' \in \weyl_k$. We assume that $g(a) < x_k$, $g(b) < y_k$ and $x_i \leq x_i'$, $y_i \leq y_i'$ for $i = 1,\dots, k$. Then there exists a probability space $(\Omega, \mathcal{F}, \mathbb{P})$, which supports two $\llbracket 1, k \rrbracket$-indexed line ensembles $\mathcal{L}^t$ and $\mathcal{L}^b$ on $[a,b]$ such that the law of $\mathcal{L}^{t}$ {\big (}resp. $\mathcal{L}^b${\big )} under $\mathbb{P}$ is given by $\mathbb{P}_{avoid}^{a,b, \vec{x}', \vec{y}', \infty, g}$ {\big (}resp. $\mathbb{P}_{avoid}^{a,b, \vec{x}, \vec{y}, \infty, g}${\big )} and such that $\mathbb{P}$-almost surely we have $\mathcal{L}_i^t(r) \geq \mathcal{L}^b_i(r)$ for all $i = 1,\dots, k$ and $r \in [a,b]$.
\end{lemma}

\begin{lemma}\label{MCLfg} Assume the same notation as in Definition \ref{DefAvoidingLaw}. Fix $k \in \mathbb{N}$, $a < b$ and two continuous functions $g^t, g^b: [a,b] \rightarrow \mathbb{R} \cup \{- \infty \}$ and assume that $\vec{x}, \vec{y} \in \weyl_k$. We assume that $g^t(r) \geq g^b(r)$ for all $r \in [a,b]$ and $g^t(a) < x_k$, $g^t(b) < y_k$. Then there exists a probability space $(\Omega, \mathcal{F}, \mathbb{P})$, which supports two $\llbracket 1, k \rrbracket$-indexed line ensembles $\mathcal{L}^t$ and $\mathcal{L}^b$ on $[a,b]$ such that the law of $\mathcal{L}^{t}$ {\big (}resp. $\mathcal{L}^b${\big )} under $\mathbb{P}$ is given by $\mathbb{P}_{avoid}^{a,b, \vec{x}, \vec{y}, \infty, g^t}$ {\big (}resp. $\mathbb{P}_{avoid}^{a,b, \vec{x}, \vec{y}, \infty, g^b}${\big )} and such that $\mathbb{P}$-almost surely we have $\mathcal{L}_i^t(r) \geq \mathcal{L}^b_i(r)$ for all $i = 1,\dots, k$ and $r \in [a,b]$.
\end{lemma}

In plain words, Lemma \ref{MCLxy} states that one can couple two line ensembles $\mathcal{L}^{t}$ and $\mathcal{L}^{b}$ of non-intersecting Brownian bridges, bounded from below by the same function $g$, in such a way that if all boundary values of $\mathcal{L}^{t}$ are above the respective boundary values of $\mathcal{L}^{b}$, then all curves of $\mathcal{L}^{t}$ are almost surely above the respective curves of $\mathcal{L}^{b}$. See the left part of Figure \ref{fig:MCL}. Lemma \ref{MCLfg}, states that one can couple two line ensembles $\mathcal{L}^{t}$ and $\mathcal{L}^{b}$ that have the same boundary values, but the lower bound $g^t$ of $\mathcal{L}^{t}$ is above the lower bound $g^b$ of $\mathcal{L}^{b}$, in such a way that all curves of $\mathcal{L}^{t}$ are almost surely above the respective curves of $\mathcal{L}^{b}$. See the right part of Figure \ref{fig:MCL}.

Lemmas \ref{MCLxy} and \ref{MCLfg} can be found in \cite[Section 2]{CorHamA}. The key idea behind their proof is to approximate the Brownian bridges by random walk bridges, for which constructing the monotone couplings is easier, and perform a limit transition. Since the details surrounding that limit transition are only briefly mentioned in \cite{CorHamA}, and since these lemmas are central results that will be used throughout Sections~\ref{Section3} and \ref{Section4}, we included their proofs in Section \ref{Section5}.

%% file: Section3.tex
\section{Preliminaries on Brownian Gibbsian line ensembles}\label{Section3}
In this section we summarize several results about Brownian Gibbsian line ensembles, which will be used in the arguments later in the text. We remark that while some of the proofs in this section are a bit technical, the statements of the various results are fairly intuitive. Consequently, readers can safely skip most of the proofs in this section without this affecting their understanding of the main argument in Section \ref{Section4} and only come back to them if interested.

%

\subsection{Properties of line ensembles}
 In this section we prove a few results about general line ensembles, which basically state that the laws of line ensembles are characterized by their finite-dimensional distributions.

 We continue with the same notation as in Section \ref{Section2.1}. In particular, we fix $\Lambda \subset \mathbb{R}$ to be an interval and $\Sigma = \llbracket 1, N\rrbracket$ with $N \in \mathbb{N}$ or $N = \infty$. Given $a,b \in \Lambda$ with $a < b$ and $k \in \Sigma$, we define $\pi_{[a,b]}^{\llbracket 1, k\rrbracket} : C (\Sigma \times \Lambda) \rightarrow C(\llbracket 1, k \rrbracket \times [a,b])$ through
\begin{equation}\label{ProjBox}
\pi_{[a,b]}^{\llbracket 1, k\rrbracket} (f)(i,x) = f(i,x) \mbox{ for $i = 1, \dots, k$ and $x \in [a,b]$.}
\end{equation}
In addition, given $n_1, \dots, n_k \in \Sigma$ and $t_1, t_2, \dots, t_k \in \Lambda$ we define $\pi_{t_1, \dots, t_k}^{n_1, \dots, n_k}  : C (\Sigma \times \Lambda) \rightarrow \mathbb{R}^k$ through
\begin{equation}\label{ProjTuple}
\pi_{t_1, \dots, t_k}^{n_1, \dots, n_k} (f)= (f(n_1, t_1), \dots, f(n_k,t_k)).
\end{equation}
Observe that since $\Sigma \times \Lambda$ is locally compact we know that the functions in (\ref{ProjBox}) and (\ref{ProjTuple}) are continuous, cf. \cite[Lemma 46.4]{Munkres}.
\begin{lemma}\label{FinDim}
Suppose that $\mathcal{A}$ is a collection of measurable subsets of $ C (\Sigma \times \Lambda)$, such that for each $k \in \mathbb{N}$, $n_1, \dots, n_k \in \Sigma$, $t_1, t_2, \dots, t_k \in \Lambda$ and $x_1, \dots, x_k \in \mathbb{R}$ we know that 
$$\left [\pi_{t_1, \dots, t_k}^{n_1, \dots, n_k}\right]^{-1} \left( (-\infty ,x_1] \times (-\infty, x_2] \times \cdots \times (-\infty, x_k] \right) \in \mathcal{A}.$$
Then the $\sigma$-algebra generated by $\mathcal{A}$, denoted by $\sigma(\mathcal{A})$, equals $\mathcal{C}_{\Sigma}$. In particular, the collection of finite-dimensional sets of $C (\Sigma \times \Lambda)$ is a {\em separating class}, cf. \cite[p. 9]{Billing}.
\end{lemma}
\begin{proof}
Since $\mathcal{A} \subset \mathcal{C}_{\Sigma}$ we know that $\sigma(\mathcal{A}) \subset \mathcal{C}_{\Sigma}$. In the remainder of the proof we show that $\mathcal{C}_\Sigma \subset \sigma (\mathcal{A})$. 

Since sets of the form $(-\infty, x_1] \times \cdots \times (-\infty, x_k]$ generate the Borel $\sigma$-algebra on $\mathbb{R}^k$, cf. \cite[Example~1.1, p.~9]{Billing}, we know that $\sigma(\mathcal{A})$ contains $\left [\pi_{t_1, \dots, t_k}^{n_1, \dots, n_k}\right]^{-1} (B)$ for any Borel set in $\mathbb{R}^k$. In particular, by \cite[Example 1.3, p. 11]{Billing} we conclude that 
\begin{equation}\label{CompactContainment}
\left[\pi_{[a,b]}^{\llbracket 1, k\rrbracket} \right]^{-1} (A) \in \sigma(\mathcal{A}),
\end{equation}
for any Borel set $A \subset C(\llbracket 1, k \rrbracket \times [a,b])$. If $\Sigma$ and $\Lambda$ are both compact this proves the lemma.\\

Suppose that $\Sigma$ or $\Lambda$ (or both) are not compact. Let $\llbracket 1, k_n \rrbracket \times [a_n, b_n]$ be a compact exhaustion of $\Sigma \times \Lambda$ and define $\pi_n:  C (\Sigma \times \Lambda)  \rightarrow C(\llbracket 1, k_n \rrbracket \times [a_n,b_n])$ through
$$\pi_n(f) = \pi_{[a_n,b_n]}^{\llbracket 1, k_n\rrbracket} (f),$$
where the latter function was defined in (\ref{ProjBox}). We also define for $m \geq n$ the functions $\pi_{m,n}: C(\llbracket 1, k_m \rrbracket \times [a_m,b_m]) \rightarrow C(\llbracket 1, k_n \rrbracket \times [a_n,b_n])$ through
$$\pi_{m,n}(f)(i,x) =  f(i,x) \mbox{ for $i = 1, \dots n$ and $x \in [a_n,b_n]$.} $$
The latter functions are also continuous by the local compactness of $\llbracket 1, m \rrbracket \times [a_m, b_m]$. 

We consider the metric $d_n$ on the space $C(\llbracket 1, k_n \rrbracket \times [a_n,b_n])$, given by
$$d_n(f,g) = \min \left(1, \sum_{i = 1}^{k_n} \sup_{x \in[a_n, b_n]} | f(i, x) - g(i,x)| \right),$$
and observe that the metric space topology induced by $d_n$ is the same as that of the topology of uniform convergence. We further define a metric on $C(\Sigma \times \Lambda)$ through
\begin{equation*}
d(f,g) = \sum_{n = 1}^\infty 2^{-n} \cdot d_n( \pi_n(f), \pi_n(g)),
\end{equation*}
and observe that the metric space topology induced by $d$ on $C(\Sigma \times \Lambda)$ is the same as the topology of uniform convergence over compacts. Moreover, $(C(\Sigma \times \Lambda), d)$ is easily seen to be a separable metric space, using that $C([a,b])$ with the uniform topology is separable, see e.g. \cite[Example 1.3, p. 11]{Billing}.

Let $f \in C(\Sigma \times \Lambda)$ and $\epsilon \geq 0$ be given. For $M \geq 1$ we define 
$$A_M = \left\{g \in C(\Sigma \times \Lambda):  \sum_{n = 1}^M 2^{-n} \cdot d_n( \pi_n(f), \pi_n(g)) \leq \epsilon \right\}.$$
Then we observe that 
$$A_M = \pi_{M}^{-1} \left( \left\{ h \in C(\llbracket 1, k_M \rrbracket \times [a_M,b_M]):  \sum_{n = 1}^M 2^{-n} \cdot d_n( \pi_n(f), \pi_{M,n}(h)) \leq \epsilon \right\} \right).$$
The continuity of $\pi_{M,n}$ and the functions $d_n( \pi_n(f), \cdot)$ on $C(\llbracket 1, k_M \rrbracket \times [a_M,b_M])$ and $C(\llbracket 1, k_n \rrbracket \times [a_n,b_n])$ respectively, imply that the set in the brackets above is closed in $C(\llbracket 1, k_M \rrbracket \times [a_M,b_M])$ and so $A_M \in \sigma(\mathcal{A})$ by (\ref{CompactContainment}). On the other hand, we see that 
$$\{ g \in C(\Sigma \times \Lambda) : d(f,g) \leq \epsilon \} = \bigcap_{M \geq 1} A_M,$$
and so closed balls in $C(\Sigma \times \Lambda)$ belong to $\sigma(\mathcal{A})$. This means that open balls also lie in $\sigma(\mathcal{A})$ and by the separability of the space we conclude that all open sets in $C(\Sigma \times \Lambda)$ belong to $\sigma(\mathcal{A})$. This implies that $\mathcal{C}_{\Sigma} \subset \sigma(\mathcal{A})$ and completes the proof.
\end{proof}

We next require the following elementary result from analysis.
\begin{lemma}\label{MonEq} Let $F_i: \mathbb{R} \rightarrow \mathbb{R}$ for $i = 1,2$ be increasing, right-continuous functions. Let $E_i$ denote the set of points in $\mathbb{R}$, where $F_i$ is continuous for $i = 1,2$ and suppose that $F_1(x) = F_2(x)$ for all $x\in E_1 \cap E_2$. Then $F_1(x) = F_2(x)$ for all $x \in \mathbb{R}$.
\end{lemma}
\begin{proof}
Put $S = E_1^c\cup E_2^c.$  From \cite[Theorem 4.30]{Rudin} we know that $S$ is an at most countable subset of $\mathbb{R}$. For any $x\in \mathbb{R}$ we can find a sequence $y_k \in S^c$ such that $y_k > x$ for all $k \in \mathbb{N}$ and $y_k \rightarrow x$ as $k \rightarrow \infty$. By the right continuity of $F_i$ at $x$ we conclude that 
\begin{equation*}
F_1(x) = \lim_{k \rightarrow \infty} F_1(y_k) = \lim_{k \rightarrow \infty} F_2(y_k) = F_2(x). \hfill\qedhere
\end{equation*}
\end{proof}

\begin{proposition}\label{PropFD}  
Let $\Sigma$ and $\Lambda$ be as in Theorem~\ref{ThmMain}. Suppose that $\mathcal{L}^1$ and $\mathcal{L}^2$ are $\Sigma$-indexed line ensembles on $\Lambda$ with laws $\mathbb{P}_1$ and $\mathbb{P}_2$ respectively. Suppose further that for every $k\in \mathbb{N}$,  $ t_1 < t_2 < \cdots < t_k$ with $t_i \in \Lambda^o$ (the interior of $\Lambda$) for $i = 1, \dots, k$; $n_1, \dots, n_k \in \Sigma$ and $x_1, \dots, x_k \in \mathbb{R}$ we have
\begin{equation}\label{FDStrict}
\mathbb{P}_1 \left( \mathcal{L}^1_{n_1}(t_1) \leq x_1, \dots,\mathcal{L}^1_{n_k}(t_k) \leq x_k  \right) =\mathbb{P}_2 \left( \mathcal{L}^2_{n_1}(t_1) \leq x_1, \dots,\mathcal{L}^2_{n_k}(t_k) \leq x_k  \right).
\end{equation}
Then we have that $\mathbb{P}_1 = \mathbb{P}_2$.
\end{proposition}
\begin{proof} For clarity we split the proof in three steps.

{\bf \raggedleft Step 1.} Let $M \in \Sigma$. In addition, suppose that $k_1, \dots, k_M \in \mathbb{N}$ be given. Let $D = \{ (i,j) \in \mathbb{Z}^2: j = 1, \dots, M \mbox{ and } i = 1, \dots, k_j \}$. Finally, fix $y_i^j \in \mathbb{R}$ and $t_i^j \in \Lambda$ with $t_1^j < t_2^j< \cdots < t_{k_j}^j$ for $(i,j) \in D$. We claim that 
\begin{equation}\label{FDE1}
\mathbb{P}_1 \left( \mathcal{L}^1_j(t_i^j) \leq y_{i}^j \mbox{ for $(i,j) \in D$}\right) = \mathbb{P}_2 \left( \mathcal{L}^2_j(t_i^j) \leq y_{i}^j \mbox{ for $(i,j) \in D$}\right).
\end{equation}
We prove (\ref{FDE1}) in the steps below. Here we assume its validity and finish the proof of the proposition.

Let $\mathcal{B}$ denote the collection of sets $A \in \mathcal{C}_\Sigma$ such that 
\begin{equation*}
\mathbb{P}_1 \left( \mathcal{L}^1 \in A \right) = \mathbb{P}_2 \left( \mathcal{L}^2 \in A \right).
\end{equation*}
By the monotone convergence theorem, we know that $\mathcal{B}$ is a $\lambda$-system. Further, by (\ref{FDE1}) we know that $\mathcal{B}$ contains the $\pi$-system of sets of the form
$$\left [\pi_{t_1, \dots, t_k}^{n_1, \dots, n_k}\right]^{-1} \left( (-\infty ,x_1] \times (-\infty, x_2] \times \cdots \times (-\infty, x_k] \right),$$
where we used the notation from (\ref{ProjTuple}). By the $\pi-\lambda$ Theorem, see \cite[Theorem 2.1.6]{Durrett}, we see that $\mathcal{B}$ contains the $\sigma$-algebra generated by the above sets, which by Lemma \ref{FinDim} is precisely $\mathcal{C}_\Sigma$. Consequently, $\mathcal{B} = \mathcal{C}_\Sigma$, which proves the proposition.\\

{\bf \raggedleft Step 2.} Let $x_i^j \in \mathbb{R}$ for $(i,j) \in D$ be given. We claim that there exists a sequence $\{ p_w\}_{w = 1}^\infty$, $p_w \in [0,1]$ such that for $v \in \{1,2\}$ we have
\begin{equation}\label{FDE2}
\begin{split}
&\mathbb{P}_v \left( \mathcal{L}^v_{j}(t^j_i) < x^j_i \mbox{ for $(i,j) \in D$}  \right) \leq \liminf_{w \rightarrow \infty} p_w \leq \limsup_{w \rightarrow \infty} p_w \leq \mathbb{P}_v \left(\mathcal{L}^v_j(t^j_i) \leq x^j_i \mbox{ for $(i,j) \in D$}  \right).
\end{split}
\end{equation}
We will prove (\ref{FDE2}) in the next step. For now we assume its validity and finish the proof of (\ref{FDE1}). 

For $r \in \mathbb{R}$ and $v \in \{1, 2\}$ we let
$$G_v(r)=\mathbb{P}_v \left( \mathcal{L}^v_{j}(t^j_i) \leq y^j_i + r \mbox{ for $(i,j) \in D$}  \right).$$
Observe that by basic properties of probability measures we know that $G_1$ and $G_2$ are increasing right-continuous functions. Moreover, if $G_1$ and $G_2$ are both continuous at a point $r$ then from (\ref{FDE2}) applied to $x^j_i = y^j_i + r$ for $(i,j) \in D$ we know that $G_1(r) = G_2(r)$. The latter and Lemma \ref{MonEq} imply that $G_1 = G_2$. In particular, $G_1(0) = G_2(0)$, which is precisely (\ref{FDE1}). \\

{\bf \raggedleft Step 3.} In this final step we prove (\ref{FDE2}). Let $t_i^j(w)$ for $(i,j) \in D$ be a sequence such that:
\begin{enumerate}
\item for each $w \in \mathbb{N}$ we have $t_{i_1}^{j_1}(w) \neq t_{i_2}^{j_2}(w)$ whenever $(i_1, j_1) \neq (i_2, j_2)$;
\item for each $w \in \mathbb{N}$ and $(i,j) \in D$ we have $t_i^j(w) \in \Lambda^o$ (the interior of $\Lambda$);
\item for each $(i,j) \in D$ we have $\lim_{w \rightarrow \infty} t^j_i(w) = t^j_i$.
\end{enumerate}
Then by (\ref{FDStrict}) we have for each $w \in \mathbb{N}$ that
$$ \mathbb{P}_1 \left( \mathcal{L}^1_{j}(t^j_i(w)) \leq x^j_i \mbox{ for $(i,j) \in D$}  \right) = \mathbb{P}_2 \left( \mathcal{L}^2_{j}(t^j_i(w)) \leq x^j_i \mbox{ for $(i,j) \in D$}  \right).$$
We let $p_w$ denote the above probability.

For $v \in\{1,2\}$ we denote
$$A_v = \{ \omega: \mathcal{L}^v_{j}(t^j_i) < x^j_i  \mbox{ for $(i,j) \in D$} \} \mbox{ and } B_v =  \{ \omega: \mathcal{L}^v_{j}(t^j_i) \leq x^j_i \mbox{ for $(i,j) \in D$} \} .$$
By the almost sure continuity of $\mathcal{L}^v$ we know that $\mathbb{P}_v$-almost surely
\begin{align*}
\lim_{w \to \infty} {\bf 1}\left\{ \mathcal{L}^v_{j}(t^j_i(w)) \leq x^j_i \mbox{ for $(i,j) \in D$} \right\} \cdot {\bf 1}_{A_v} &= {\bf 1}_{A_v},\\
\lim_{w \to \infty} {\bf 1}\left\{ \mathcal{L}^v_{j}(t^j_i(w)) \leq x^j_i \mbox{ for $(i,j) \in D$} \right\} \cdot {\bf 1}_{B^c_v} &= 0.
\end{align*}
The second line above and the bounded convergence theorem imply that 
$$\limsup_{w \rightarrow \infty} p_w \leq \limsup_{w \rightarrow \infty} \mathbb{E}_v \left[ { \bf 1}\left\{ \mathcal{L}^v_{j}(t^j_i(w)) \leq x^j_i \mbox{ for $(i,j) \in D$} \right\} \cdot {\bf 1}_{B^c_v}  + {\bf 1}_{B_v} \right] = \mathbb{P}_v \left(B_v\right) .$$
On the other hand, the top line and the bounded convergence theorem imply that 
$$\liminf_{w \rightarrow \infty} p_w \geq \liminf_{w \rightarrow \infty} \mathbb{E}_v \left[ { \bf 1}\left\{ \mathcal{L}^v_{j}(t^j_i(w)) \leq x^j_i \mbox{ for $(i,j) \in D$} \right\} \cdot {\bf 1}_{A_v}\right] = \mathbb{P}_v \left(A_v\right) .$$
The last two statements imply (\ref{FDE2}), which concludes the proof of the proposition.
\end{proof}

%

\subsection{Properties of avoiding Brownian line ensembles}
 In this section we prove several results about the line-ensembles from Definition \ref{DefAvoidingLaw}.

Fix $x,y,a,b \in \mathbb{R}$ with $a < b$ and let $B(t)$ denote the Brownian bridge from $B(a) = x$ to $B(b) = y$ with diffusion parameter $1$, see (\ref{BBDef}). Then by \cite[Eq.~6.28, p. 359]{KS} we know that the random vector $(B(t_1), \dots, B(t_n)) \in \mathbb{R}^n$ with $a \leq t_1 < t_2 < \cdots < t_n \leq b$ has the following density function
\begin{equation}\label{BBDensity}
f_{BB}(x_1, \dots, x_n) = \prod_{i = 1}^n p(t_i - t_{i-1}; x_{i-1}, x_i) \cdot \frac{p(b - t_n; x_n,y)}{p(b-a;x,y)}, \mbox{ where } p(t; x,y) = \frac{e^{-(x-y)^2/2t}}{\sqrt{2\pi t}},
\end{equation}
where $x_0 = x$ and we interpret $p(0; x,y)dy = \delta_{x}(y)$ as the delta function at $x$ (these expressions can occur if $t_1 = a$ or $t_n = b$ in (\ref{BBDensity})).

The following lemma explains why equation (\ref{BGPTower}) makes sense, see also Remark \ref{RemMeas}.
\begin{lemma}\label{LemmaMeasExp} Assume the same notation as in Definition \ref{DefAvoidingLaw} and suppose that $F : C( \llbracket 1, k \rrbracket \times [a,b]) \rightarrow \mathbb{R}$ is a bounded Borel-measurable function. Let 
\begin{equation*}
\begin{split}
S_{t,b} &= \{ (\vec{x}, \vec{y}, f,g) \in \weyl_k \times  \weyl_k \times C([a,b]) \times C([a,b]): f(t) > g(t) \mbox{ for $t \in [a,b]$, }\\
&\hspace{5cm} \mbox{  $f(a) > x_1$, $f(b) > y_1$, $g(a) < x_k$, $g(b)< y_k$}\}, \\
S_{t} &= \{ (\vec{x}, \vec{y}, f) \in \weyl_k \times  \weyl_k \times C([a,b]): \mbox{  $f(a) > x_1$, $f(b) > y_1$}\}, \\
S_{b} &= \{ (\vec{x}, \vec{y}, g) \in \weyl_k \times  \weyl_k \times C([a,b]): \mbox{  $g(a) < x_k$, $g(b) < y_k$}\},
\end{split}
\end{equation*}
and $S =\weyl_k \times  \weyl_k$, where each of the above sets is endowed with the subspace topology coming from the product topology and corresponding Borel $\sigma$-algebra. Then the functions $G_F: S \rightarrow \mathbb{R}$, $G_F^{t} : S_t \rightarrow \mathbb{R}$, $G_F^b: S_b \rightarrow \mathbb{R}$ and $G^{t,b}_F: S_{t,b} \rightarrow \mathbb{R}$ given by
\begin{equation}\label{MeasExpFun}
\begin{split}
G^{t,b}_F(\vec{x}, \vec{y}, f,g) &= \mathbb{E}_{avoid}^{a,b,\vec{x}, \vec{y},f,g}[F(\mathcal{Q})], \hspace{7mm} G^{t}_F(\vec{x}, \vec{y}, f) = \mathbb{E}_{avoid}^{a,b,\vec{x}, \vec{y},f,-\infty}[F(\mathcal{Q})], \\
G^{b}_F(\vec{x}, \vec{y},g) &= \mathbb{E}_{avoid}^{a,b,\vec{x}, \vec{y},\infty,g}[F(\mathcal{Q})], \hspace{7mm} G_F(\vec{x}, \vec{y}) = \mathbb{E}_{avoid}^{a,b,\vec{x}, \vec{y},\infty,-\infty}[F(\mathcal{Q})]
\end{split}
\end{equation}
are all measurable.
\end{lemma}
\begin{proof} For clarity we split the proof into four steps. \\

{\bf \raggedleft Step 1.} We prove that $G^{t,b}_F$ is measurable in the steps below. In this step we assume that $ G^{t,b}_F$ is measurable and deduce that all the other functions in the statement of the lemma are measurable. 

Let $N_0$ be sufficiently large so that $N_0 > \max(x_1, y_1)$ and $-N_0 < \min (x_k,y_k)$. We also denote by $f_N : [a,b] \rightarrow \mathbb{R}$ the functions such that $f_N(x) = N$ and set $g_N = - f_N$. From Definition \ref{DefAvoidingLaw} we know that for $N \geq N_0$ we have that 
$$G^{t,b}_F(\vec{x}, \vec{y}, f_N, g) = \frac{\mathbb{E}_{avoid}^{a,b,\vec{x}, \vec{y},\infty ,g}[F(\mathcal{Q}) \cdot {\bf 1}\{\mathcal{Q}(x) < N \} ]}{\mathbb{E}_{avoid}^{a,b,\vec{x}, \vec{y},\infty ,g}[ {\bf 1}\{\mathcal{Q}(x) < N \} ]}.$$
Since $G^{t,b}_F$ is measurable we know that the above functions are measurable on $S^b$ for all $N \geq N_0$. By the bounded convergence theorem the above functions converge to $G^b_F$ and so the latter is also measurable on $S^b$. Analogous arguments applied to the functions $G^{t,b}_F(\vec{x}, \vec{y}, f, g_N)$ and $G^{t,b}_F(\vec{x}, \vec{y}, f_N, g_N)$ show that $G^t_F$ and $G_F$ are measurable as well.\\

{\bf \raggedleft Step 2.} Here we show $G^{t,b}_F$ is measurable. Fix $K \in \mathbb{N}$ and $n_1, \dots,n_K \in \llbracket 1, k \rrbracket$, $t_1, \dots, t_K \in [a,b]$ and $z_1, \dots, z_K \in \mathbb{R}$. We define with this data the function $H: C( \llbracket 1, k \rrbracket \times [a,b]) \rightarrow \mathbb{R}$ through
$$H(h) = \prod_{i = 1}^K {\bf 1} \{ h(n_i, t_i) \leq z_i \}.$$ 
We claim that the function
\begin{equation}\label{PiSystemMeas}
G^{s,t}_H(\vec{x}, \vec{y}, f,g) = \mathbb{E}_{avoid}^{a,b,\vec{x}, \vec{y},f,g}[H(\mathcal{Q})]
\end{equation}
is measurable. We establish the latter statement in the steps below. For now we assume its validity and conclude the proof of the lemma.

Let $\mathcal{H}$ denote the set of bounded Borel-measurable functions $F$ for which $G^{t,b}_F$ as in (\ref{MeasExpFun}) is measurable. It is clear that $\mathcal{H}$ is closed under linear combinations (by linearity of the expectation). Furthermore, if $F_n \in \mathcal{H}$ is an increasing sequence of non-negative measurable functions that increase to a bounded function $F$ then $F \in \mathcal{H}$ by the monotone convergence theorem. Finally, in view of (\ref{PiSystemMeas}) we know that ${\bf 1}_{A} \in \mathcal{H}$ for any set $A \in \mathcal{A}$, where $\mathcal{A}$ is the $\pi$-system of sets of the form
$$ \{ h \in C( \llbracket 1, k \rrbracket \times [a,b]) : h(n_i, t_i) \leq z_i \mbox{ for $i = 1, \dots, K$}\}.$$
By the monotone class theorem, see e.g. \cite[Theorem 5.2.2]{Durrett}, we have that $\mathcal{H}$ contains all bounded measurable functions with respect to $\sigma(\mathcal{A})$, and the latter is $\mathcal{C}_{\llbracket 1, k \rrbracket}$ in view of Lemma \ref{FinDim}. This proves the measurability of $G^{t,b}_F$ in (\ref{MeasExpFun}) for any bounded measurable $F$.\\

{\bf \raggedleft Step 3.} Let $\mathcal{B}$ be the $\llbracket 1, k \rrbracket$-indexed line ensemble on $[a,b]$ with distribution $\mathbb{P}^{a,b, \vec{x},\vec{y}}_{free}$ (the law of $k$ independent Brownian bridges $\{B_i: [a,b] \rightarrow \mathbb{R} \}_{i = 1}^k$ from $B_i(a) = x_i$ to $B_i(b) = y_i$ with diffusion parameter $1$, where we have rewritten $\mathcal{B}(i, \cdot) = B_i(\cdot)$). Let $E$ be the event 
$$E = \left\{ f(r) > B_1(r) > B_2(r) > \cdots > B_k(r) > g(r) \mbox{ for all $r \in[a,b]$} \right\}.$$ 
From Definition \ref{DefAvoidingLaw} we know that 
\begin{equation*}
G^{t,b}_H(\vec{x}, \vec{y}, f,g) = \frac{\mathbb{P}^{a,b, \vec{x},\vec{y}}_{free} \left( \{ B_{n_i}(t_i) \leq z_i   \mbox{ for $i = 1, \dots, K$}\} \cap E \right) }{\mathbb{P}^{a,b, \vec{x},\vec{y}}_{free}(E)},
\end{equation*}
from which we conclude that it suffices to show that 
\begin{equation}\label{setmeas}
\mathbb{P}^{a,b, \vec{x},\vec{y}}_{free} \left( \{ B_{n_i}(t_i) \leq z_i   \mbox{ for $i = 1, \dots, K$}\} \cap E \right)
\end{equation}
is a measurable function. Indeed, if we can establish the above then taking $z_i \rightarrow \infty$ for $i = 1, \dots, K$ would imply that $\mathbb{P}^{a,b, \vec{x},\vec{y}}_{free}(E)$ is positive and measurable and then $G_H(\vec{x}, \vec{y}, f,g) $ is measurable as the ratio of two measurable functions with a non-vanishing denominator. In the remainder we focus on proving that (\ref{setmeas}) is measurable.

Let $N_0$ be sufficiently large so that $3N_0^{-1} < \min_{i = 0, \dots, k} [x_{i} - x_{i+1} ]$ and $3N_0^{-1} < \min_{i = 0, \dots, k} [y_{i} - y_{i+1} ]$, where $\vec{x}=  (x_1, \dots, x_k)$, $\vec{y} = (y_1, \dots, y_k)$ and we used the convention $x_0 = f(a)$, $x_{k+1} = g(a)$, $y_0 = f(b)$ and $y_{k+1} = g(b)$. Then for $w \geq N_0$ we define
\begin{align*}
E_w &= \{ f(r) - w^{-1} \geq  B_1(r) + w^{-1} > B_1(r) - w^{-1} \geq B_2(r) + w^{-1} >B_2(r) -  w^{-1} \geq \cdots \\
&\hspace{3cm} \geq B_k(r) + w^{-1} > B_k(r) - w^{-1}  \geq g(r) + w^{-1} \mbox{ for all $r \in[a,b]$} \}.
\end{align*}
Notice that by the monotone convergence theorem we have that 
$$\mathbb{P}^{a,b, \vec{x},\vec{y}}_{free} \left( \{ B_{n_i}(t_i) \leq z_i   \mbox{ for $i = 1, \dots, K$}\} \cap E \right) = \lim_{w \rightarrow \infty} \mathbb{P}^{a,b, \vec{x},\vec{y}}_{free} \left( \{ B_{n_i}(t_i) \leq z_i   \mbox{ for $i = 1, \dots, K$}\} \cap E_w \right),$$
and so it suffices to prove that $ \mathbb{P}^{a,b, \vec{x},\vec{y}}_{free} \left( \{ B_{n_i}(t_i) \leq z_i   \mbox{ for $i = 1, \dots, K$}\} \cap E_w \right)$ are measurable functions for $w \geq N_0$. Let $\{q_n: n \in \mathbb{N} \}$ be an enumeration of the rationals in $(a,b) \setminus \{t_1, \dots, t_K\}$. Using the almost sure continuity of Brownian bridges we see that 
 $$ \mathbb{P}^{a,b, \vec{x},\vec{y}}_{free} \left( \{ B_{n_i}(t_i) \leq z_i   \mbox{ for $i = 1, \dots, K$}\} \cap E_w \right) = \lim_{N \rightarrow \infty}  \mathbb{P}^{a,b, \vec{x},\vec{y}}_{free} \left( \{ B_{n_i}(t_i) \leq z_i   \mbox{ for $i = 1, \dots, K$}\} \cap E^N_w \right) ,$$
where 
\begin{align*}
E^N_w &=  \{ f(r) - w^{-1} \geq  B_1(r) + w^{-1} > B_1(r) - w^{-1} \geq B_2(r) + w^{-1} >B_2(r) -  w^{-1} \geq \cdots\\
&\qquad\qquad \geq B_k(r) + w^{-1} > B_k(r) - w^{-1}  \geq g(r) + w^{-1} \mbox{ when $r = q_n$ with $1 \leq n \leq N$} \}.
\end{align*}
Combining the last few statements we see that we have reduced the proof that (\ref{setmeas}) is measurable to showing that
\begin{equation*}
\mathbb{P}^{a,b, \vec{x},\vec{y}}_{free} \left( \{ B_{n_i}(t_i) \leq z_i   \mbox{ for $i = 1, \dots, K$}\} \cap E_w^N \right)
\end{equation*}
is measurable for all $w \geq N_0$ and $N \in \mathbb{N}$. We prove this in the next step.\\

{\bf \raggedleft Step 4.} Let $S_N = \{t_1, \dots, t_K\} \cup \{ q_n: 1 \leq n \leq N\}$ and let $a \leq s_1 < s_2 < \cdots < s_{N+K} \leq b$ be an ordering of the elements of $S$ in increasing order. In view of (\ref{BBDensity}) we know that 
\begin{equation}\label{setmeasV3}
\begin{split}
&\mathbb{P}^{a,b, \vec{x},\vec{y}}_{free} \left( \{ B_{n_i}(t_i) \leq z_i   \mbox{ for $i = 1, \dots, K$}\} \cap E_w^N \right) = \int_{\mathbb{R}} \cdots \int_{\mathbb{R}} \prod_{j = 1}^k \frac{p(b - s_{N+K}; x^j_{N+K},y_j)}{p(b-a;x_j,y_j)} \\
& \hspace{5cm} \times \prod_{i = 1}^{N+K}  H_i(x_i^1, \dots, x_i^k) \cdot \prod_{i = 1}^{N+K}  \prod_{j = 1}^k p(s_i - s_{i-1}; x^j_{i-1}, x^j_i) dx_i^j,
\end{split}
\end{equation}
where $x_0^j = x_j$ for $j = 1, \dots, k$ and the functions $H_i$ are given by
$$ H_i(x_i^1, \dots, x_i^k) = \begin{cases}  {\bf 1}_{F_{w, q_n}}& \mbox{ if $s_i = q_n$ for $n = 1, \dots, N,$} \\  {\bf 1}\{ x_{i}^{n_u} \leq z_u\}  &\mbox{ if $s_i = t_u$ for $u = 1, \dots, K$,}\end{cases}$$
with 
\begin{align*}
F_{w,q}(x_1, \dots, x_k) &=  \{ f(q) - w^{-1} \geq x_1 + w^{-1} > x_1- w^{-1} \geq x_2 + w^{-1} >x_2 -  w^{-1} \geq \cdots\\
&\hspace{5cm} \geq x_k + w^{-1} > x_k - w^{-1}  \geq g(q) + w^{-1}  \}.
\end{align*}
From equation (\ref{setmeasV3}) we see that $\mathbb{P}^{a,b, \vec{x},\vec{y}}_{free} \left( \{ B_{n_i}(t_i) \leq z_i   \mbox{ for $i = 1, \dots, K$}\} \cap E_w^N \right) $ is the integral of a non-negative measurable function and is thus itself measurable, cf. \cite[Chapter 2, Theorem 3.2]{Stein3}.
\end{proof}

The following lemma explains how the law of $\mathbb{P}_{avoid}^{a,b,\vec{x},\vec{y},\infty,-\infty}$ from Definition \ref{DefAvoidingLaw} behaves under affine transformations.
\begin{lemma}\label{LemmaAffine} Assume the same notation as in Definition \ref{DefAvoidingLaw} and suppose that $\mathcal{Q}$ is a $\llbracket 1, k \rrbracket$-indexed line ensemble on $[a,b]$ with probability distribution $\mathbb{P}_{avoid}^{a,b,\vec{x},\vec{y},\infty,-\infty}$. Suppose that $r, u \in \mathbb{R}$ and $c > 0$ are given. With this data we define the random $\llbracket 1, k \rrbracket$-indexed line ensemble $\tilde{\mathcal{Q}}$ on $[a',b'] = [c^2 a+u, c^2b+u]$ through
\begin{equation*}
\tilde{\mathcal{Q}}(i,x) = c \cdot \mathcal{Q}(i, c^{-2} x-u) + r \mbox{ for $i = 1, \dots, k$ and $x \in [a',b']$.}
\end{equation*}
Then the law of $\tilde{\mathcal{Q}}$ is $\mathbb{P}_{avoid}^{a',b',\vec{x}',\vec{y}',\infty,-\infty}$, where $x'_i = x_i \cdot c + r$ and $y'_i = y_i \cdot c + r$ for $i = 1, \dots, k$.
\end{lemma}
\begin{proof} We split the proof into two steps. In the first we show that if we perform the affine transformations in the statement of the lemma to the line-ensemble of independent Brownian bridges, then we have a similar result with $\mathbb{P}_{avoid}$, replaced with $\mathbb{P}_{free}$. In the second part we prove the lemma for the laws $\mathbb{P}_{avoid}$ by appealing to the definition of these laws through $\mathbb{P}_{free}$ as in Definition \ref{DefAvoidingLaw}.\\

{\bf \raggedleft Step 1.} Let $\mathcal{B}$ be the $\llbracket 1, k \rrbracket$-indexed line ensemble on $[a,b]$ with distribution $\mathbb{P}^{a,b, \vec{x},\vec{y}}_{free}$ (the law of $k$ independent Brownian bridges $\{B_i: [a,b] \rightarrow \mathbb{R} \}_{i = 1}^k$ from $B_i(a) = x_i$ to $B_i(b) = y_i$ with diffusion parameter $1$, where we have rewritten $\mathcal{B}(i, \cdot) = B_i(\cdot)$). In addition, let ${\mathcal{B}'}$ be the $\llbracket 1, k \rrbracket$-indexed line ensemble on $[a',b']$ with distribution $\mathbb{P}^{a',b', \vec{x}',\vec{y}'}_{free}$ (the law of $k$ independent Brownian bridges $\{{B}'_i: [a',b'] \rightarrow \mathbb{R} \}_{i = 1}^k$ from ${B}'_i(a) = x'_i$ to ${B}'_i(b) = y'_i$ with diffusion parameter $1$, where we have rewritten ${\mathcal{B}'}(i, \cdot) = {B}'_i(\cdot)$). Finally, we define the $\llbracket 1, k \rrbracket$-indexed line ensemble $\tilde{\mathcal{B}}$ on $[a,b]$ through
\begin{equation*}
\tilde{\mathcal{B}}(i,x) = c \cdot \mathcal{B}(i, c^{-2} x-u) + r \mbox{ for $i = 1, \dots, k$ and $x \in [a',b']$.}
\end{equation*}  

We first claim that the law of $\tilde{\mathcal{B}}$ under $\mathbb{P}^{a,b, \vec{x},\vec{y}}_{free}$ is the same as that of $\mathcal{B}'$ under $\mathbb{P}^{a',b', \vec{x}',\vec{y}'}_{free}$. To see the latter, fix $K \in \mathbb{N}$ and $n_1, \dots,n_K \in \llbracket 1, k \rrbracket$, $t_1, \dots, t_K \in [a',b']$ and $z_1, \dots, z_K \in \mathbb{R}$. We then have from (\ref{BBDensity}) that 
\begin{equation}\label{BBCDF1}
\begin{split}
&\mathbb{P}^{a,b, \vec{x},\vec{y}}_{free} \left( \tilde{{B}}_{n_i}(t_i) \leq z_i \mbox{ for $i \in \llbracket 1, K \rrbracket $ } \hspace{-1mm} \right) = \mathbb{P}^{a,b, \vec{x},\vec{y}}_{free} \left( {B}_{n_i}(c^{-2}t_i -u) \leq \frac{z_i - r}{c} \mbox{ for $i \in \llbracket 1, K \rrbracket $ } \hspace{-1mm} \right)\\ 
&= \int_{\mathbb{R}} \cdots \int_{\mathbb{R}} \prod_{j = 1}^k \frac{p(b - [c^{-2}t_{K} - u]; \tilde{x}^j_{K},{y}_j)}{p(b-a;x_j,y_j)}  \prod_{i = 1}^{K}  H_i(\tilde{x}^{n_i}_i) \cdot \prod_{i = 1}^{K}  \prod_{j = 1}^k p(c^{-2}[t_i - t_{i-1}]; \tilde{x}^j_{i-1}, \tilde{x}^j_i) d\tilde{x}_i^j,\\
\end{split}
\end{equation}
where $\tilde{x}_0^j = x_j$ for $j = 1,\dots, k$ and $H_i(x) = {\bf 1} \{ x \leq c^{-1}[z_i - r] \}$ for $i = 1, \dots, K$. On the other hand,
\begin{equation}\label{BBCDF2}
\begin{split}
&\mathbb{P}^{a',b', \vec{x}',\vec{y}'}_{free} \left( {B}'_{n_i}(t_i) \leq z_i \mbox{ for $i \in \llbracket 1, K \rrbracket $ } \hspace{-1mm} \right) \\ 
&= \int_{\mathbb{R}} \cdots \int_{\mathbb{R}} \prod_{j = 1}^k \frac{p( b' - t_{K}; x^j_{K},y'_j)}{p(b'-a';x'_j,y'_j)}  \prod_{i = 1}^{K}  H'_i(x^{n_i}_i) \cdot \prod_{i = 1}^{K}  \prod_{j = 1}^k p(t_i - t_{i-1}; x^j_{i-1}, x^j_i) dx_i^j,\\
\end{split}
\end{equation}
where $x_0^j = x'_j$ for $j = 1,\dots, k$ and $H'_i(x) = {\bf 1} \{ x \leq z_i  \}$ for $i = 1, \dots, K$. Upon performing the change of variables $\tilde{x}_i^j = \frac{x_i^j - r}{c}$ in (\ref{BBCDF1}) and using the scaling property of the heat kernel we obtain precisely the expression in the second line of (\ref{BBCDF2}). Consequently, $\tilde{\mathcal{B}}$ under $\mathbb{P}^{a,b, \vec{x},\vec{y}}_{free}$ and $\mathcal{B}'$ under $\mathbb{P}^{a',b', \vec{x}',\vec{y}'}_{free}$ have the same finite-dimensional distributions. By Proposition \ref{PropFD} we conclude that the laws of these line ensembles are the same.\\

{\bf \raggedleft Step 2.} Continuing with the notation from Step 1, we define
\begin{align*}
E &= \left\{  B_1(r) > B_2(r) > \cdots > B_k(r)  \mbox{ for all $r \in[a,b]$} \right\},\\ 
\tilde{E} &= \left\{  \tilde{B}_1(r) > \tilde{B}_2(r) > \cdots >\tilde{B}_k(r)  \mbox{ for all $r \in[a',b']$} \right\},\\
E' &= \left\{  B'_1(r) > B'_2(r) > \cdots > B'_k(r)  \mbox{ for all $r \in[a',b']$} \right\}.
\end{align*}
We also let $\mathcal{Q}'$ be a $\llbracket 1, k \rrbracket$-indexed line ensemble on $[a,b]$ with law $\mathbb{P}_{avoid}^{a',b',\vec{x}',\vec{y}',\infty,-\infty} $.

If we fix $K \in \mathbb{N}$ and $n_1, \dots,n_K \in \llbracket 1, k \rrbracket$, $t_1, \dots, t_K \in [a',b']$ and $z_1, \dots, z_K \in \mathbb{R}$ we have
\begin{equation*}
\begin{split}
&\mathbb{P}_{avoid}^{a,b,\vec{x},\vec{y},\infty,-\infty} \left( \tilde{\mathcal{Q}}_{n_i}(t_i) \leq z_i \mbox{ for $i \in \llbracket 1, K \rrbracket $ } \hspace{-2mm} \right)= \mathbb{P}_{avoid}^{a,b,\vec{x},\vec{y},\infty,-\infty} \left( \mathcal{Q}_{n_i}(c^{-2}t_i -u) \leq \frac{z_i - r}{c} \mbox{ for $i \in \llbracket 1, K \rrbracket $ }  \hspace{-2mm} \right)\\
& =  \frac{\mathbb{P}^{a,b, \vec{x},\vec{y}}_{free} \left( \{ B_{n_i}(c^{-2}t_i -u) \leq c^{-1} \cdot [z_i - r]   \mbox{ for $i\in \llbracket 1, K \rrbracket $}\} \cap E \right) }{\mathbb{P}^{a,b, \vec{x},\vec{y}}_{free}(E)}  \\
& = \frac{\mathbb{P}^{a,b, \vec{x},\vec{y}}_{free} \bigl( \{ \tilde{B}_{n_i}(t_i) \leq z_i  \mbox{ for $i\in \llbracket 1, K \rrbracket $}\} \cap \tilde{E} \bigr) }{\mathbb{P}^{a,b, \vec{x},\vec{y}}_{free}(\tilde{E})} \\
& = \frac{\mathbb{P}^{a',b', \vec{x}',\vec{y}'}_{free} \left( \{ B'_{n_i}(t_i) \leq z_i   \mbox{ for $i\in \llbracket 1, K \rrbracket $}\} \cap E' \right) }{\mathbb{P}^{a',b', \vec{x}',\vec{y}'}_{free}(E')} = \mathbb{P}_{avoid}^{a',b',\vec{x}',\vec{y}',\infty,-\infty} \left( {\mathcal{Q}}'_{n_i}(t_i) \leq z_i \mbox{ for $i \in \llbracket 1, K \rrbracket $ } \hspace{-2mm} \right),
\end{split}
\end{equation*}
where the first equality follows from the definition of $\tilde{Q}$; the second and last equality follow from Definition \ref{DefAvoidingLaw}; the third one follows from the definition of $\tilde{B}$ and the fourth one follows from the equality of laws for $\tilde{\mathcal{B}}$ and $\mathcal{B}'$ established in Step 1. The above equation shows that the finite dimensional distributions of $\tilde{\mathcal{Q}}$ under $\mathbb{P}_{avoid}^{a,b,\vec{x},\vec{y},\infty,-\infty}$ agree with those of $\mathcal{Q}'$ under $\mathbb{P}_{avoid}^{a',b',\vec{x}',\vec{y}',\infty,-\infty}$, which by Proposition \ref{PropFD} implies that the laws of these line ensembles are the same.
\end{proof}

\begin{lemma}\label{LemmaFlip} Assume the same notation as in Definition \ref{DefAvoidingLaw} and suppose that $\mathcal{Q}$ is a $\llbracket 1, k \rrbracket$-indexed line ensemble on $[a,b]$ with probability distribution $\mathbb{P}_{avoid}^{a,b,\vec{x},\vec{y},\infty,-\infty}$.  Let $\tilde{\mathcal{Q}}$ be the random $\llbracket 1, k \rrbracket$-indexed line ensemble on $[a,b]$, defined through
\begin{equation*}
\tilde{\mathcal{Q}}(i,x) = -\mathcal{Q}(k- i + 1,x)  \mbox{ for $i = 1, \dots, k$ and $x \in [a,b]$.}
\end{equation*}
Then the law of $\tilde{\mathcal{Q}}$ under $\mathbb{P}_{avoid}^{a,b,\vec{x},\vec{y},\infty,-\infty}$ is  $\mathbb{P}_{avoid}^{a,b,-\vec{x},-\vec{y},\infty,-\infty}$.
\end{lemma}
\begin{proof}
Similarly, to the proof of Lemma \ref{LemmaAffine} we split the proof into two steps. In the first we show that if we perform the reflections in the statement of the lemma to the line-ensemble of independent Brownian bridges, then we have a similar result with $\mathbb{P}_{avoid}$, replaced with $\mathbb{P}_{free}$. In the second part we prove the lemma for the laws $\mathbb{P}_{avoid}$ by appealing to the definition of these laws through $\mathbb{P}_{free}$ as in Definition \ref{DefAvoidingLaw}.\\

{\bf \raggedleft Step 1.} Let $\mathcal{B}$ be the $\llbracket 1, k \rrbracket$-indexed line ensemble on $[a,b]$ with distribution $\mathbb{P}^{a,b, \vec{x},\vec{y}}_{free}$ (the law of $k$ independent Brownian bridges $\{B_i: [a,b] \rightarrow \mathbb{R} \}_{i = 1}^k$ from $B_i(a) = x_i$ to $B_i(b) = y_i$ with diffusion parameter $1$, where we have rewritten $\mathcal{B}(i, \cdot) = B_i(\cdot)$). In addition, let ${\mathcal{B}'}$ be the $\llbracket 1, k \rrbracket$-indexed line ensemble on $[a,b]$ with distribution $\mathbb{P}^{a,b, -\vec{x},-\vec{y}}_{free}$ (the law of $k$ independent Brownian bridges $\{{B}'_i: [a,b] \rightarrow \mathbb{R} \}_{i = 1}^k$ from ${B}'_i(a) = -x_{k-i+1}$ to ${B}'_i(b) = -y_{k-i+1}$ with diffusion parameter $1$, where we have rewritten ${\mathcal{B}'}(i, \cdot) = {B}'_i(\cdot)$). Finally, we define the $\llbracket 1, k \rrbracket$-indexed line ensemble $\tilde{\mathcal{B}}$ on $[a,b]$ through
\begin{equation*}
\tilde{\mathcal{B}}(i,x) =  -\mathcal{B}(k-i + 1,x) \mbox{ for $i = 1, \dots, k$ and $x \in [a,b]$.}
\end{equation*}

We first claim that the law of $\tilde{\mathcal{B}}$ under $\mathbb{P}^{a,b, \vec{x},\vec{y}}_{free}$ is the same as that of $\mathcal{B}'$ under $\mathbb{P}^{a,b, -\vec{x}, -\vec{y}}_{free}$. To see the latter, fix $K \in \mathbb{N}$ and $n_1, \dots,n_K \in \llbracket 1, k \rrbracket$, $t_1, \dots, t_K \in [a,b]$ and $z_1, \dots, z_K \in \mathbb{R}$. We then have from (\ref{BBDensity}) that 
\begin{equation}\label{FlipBBCDF1}
\begin{split}
&\mathbb{P}^{a,b, \vec{x},\vec{y}}_{free} \left( \tilde{{B}}_{n_i}(t_i) \leq z_i \mbox{ for $i \in \llbracket 1, K \rrbracket $ } \hspace{-1mm} \right) = \mathbb{P}^{a,b, \vec{x},\vec{y}}_{free} \left( -{B}_{k - n_i + 1}(t_i ) \leq z_i \mbox{ for $i \in \llbracket 1, K \rrbracket $ } \hspace{-1mm} \right) \\ 
&= \int_{\mathbb{R}} \cdots \int_{\mathbb{R}} \prod_{j = 1}^k \frac{p(b - t_{K}; \tilde{x}^j_{K},{y}_j)}{p(b-a;x_j,y_j)}  \prod_{i = 1}^{K}  H_i(\tilde{x}^{k - n_i+1}_i) \cdot \prod_{i = 1}^{K}  \prod_{j = 1}^k p(t_i - t_{i-1}; \tilde{x}^j_{i-1}, \tilde{x}^j_i) d\tilde{x}_i^j,\\
\end{split}
\end{equation}
where $\tilde{x}_0^j = x_j$ for $j = 1,\dots, k$ and $H_i(x) = {\bf 1} \{ - x \leq z_i\}$ for $i = 1, \dots, K$. On the other hand, 
\begin{equation}\label{FlipBBCDF2}
\begin{split}
&\mathbb{P}^{a,b, -\vec{x}', -\vec{y}}_{free} \left( {B}'_{n_i}(t_i) \leq z_i \mbox{ for $i \in \llbracket 1, K \rrbracket $ } \hspace{-1mm} \right) \\
&= \int_{\mathbb{R}} \cdots \int_{\mathbb{R}} \prod_{j = 1}^k \frac{p( b - t_{K}; x^j_{K},-y_{k-j+1})}{p(b-a;-x_{k-j+1},-y_{k-j+1})}  \prod_{i = 1}^{K}  H'_i(x^{n_i}_i) \cdot \prod_{i = 1}^{K}  \prod_{j = 1}^k p(t_i - t_{i-1}; x^j_{i-1}, x^j_i) dx_i^j,\\
\end{split}
\end{equation}
where $x_0^j = -x_{k -j+1}$ for $j = 1,\dots, k$ and $H'_i(x) = {\bf 1} \{ x \leq z_i  \}$ for $i = 1, \dots, K$. Upon performing the change of variables $\tilde{x}_i^j = -x_i^{k-j +1}$ in (\ref{FlipBBCDF1}) and using symmetry of the heat kernel we obtain precisely the expression in the second line of (\ref{FlipBBCDF2}). Consequently, $\tilde{\mathcal{B}}$ under $\mathbb{P}^{a,b, \vec{x},\vec{y}}_{free}$ and $\mathcal{B}'$ under $\mathbb{P}^{a,b, -\vec{x},-\vec{y}}_{free}$ have the same finite-dimensional distributions. By Proposition \ref{PropFD} we conclude that the laws of these line ensembles are the same.\\

{\bf \raggedleft Step 2.} Continuing with the notation from Step 1, we define
\begin{align*}
E &= \left\{  B_1(r) > B_2(r) > \cdots > B_k(r)  \mbox{ for all $r \in[a,b]$} \right\},\\
\tilde{E} &= \left\{  \tilde{B}_1(r) > \tilde{B}_2(r) > \cdots >\tilde{B}_k(r)  \mbox{ for all $r \in[a',b']$} \right\},\\
E' &= \left\{  B'_1(r) > B'_2(r) > \cdots > B'_k(r)  \mbox{ for all $r \in[a',b']$} \right\}.
\end{align*}
We also let $\mathcal{Q}'$ be a $\llbracket 1, k \rrbracket$-indexed line ensemble on $[a,b]$ with law $\mathbb{P}_{avoid}^{a,b,-\vec{x}, -\vec{y},\infty,-\infty} $.

If we fix $K \in \mathbb{N}$ and $n_1, \dots,n_K \in \llbracket 1, k \rrbracket$, $t_1, \dots, t_K \in [a',b']$ and $z_1, \dots, z_K \in \mathbb{R}$ we have
\begin{equation*}
\begin{split}
&\mathbb{P}_{avoid}^{a,b,\vec{x},\vec{y},\infty,-\infty} \left( \tilde{\mathcal{Q}}_{n_i}(t_i) \leq z_i \mbox{ for $i \in \llbracket 1, K \rrbracket $ } \hspace{-2mm} \right)= \mathbb{P}_{avoid}^{a,b,\vec{x},\vec{y},\infty,-\infty} \left( -\mathcal{Q}_{k - n_i + 1}(t_i) \leq z_i \mbox{ for $i \in \llbracket 1, K \rrbracket $ }  \hspace{-1mm} \right) \\ 
&= \frac{\mathbb{P}^{a,b, \vec{x},\vec{y}}_{free} \left( \{- B_{k - n_i + 1}(t_i ) \leq z_i   \mbox{ for $i\in \llbracket 1, K \rrbracket $}\} \cap E \right) }{\mathbb{P}^{a,b, \vec{x},\vec{y}}_{free}(E)} = \frac{\mathbb{P}^{a,b, \vec{x},\vec{y}}_{free} \left( \{ \tilde{B}_{n_i}(t_i) \leq z_i  \mbox{ for $i\in \llbracket 1, K \rrbracket $}\} \cap \tilde{E} \right) }{\mathbb{P}^{a,b, \vec{x},\vec{y}}_{free}(\tilde{E})}\\
 &= \frac{\mathbb{P}^{a,b, -\vec{x},-\vec{y}}_{free} \left( \{ B'_{n_i}(t_i) \leq z_i   \mbox{ for $i\in \llbracket 1, K \rrbracket $}\} \cap E' \right) }{\mathbb{P}^{a,b, -\vec{x},-\vec{y}}_{free}(E')} = \mathbb{P}_{avoid}^{a,b,-\vec{x},-\vec{y},\infty,-\infty} \left( {\mathcal{Q}}'_{n_i}(t_i) \leq z_i \mbox{ for $i \in \llbracket 1, K \rrbracket $ } \hspace{-2mm} \right),
\end{split}
\end{equation*}
where the first equality follows from the definition of $\tilde{Q}$; the second and last equality follow from Definition \ref{DefAvoidingLaw}; the third one follows from the definition of $\tilde{B}$ and the fourth one follows from the equality of laws for $\tilde{\mathcal{B}}$ and $\mathcal{B}'$ established in Step 1. The above equation shows that the finite dimensional distributions of $\tilde{\mathcal{Q}}$ under $\mathbb{P}_{avoid}^{a,b,\vec{x},\vec{y},\infty,-\infty}$ agree with those of $\mathcal{Q}'$ under $\mathbb{P}_{avoid}^{a,b,-\vec{x},-\vec{y},\infty,-\infty}$, which by Proposition \ref{PropFD} implies that the laws of these line ensembles are the same.
\end{proof}

%
%
\subsection{Auxiliary results}
 In this section we summarize some auxiliary results, which will be useful in the proof of Theorem \ref{ThmMain}.

\begin{lemma}\label{LNoAtoms}  Let $\Sigma = \llbracket 1, N \rrbracket$ with $N \in \mathbb{N}$, $N \geq 2$ and $\Lambda = [a,b] \subset \mathbb{R}$. Suppose $\mathcal{L}$ is a $\Sigma$-indexed line ensembles on $\Lambda$ that satisfies the partial Brownian Gibbs property of Definition \ref{DefPBGP}. Fix $t \in (a,b)$, $n \in \llbracket 1, N-1 \rrbracket $ and $s \in \mathbb{R}$. Then 
\begin{equation*}
\mathbb{P} \left( \mathcal{L}_{n}(t) = s  \right) = 0.
\end{equation*}
\end{lemma}
\begin{proof}
Fix $\vec{x}, \vec{y} \in \weyl_{N-1}$ and Let $g :[a,b] \rightarrow \mathbb{R}$ be a continuous function such that $g(a) < x_{N-1}$ and $g(b) < y_{N-1}$. From Definition \ref{DefAvoidingLaw} we know that $\mathbb{P}_{avoid}^{a,b, \vec{x}, \vec{y}, \infty ,g}$ is absolutely continuous with respect to $\mathbb{P}_{free}^{a,b,\vec{x},\vec{y}}$. Since Brownian bridges have no atoms we conclude that 
$$\mathbb{E}_{avoid}^{a,b,\vec{x},\vec{y},\infty, g} \left[{\bf 1} \{ \mathcal{Q}_{n}(t) = s \}\right] = 0.$$
Consequently, by the partial Brownian Gibbs property and the tower property for conditional expectations we deduce that
\begin{equation*}
\begin{split}
\mathbb{P} \left( \mathcal{L}_{n}(t) = s  \right) = \mathbb{E} \bigl[ \mathbb{E} \left[{\bf 1} \{ \mathcal{L}_{n}(t) = s \} \vert \mathcal{F}_{ext}(K \times (a,b)) \right]\bigr] = \mathbb{E} \bigl[   \mathbb{E}_{avoid}^{a,b, \vec{x},\vec{y}, \infty, g} \left[{\bf 1} \{ \mathcal{Q}_{n}(t) = s \}\right] \bigr] = \mathbb{E}[0] = 0,
\end{split}
\end{equation*}
where $K = \llbracket 1, N-1\rrbracket$, $\vec{x} = (\mathcal{L}_{1}(a), \dots, \mathcal{L}_{N-1}(a))$, $\vec{y} = ( \mathcal{L}_{1}(b), \dots,  \mathcal{L}_{N-1}(b))$, and $g = \mathcal{L}_{N}[a,b]$. 
\end{proof}

Let $\Phi(x)$ be the cumulative distribution function of a standard normal random variable, and $\phi(x)$ denote its density.
The following result can be found in \cite[Section 4.2]{MZ}
\begin{lemma}\label{LemmaI1}
There is a constant $c_0 > 1$ such that for all $x \geq 0$ we have
\begin{equation}\label{LI2}
 \frac{1}{c_0(1+x)} \leq \frac{1 - \Phi(x)}{\phi(x)} \leq \frac{c_0}{1 +x},
\end{equation}
\end{lemma}

The following result can be found in \cite[Chapter 4, Eq.~3.40]{KS}. 
\begin{lemma}\label{LemmaBBmax} Let $a \in \mathbb{R}$, $T > 0$ and $\beta > 0$. Let $B: [0,T] \rightarrow \mathbb{R}$ denote a Brownian bridge from $B(0) = 0$ to $B(T) = a$ with diffusion parameter $1$. Then we have
\begin{equation*}
\mathbb{P}^{0,T,0,a}_{free}\left( \max_{0 \leq t \leq T} B(t) \geq \beta \right) = \mathbb{P}^{0,T,0,-a}_{free}\left( \min_{0 \leq t \leq T} B(t) \leq -\beta \right)  = e^{-2\beta (\beta -a)/T}.
\end{equation*}
\end{lemma}

\begin{lemma}\label{LemmaBotMax} Assume the same notation as in Definition \ref{DefAvoidingLaw} and suppose that $\mathcal{Q}$ is a $\llbracket 1, k \rrbracket$-indexed line ensemble on $[a,b]$ with probability distribution $\mathbb{P}_{avoid}^{a,b,\vec{x},\vec{y},\infty,-\infty}$. Then we have for $r \geq 0$
\begin{equation}\label{BottomCurveLB}
\mathbb{P}_{avoid}^{a,b,\vec{x},\vec{y},\infty,-\infty} \left(  \mathcal{Q}_k( (a+b)/2 ) \geq \max(x_k, y_k) +  (b-a)^{1/2}r \right) \leq \frac{c_0 e^{-2r^2}}{\sqrt{2\pi} (1 + 2r)},
\end{equation}
where $c_0$ is as in Lemma \ref{LemmaI1}.
\end{lemma}
\begin{proof}Let $A$ denote the left side of (\ref{BottomCurveLB}). Define $\vec{z} \in \weyl_k$ through $z_i = \max(x_i, y_i) $. By Lemma~\ref{MCLxy} we have that
\begin{equation}\label{UBW1}
\begin{split}
A &\leq \mathbb{P}_{avoid}^{a,b,\vec{z},\vec{z},\infty,-\infty} \left(   \mathcal{Q}_k( (a+b)/2 ) \geq z_k +  (b-a)^{1/2}r \right) \\
&= \mathbb{P}_{avoid}^{a,b,-\vec{z},-\vec{z},\infty,-\infty} \left(   \mathcal{Q}_1( (a+b)/2 ) \leq -z_k -  (b-a)^{1/2}r \right),
\end{split}
\end{equation}
where the equality follows from Lemma \ref{LemmaFlip}. By Lemma \ref{AvoidIsGibbs} we know that $\llbracket 1, k \rrbracket$-indexed line ensembles distributed according to $\mathbb{P}_{avoid}^{a,b,-\vec{z},-\vec{z},\infty,-\infty}$ satisfy the Brownian Gibbs property and so by Definition \ref{DefBGP} we have
\begin{equation}\label{UBW2}
\begin{split}
&\mathbb{P}_{avoid}^{a,b,-\vec{z},-\vec{z},\infty,-\infty} \left(   \mathcal{Q}_1( (a+b)/2 ) \leq -z_k -  (b-a)^{1/2}r \right) \\
&= \mathbb{E}_{avoid}^{a,b,-\vec{z},-\vec{z},\infty,-\infty} \left[  \mathbb{E} \left[  {\bf 1} \{\mathcal{Q}_1( (a+b)/2 ) \leq -z_k - (b-a)^{1/2}r \} {\Big \vert} \mathcal{F}_{ext} (\{1\} \times (a,b))  \right]  \right] \\
&= \mathbb{E}_{avoid}^{a,b,-\vec{z},-\vec{z},\infty,-\infty} \left[  \mathbb{E}_{avoid}^{a,b, -z_k, -z_k, \infty, \mathcal{Q}_2[a,b]} \left[  {\bf 1} \{\mathcal{Q}_1( (a+b)/2 ) \leq -z_k - (b-a)^{1/2}r \} \right]  \right] \\
&\leq \mathbb{E}_{avoid}^{a,b,-\vec{z},-\vec{z},\infty,-\infty} \left[  \mathbb{E}_{avoid}^{a,b, -z_k, -z_k, \infty, -\infty} \left[  {\bf 1} \{B( (a+b)/2 ) \leq -z_k - (b-a)^{1/2}r \} \right]  \right] \\
&= \mathbb{E}_{avoid}^{a,b,-\vec{z},-\vec{z},\infty,-\infty} \left[  \Phi(-2r)  \right]=  \Phi(-2r) = 1 - \Phi(2r),
\end{split}
\end{equation}
where in going from the third to the fourth line we use Lemma \ref{MCLfg} and in going from the fourth to the fifth line we used that under $\mathbb{P}_{avoid}^{a,b, -z_k, -z_k, \infty, -\infty} $ the curve $B$ is precisely a Brownian bridge from $B(a) = -z_k$ to $B(b) = -z_k$ with diffusion parameter $1$. The latter and (\ref{BBDensity}) imply that $B((a+b)/2)$ is distributed like a Gaussian random variable with mean $0$ and variance $(b-a)/4$, which implies the formulas above. Combining (\ref{UBW1}), (\ref{UBW2}) and (\ref{LI2}) we conclude (\ref{BottomCurveLB}).
\end{proof}

\begin{lemma}\label{LemmaBotMin} Assume the same notation as in Definition \ref{DefAvoidingLaw} and suppose that $\mathcal{Q}$ is a $\llbracket 1, k \rrbracket$-indexed line ensemble on $[a,b]$ with probability distribution $\mathbb{P}_{avoid}^{a,b,\vec{x},\vec{y},\infty,-\infty}$. Then we have for $r \geq 0$
\begin{equation}\label{BottomCurveUB}
\mathbb{P}_{avoid}^{a,b,\vec{x},\vec{y},\infty,-\infty} \left(  \mathcal{Q}_k( (a+b)/2 ) \leq \max(x_k, y_k) -  (b-a)^{1/2}r \right) \geq \frac{e^{-2r^2}}{c_0 \sqrt{2\pi} (1 + 2r)},
\end{equation}
where $c_0$ is as in Lemma \ref{LemmaI1}.
\end{lemma}
\begin{proof}Let $A$ denote the left side of (\ref{BottomCurveUB}). Define $\vec{z} \in \weyl_k$ through $z_i = \max(x_i, y_i) $. By Lemma~\ref{MCLxy} we have that
\begin{equation}\label{LBW1}
\begin{split}
A &\geq \mathbb{P}_{avoid}^{a,b,\vec{z},\vec{z},\infty,-\infty} \left(   \mathcal{Q}_k( (a+b)/2 ) \leq z_k -  (b-a)^{1/2}r \right) \\
&= \mathbb{P}_{avoid}^{a,b,-\vec{z},-\vec{z},\infty,-\infty} \left(   \mathcal{Q}_1( (a+b)/2 ) \geq -z_k +  (b-a)^{1/2}r \right),
\end{split}
\end{equation}
where the equality follows from Lemma \ref{LemmaFlip}. By Lemma \ref{AvoidIsGibbs} we know that $\llbracket 1, k \rrbracket$-indexed line ensembles distributed according to $\mathbb{P}_{avoid}^{a,b,-\vec{z},-\vec{z},\infty,-\infty}$ satisfy the Brownian Gibbs property and so by Definition \ref{DefBGP} we have
\begin{equation}\label{LBW2}
\begin{split}
&\mathbb{P}_{avoid}^{a,b,-\vec{z},-\vec{z},\infty,-\infty} \left(   \mathcal{Q}_1( (a+b)/2 ) \geq -z_k +  (b-a)^{1/2}r \right) \\
&= \mathbb{E}_{avoid}^{a,b,-\vec{z},-\vec{z},\infty,-\infty} \left[  \mathbb{E} \left[  {\bf 1} \{\mathcal{Q}_1( (a+b)/2 ) \geq -z_k + (b-a)^{1/2}r \} {\Big \vert} \mathcal{F}_{ext} (\{1\} \times (a,b))  \right]  \right] \\
 &= \mathbb{E}_{avoid}^{a,b,-\vec{z},-\vec{z},\infty,-\infty} \left[  \mathbb{E}_{avoid}^{a,b, -z_k, -z_k, \infty, \mathcal{Q}_2[a,b]} \left[  {\bf 1} \{\mathcal{Q}_1( (a+b)/2 ) \geq -z_k + (b-a)^{1/2}r \} \right]  \right]\\
 &\geq \mathbb{E}_{avoid}^{a,b,-\vec{z},-\vec{z},\infty,-\infty} \left[  \mathbb{E}_{avoid}^{a,b, -z_k, -z_k, \infty, -\infty} \left[  {\bf 1} \{B( (a+b)/2 ) \geq -z_k + (b-a)^{1/2}r \} \right]  \right] \\
&= \mathbb{E}_{avoid}^{a,b,-\vec{z},-\vec{z},\infty,-\infty} \left[ 1-  \Phi(2r)  \right]= 1-  \Phi(2r),
\end{split}
\end{equation}
where in going from the thirf to the fourth line we use Lemma \ref{MCLfg} and in going from the fourth to the fifth line we used that under $\mathbb{P}_{avoid}^{a,b, -z_k, -z_k, \infty, -\infty} $ the curve $B$ is precisely a Brownian bridge from $B(a) = -z_k$ to $B(b) = -z_k$ with diffusion parameter $1$. The latter and (\ref{BBDensity}) imply that $B((a+b)/2)$ is distributed like a Gaussian random variable with mean $0$ and variance $(b-a)/4$, which implies the formulas above. Combining (\ref{LBW1}), (\ref{LBW2}) and (\ref{LI2}) we conclude (\ref{BottomCurveUB}).
\end{proof}

The following result can be found in \cite[Lemma 2.25]{Ham4}. We give a proof for the sake of completeness. 
\begin{lemma}\label{LemmaAvoidMax} Assume the same notation as in Definition \ref{DefAvoidingLaw} and suppose that $\mathcal{Q}$ is a $\llbracket 1, k \rrbracket$-indexed line ensemble on $[a,b]$ with probability distribution $\mathbb{P}_{avoid}^{a,b,\vec{x},\vec{y},\infty,-\infty}$. Then we have for $r \geq 0$ that
\begin{equation}\label{BottomCurveLB2}
\mathbb{P}_{avoid}^{a,b,\vec{x},\vec{y},\infty,-\infty} \left(  \inf_{x \in [a,b]} \mathcal{Q}_k(x) \leq \min(x_k, y_k) - \sqrt{2} (b-a)^{1/2} (k + r -1) \right) \leq (1 - 2e^{-1})^{-k} e^{-4r^2}.
\end{equation}
\end{lemma}
\begin{proof}
Let $A$ denote the left side of (\ref{BottomCurveLB2}). Define $\vec{z} \in \weyl_k$ through $z_i = \min(x_i, y_i) - \sqrt{2} (b-a)^{1/2} (i-1)$. By Lemma \ref{MCLxy} we have that
\begin{equation}\label{UBH1}
A \leq \mathbb{P}_{avoid}^{a,b,\vec{z},\vec{z},\infty,-\infty} \left(  \inf_{x \in [a,b]} \mathcal{Q}_k(x) \leq \min(x_k, y_k) - \sqrt{2} (b-a)^{1/2} (k + r -1) \right). 
\end{equation}
 Let $\mathcal{B}$ be the $\llbracket 1, k \rrbracket$-indexed line ensemble on $[a,b]$ with distribution $\mathbb{P}^{a,b, \vec{z},\vec{z}}_{free}$ (the law of $k$ independent Brownian bridges $\{B_i: [a,b] \rightarrow \mathbb{R} \}_{i = 1}^k$ from $B_i(a) = z_i$ to $B_i(b) = z_i$ with diffusion parameter $1$, where we have rewritten $\mathcal{B}(i, \cdot) = B_i(\cdot)$). Let
$$E = \left\{  B_1(r) > B_2(r) > \cdots > B_k(r)  \mbox{ for all $r \in[a,b]$} \right\}, $$ 
Then from (\ref{UBH1}) and Definition \ref{DefAvoidingLaw} we have that 
\begin{equation}\label{UBH2}
\begin{split}
A &\leq \frac{\mathbb{P}_{free}^{a,b,\vec{z},\vec{z}} \left( \inf_{x \in [a,b]} B_k(x) \leq \min(x_k, y_k) - \sqrt{2} (b-a)^{1/2} (k + r -1) \right)}{\mathbb{P}_{free}^{a,b,\vec{z},\vec{z}} \left(E \right)} \\
&= \frac{\mathbb{P}_{free}^{a,b,\vec{z},\vec{z}} \left( \inf_{x \in [a,b]} B_k(x) \leq z_k - \sqrt{2} (b-a)^{1/2}r \right)}{\mathbb{P}_{free}^{a,b,\vec{z},\vec{z}} \left(E \right)} = \frac{e^{-4r^2}}{\mathbb{P}_{free}^{a,b,\vec{z},\vec{z}} \left(E \right)},
\end{split}
\end{equation}
where in the first equality we used the definition of $z_k$, while in the second one we used Lemma \ref{LemmaBBmax} and the fact that $\tilde{B}(x) =B_k(x - a) - z_k$ has law $ \mathbb{P}^{0,b -a ,0,0}_{free}$ as follows from Step 1 in the proof of Lemma \ref{LemmaAffine}. Finally, we observe that
\begin{equation*}
\begin{split}
&\mathbb{P}_{free}^{a,b,\vec{z},\vec{z}} \left(E \right) \geq \mathbb{P}_{free}^{a,b,\vec{z},\vec{z}} \left( \sup_{x \in [a,b]} |B_i(x) - z_i| < [(b-a)/2]^{1/2} \mbox{ for $i = 1, \dots, k$}  \right) \\
&= \prod_{i = 1}^k \left[1 -  \mathbb{P}_{free}^{a,b,\vec{z},\vec{z}} \left( \sup_{x \in [a,b]} |B_i(x) - z_i| \geq[(b-a)/2]^{1/2}  \right) \right] \\
&\geq \prod_{i = 1}^k \left[1 -  \mathbb{P}_{free}^{a,b,\vec{z},\vec{z}} \left( \sup_{x \in [a,b]} B_i(x) - z_i \geq [(b-a)/2]^{1/2} \hspace{-1mm}  \right) \hspace{-1mm} -  \mathbb{P}_{free}^{a,b,\vec{z},\vec{z}} \left( \inf_{x \in [a,b]} B_i(x) - z_i \leq -[(b-a)/2]^{1/2}  \hspace{-1mm} \right) \hspace{-0.5mm}\right] \\
& = (1 - 2e^{-1})^k,
\end{split}
\end{equation*}
where in the last equality we used Lemma \ref{LemmaBBmax} and the fact that $\tilde{B}_i(x) =B_i(x - a) - z_i$ has law $ \mathbb{P}^{0,b -a ,0,0}_{free}$ as follows from Step 1 in the proof of Lemma \ref{LemmaAffine}. Combining the last inequality with (\ref{UBH2}) we arrive at (\ref{BottomCurveLB2}).
\end{proof}

%% file: Section4.tex
%
%
\section{Proof of Theorem \ref{ThmMain} }\label{Section4} The purpose of this section is to prove Theorem \ref{ThmMain}. We first state the main result of this section as Proposition \ref{PropMain} and deduce Theorem \ref{ThmMain} from it.  In Section \ref{Section4.1} we present the proof of a basic case of Proposition \ref{PropMain} to illustrate some of the key ideas and we give the full proof in Section \ref{Section4.2}.

\begin{proposition}\label{PropMain}  Let $\Sigma = \llbracket 1, N \rrbracket$ with $N \in \mathbb{N}$ and $\Lambda = [a,b] \subset \mathbb{R}$.  Suppose that $\mathcal{L}^1$ and $\mathcal{L}^2$ are $\Sigma$-indexed line ensembles on $\Lambda$ that satisfy the partial Brownian Gibbs property with laws $\mathbb{P}_1$ and $\mathbb{P}_2$ respectively. Suppose further that for every $k\in \mathbb{N}$,  $a = t_0 < t_1 < t_2 < \cdots < t_k < t_{k+1} = b$ and $x_1, \dots, x_k \in \mathbb{R}$ we have that 
\begin{equation}\label{FDEPM}
\mathbb{P}_1 \left( \mathcal{L}^1_1(t_1) \leq x_1, \dots,\mathcal{L}^1_1(t_k) \leq x_k  \right) =\mathbb{P}_2 \left( \mathcal{L}^2_1(t_1) \leq x_1, \dots,\mathcal{L}^2_1(t_k) \leq x_k  \right).
\end{equation}
Then for every $k\in \mathbb{N}$,  $a = t_0 < t_1 < t_2 < \cdots < t_k < t_{k+1} = b$, $n_1, \dots, n_k \in \llbracket 1, N \rrbracket$ and $x_1, \dots, x_k \in \mathbb{R}$ we have
\begin{equation}\label{FDEPM2}
\mathbb{P}_1 \left( \mathcal{L}^1_{n_1}(t_1) \leq x_1, \dots,\mathcal{L}^1_{n_k}(t_k) \leq x_k  \right) =\mathbb{P}_2 \left( \mathcal{L}^2_{n_1}(t_1) \leq x_1, \dots,\mathcal{L}^2_{n_k}(t_k) \leq x_k  \right).
\end{equation}
\end{proposition}
The proof of Proposition \ref{PropMain} is given in Section \ref{Section4.2} below. In the remainder of this section we assume its validity and prove Theorem \ref{ThmMain}
\begin{proof}[Proof of Theorem \ref{ThmMain}] We assume the same notation as in Theorem \ref{ThmMain}. Let $a, b \in \Lambda$ with $a < b$ and $K \in \Sigma$ be given. Let $\pi_{[a,b]}^{\llbracket 1, K \rrbracket}$ be as in (\ref{ProjBox}) and note that by Definition \ref{DefPBGP} we have that under $\mathbb{P}_v$ the $\llbracket 1, K \rrbracket$-indexed line ensembles $\pi_{[a,b]}^{\llbracket 1, K \rrbracket}(\mathcal{L}^v)$ on $[a,b]$ satisfies the partial Brownian Gibbs property, where $v \in \{1,2\}$. Here it is important that we work with the partial Brownian Gibbs property and not the usual Brownian Gibbs property, cf. Remark \ref{RPBGP}. Consequently, by Proposition \ref{PropMain} we conclude that for every $k\in \mathbb{N}$,  $a = t_0 < t_1 < t_2 < \cdots < t_k < t_{k+1} = b$, $n_1, \dots, n_k \in \llbracket 1, K \rrbracket$ and $x_1, \dots, x_k \in \mathbb{R}$ we have
\begin{equation*}
\mathbb{P}_1 \left( \mathcal{L}^1_{n_1}(t_1) \leq x_1, \dots,\mathcal{L}^1_{n_k}(t_k) \leq x_k  \right) =\mathbb{P}_2 \left( \mathcal{L}^2_{n_1}(t_1) \leq x_1, \dots,\mathcal{L}^2_{n_k}(t_k) \leq x_k  \right).
\end{equation*}
 Since $[a,b] \subset \Lambda$ and $K \in \Sigma$ were arbitrary we conclude that the latter equality holds for any $k\in \mathbb{N}$;  $ t_1 < t_2 < \cdots < t_k $, with $t_i \in \Lambda^o$ for $i =1 ,\dots, k$; $n_1, \dots, n_k \in \Sigma$ and $x_1, \dots, x_k \in \mathbb{R}$ and then from Proposition \ref{PropFD} we conclude that $\mathbb{P}_1 = \mathbb{P}_2$. 
\end{proof}

%
%
\subsection{Basic case of Proposition \ref{PropMain}}\label{Section4.1} In this section we work under the same assumptions as in Proposition \ref{PropMain} when $N = 2$ and prove (\ref{FDEPM2}) in the simplest non-trivial case when $k = 1$ and $n_1 = 2$. As we will see many of the key ideas that go into the proof of Proposition \ref{PropMain} are already present in this simple case. The goal is to illustrate the main arguments and explain the meaning and significance of different constructions, so that the reader is better equipped before proceeding with the general proof in the next section.

The special case above consists of proving that for $t_1 \in (a,b)$ and $y_1 \in \mathbb{R}$ we have
\begin{equation}\label{Special}
\mathbb{P}_1( \mathcal{L}^1_2(t_1) \leq y_1 ) = \mathbb{P}_2( \mathcal{L}^2_2(t_1) \leq y_1 ).
\end{equation}

Equation (\ref{FDEPM}) implies by virtue of Proposition \ref{PropFD} that $\mathcal{L}^1_1$ under $\mathbb{P}_1$ has the same law as $\mathcal{L}^2_1$ under $\mathbb{P}_2$ as $\{1\}$-indexed line ensembles on $[a,b]$ or equivalently, as random variables taking values in $(C([a,b]), \mathcal{C})$. In particular, if $H: C([a,b]) \rightarrow \mathbb{R}$ is any bounded measurable function we have
\begin{equation}\label{S41}
\mathbb{E} \left[ H(\mathcal{L}_1^1) \right] = \mathbb{E} \left[ H(\mathcal{L}_1^2) \right],
\end{equation}
where we will use $\mathbb{E}$ to denote the expectation with respect to $\mathbb{P}_1$ or $\mathbb{P}_2$. It will be clear which measure is meant by the expression inside of the expectation. 

The main idea of the argument is to construct a sequence of measurable functions $H_{w} : C([a,b]) \rightarrow \mathbb{R}$, $w\in \mathbb{N}$, for which the equality in (\ref{S41}) holds and such that the left (resp. right) side of (\ref{S41}) approximates the left (resp. right) side of (\ref{Special}) as $w \rightarrow \infty$. Specifically, we will construct sequences $H_{w}$ such that for a given $x_1 \in \mathbb{R}$ we have
\begin{equation}\label{S42}
\begin{split}
&p_w = \mathbb{E} \left[ H_w(\mathcal{L}_1^1) \right] = \mathbb{E} \left[ H_w(\mathcal{L}_1^2) \right] \mbox{ for $w \in \mathbb{N}$ and } \\
&\mathbb{P}_v( \mathcal{L}^v_2(t_1) < x_1 ) \leq \liminf_{w \rightarrow \infty} p_w \leq \limsup_{w \rightarrow \infty} p_w \leq  \mathbb{P}_v( \mathcal{L}^v_2(t_1) \leq x_1 ) \mbox{ where $v \in \{1,2\}$}.
\end{split}
\end{equation}
The second line in (\ref{S42}) is what we mean by ``approximate''. 

The hard part of the proof is finding functions $H_w$ that satisfy (\ref{S42}), but once we have them concluding (\ref{Special}) is easy. Indeed, if we set for $x_1 \in \mathbb{R}$ and $v \in \{1, 2\}$ 
$$G_v(x_1)=\mathbb{P}_v \left( \mathcal{L}^v_{2}(t_1) \leq x_1 \right),$$
then by basic properties of probability measures we know that $G_1$ and $G_2$ are increasing right-continuous functions. Moreover, if $G_1$ and $G_2$ are both continuous at a point $x_1$ then from (\ref{S42}) we know that $G_1(x_1) = G_2(x_1)$. The latter and Lemma \ref{MonEq} imply that $G_1 = G_2$. In particular, $G_1(y_1) = G_2(y_1)$, which is precisely (\ref{Special}).\\

In the remainder of the section we detail our choice of $H_w$ and show that it satisfies (\ref{S42}). Given $s,t,r,x,y \in \mathbb{R}$ with $s < t$ we define 
\begin{equation*}
F(r; s,t,x,y)= \mathbb{P}_{free}^{s,t,x,y} \left( B((s+t)/2) \leq r\right),
\end{equation*}
which is the probability that a Brownian bridge from $B(s) = x$ to $B(t) = y$ with diffusion parameter $1$ has its midpoint below $r$. We also let $a_w = t_1 - w^{-1}$ and $b_w = t_1 + w^{-1}$ for $w \geq W_0$, where $W_0$ is sufficiently large so that $a_w, b_w \in (a,b)$. Here it is important that $t_1 \in (a,b)$ and is not one of the end-points. With the latter data we define for $f \in C([a,b])$ the functions
\begin{equation*}
H_w(f) = \frac{{\bf 1} \{ f(t_1) \leq x_1 \}}{F(x_1; a_w, b_w, f(a_w), f(b_w))} \mbox{, where $w \geq W_0$}.
\end{equation*}
This is the choice of $H_w$ that satisfies (\ref{S42}). In order to see why this choice of functions is suitable for proving (\ref{S42}) we need to apply the partial Brownian Gibbs property, and a technical aspect of the latter, is that it requires that we work with bounded functions, and the $H_w$ are not bounded. Consequently, we define the sequence 
$$H^M_w(f) ={\bf 1} \{ f(t_1) \leq x_1 \} \cdot  \min\left(M, \frac{1}{F(x_1; a_w, b_w, f(a_w), f(b_w))}\right)$$
 for $M \in \mathbb{N}$. For each $M \in \mathbb{N}$ we have from (\ref{S41}) that 
$$\mathbb{E} \left[ H_w^M(\mathcal{L}_1^1) \right] = \mathbb{E} \left[ H_w^M(\mathcal{L}_1^2) \right].$$
We remark that the measurability of $H_w^M$ is a consequence of Lemma \ref{LemmaMeasExp}. Taking the limit as $M \rightarrow \infty$ and applying the monotone convergence theorem gives
\begin{equation}\label{S45}
p_w = \mathbb{E} \left[ H_w(\mathcal{L}_1^1) \right] = \lim_{M\rightarrow \infty} \mathbb{E} \left[ H_w^M(\mathcal{L}_1^1) \right] = \lim_{M\rightarrow \infty} \mathbb{E} \left[ H_w^M(\mathcal{L}_1^2) \right]=  \mathbb{E} \left[ H_w(\mathcal{L}_1^2) \right].
\end{equation}
On the other hand, by the partial Brownian Gibbs property, cf. Definition \ref{DefPBGP}, and the tower property we have for $v \in \{1, 2\}$ that
\begin{equation*}
\begin{split}
\mathbb{E} \left[ H_w^M(\mathcal{L}_1^v) \right] &=  \mathbb{E} \left[ \mathbb{E} \left[ H_w^M(\mathcal{L}_1^v) \big{\vert} \mathcal{F}_{ext}(\{1\} \times (a_w, b_w)) \right] \right] \\
&= \mathbb{E} \left[ \mathbb{P}_{avoid}^{v,w}\left(\mathcal{Q}(t_1) \leq x_1 \right) \cdot \min\left(M, \frac{1}{F(x_1; a_w, b_w, \mathcal{L}^v_1(a_w), \mathcal{L}^v_1(b_w))}\right) \right],
\end{split}
\end{equation*}
where we wrote $ \mathbb{P}_{avoid}^{v,w}$ in place of $\mathbb{P}_{avoid}^{a_w, b_w, \mathcal{L}^v_1(a_w), \mathcal{L}^v_1(b_w), \infty, \mathcal{L}^v_2[a_w,b_w]}$ to simplify the expression, and where $\mathcal{Q}$ is a Brownian bridge, going from $\mathcal{L}^v_1(a_w)$ to $\mathcal{L}^v_1(b_w)$ on the time interval $[a_w, b_w]$ and staying above $\mathcal{L}^v_2[a_w,b_w]$. Taking the limit as $M \rightarrow \infty$ and utilizing the monotone convergence theorem again we see that for $v \in \{1,2\}$
\begin{equation*}
p_w =  \mathbb{E} \left[  \frac{\mathbb{P}_{avoid}^{v,w}\left( \mathcal{Q}(t_1) \leq x_1 \right)}{F(x_1; a_w, b_w, \mathcal{L}^v_1(a_w), \mathcal{L}^v_1(b_w))} \right].
\end{equation*}

The key observation that motivates much of the proof is that in a sense, which will be made precise later, we have for large enough $w$ that
\begin{equation}\label{KeyObservation}
{\bf 1} \{ \mathcal{L}_2^v(t_1) \leq x_1 \} \approx \frac{\mathbb{P}_{avoid}^{v,w}\left( \mathcal{Q}(t_1) \leq x_1 \right)}{F(x_1; a_w, b_w, \mathcal{L}^v_1(a_w), \mathcal{L}^v_1(b_w))}.
\end{equation}
 To begin understanding (\ref{KeyObservation}) we note that if $\mathcal{L}_2^v(t_1) > x_1$ we know that $\mathbb{P}_{avoid}^{v,w}\left( \mathcal{Q}(t_1) \leq x_1 \right) = 0$, since $\mathcal{Q}(t_1) \geq \mathcal{L}_2^v(t_1) > x_1$. In addition,  by Lemma \ref{MCLfg} applied to $a = a_w$, $b = b_w$, $\vec{x} = \mathcal{L}^v_1(a_w)$, $\vec{y} = \mathcal{L}^v_1(a_w)$, $g^t = \mathcal{L}^v_2[a_w, b_w]$ and $g^b = -\infty$ we know that 
$$\mathbb{P}_{avoid}^{v,w}\left( \mathcal{Q}(t_1) \leq x_1 \right) \leq F(x_1; a_w, b_w, \mathcal{L}^v_1(a_w), \mathcal{L}^v_1(b_w)).$$
Explained in simple words, the quantities on the left and right side of the above inequality both measure the probability that a Brownian bridge from $B(a_w) = \mathcal{L}^v_1(a_w)$ to $B(b_w) =  \mathcal{L}^v_1(b_w)$ has its midpoint below $x_1$, with the difference that on the left side the Brownian bridge is conditioned on staying above the curve $\mathcal{L}^v_2[a_w, b_w]$. The content of Lemma \ref{MCLfg} is that such a conditioning stochastically pushes the bridge up, making it less likely to fall below the point $x_1$. Combining the last two arguments we conclude that $\mathbb{P}_v$-almost surely we have
\begin{equation}\label{DominationV1}
\frac{\mathbb{P}_{avoid}^{v,w}\left( \mathcal{Q}(t_1) \leq x_1 \right)}{F(x_1; a_w, b_w, \mathcal{L}^v_1(a_w), \mathcal{L}^v_1(b_w))} \leq {\bf 1} \{ \mathcal{L}^v_2(t_1) \leq x_1\}.\end{equation}
This establishes a one-sided inequality for (\ref{KeyObservation}). The reverse inequality will be weaker in two ways. Firstly, we will replace $\{ \mathcal{L}_2^v(t_1) \leq x_1 \}$ with $\{  \mathcal{L}_2^v(t_1) < x_1 - \epsilon \}$ and secondly, the inequality will not be in the almost sure sense, but in some average sense for large enough $w$. We will make these statements precise later, and continue discussing the heuristics behind the fact that on the event $\{  \mathcal{L}_2^v(t_1) < x_1 - \epsilon \}$ the right side of (\ref{KeyObservation}) is approximately $1$ with high probability.\\

Suppose that $\mathcal{L}_2^v(t_1) < x_1 - 2\epsilon_2$ for some small $\epsilon_2$. Then the continuity of $\mathcal{L}^v_2$ implies that with high probability the whole curve $\mathcal{L}^v_2[a_w, b_w]$ lies below $x_1 - \epsilon_2$ (as long as $w$ is sufficiently large). In addition, by making $\epsilon_2$ small we can make the event $\mathcal{L}_1^v(t_1) \in (x_1 - \epsilon_2, x_1 + \epsilon_2)$ very unlikely. The latter is true since by Lemma \ref{LNoAtoms} the random variable $\mathcal{L}^v_1(t_1)$ has no atoms. Also, since $\mathcal{L}^v_1$ is continuous, we know that the whole curve $\mathcal{L}^v_1[a_w, b_w]$ will be bounded away from $x_1$ for large enough $w$. We are thus naturally split into the two situations of arguing that 
\begin{equation*}
 \frac{\mathbb{P}_{avoid}^{v,w}\left( \mathcal{Q}(t_1) \leq x_1 \right)}{F(x_1; a_w, b_w, \mathcal{L}^v_1(a_w), \mathcal{L}^v_1(b_w))} \approx 1,
\end{equation*}
when $\mathcal{L}^v_1[a_w, b_w]$ stays above $x_1 + 2\epsilon_2/5$ or below $x_1 - 2\epsilon_2/5$. 

If $\mathcal{L}_1^v[a_w,b_w]$ is below $x_1 - 2\epsilon_2/5$ then both $\mathbb{P}_{avoid}^{v,w}\left( \mathcal{Q}(t_1) \leq x_1 \right)$ and $F(x_1; a_w, b_w, \mathcal{L}^v_1(a_w), \mathcal{L}^v_1(b_w))$ can be shown to be close to $1$. In words, if a Brownian bridge is started very low, then its mid-point will also be low with high probability. If $\mathcal{L}_1^v[a_w,b_w]$ is above $x_1 + 2\epsilon_2/5$ then both $\mathbb{P}_{avoid}^{v,w}\left( \mathcal{Q}(t_1) \leq x_1 \right)$ and $F(x_1; a_w, b_w, \mathcal{L}^v_1(a_w), \mathcal{L}^v_1(b_w))$ are very small, but their ratio is actually close to $1$. The reason for this is that $\mathcal{L}^v_2[a_w, b_w]$ is below the line $x_1 - \epsilon_2$, which is very low. Consequently, a Brownian bridge that is started very high, and is conditioned on staying above a curve $\mathcal{L}^v_2[a_w, b_w]$, which is very low, will asymptotically (as $w\rightarrow \infty$) not {\em feel} the effect of this conditioning and behave as if it is a regular Brownian bridge. The latter implies that the probability of any event under the unconditional and conditional Brownian bridge asymptotically are equal. Of course, in this setting both $\mathbb{P}_{avoid}^{v,w}\left( \mathcal{Q}(t_1) \leq x_1 \right)$ and $F(x_1; a_w, b_w, \mathcal{L}^v_1(a_w), \mathcal{L}^v_1(b_w))$ are tail probabilities, but one can show that their ratio is still going to $1$ as $w\rightarrow \infty$. Justifying carefully the heuristic in this paragraph is the focus of the rest of this section.\\

Summarizing the work done so far, we have that $p_w$ as defined in (\ref{S45}) satisfy the first line in (\ref{S42}) and the third inequality of the second line in (\ref{S42}). What remains to be seen is that for $v \in \{1, 2\}$
\begin{equation}\label{Red1V1}
\begin{split}
\mathbb{P}_v( \mathcal{L}^v_2(t_1) < x_1 ) \leq \liminf_{w \rightarrow \infty} p_w\mbox{ where $v \in \{1,2\}$}.
\end{split}
\end{equation}
We will establish (\ref{Red1V1}) in the four steps below, but first we make a couple of remarks. The work done above corresponds to the first three steps in the general proof of Proposition \ref{PropMain} in the next section. The arguments we present below correspond to Steps 4-7. The main flow of the argument of our work here is the same as the general proof, except that the functions $F(r;s,t,x,y)$ get replaced by more involved expressions that are necessary from the fact that we work with general $N \in \mathbb{N}$ and not just $N = 2$. Also in the end of Steps 3 and 4, we will use some exact results about Brownian bridges, which in the general proof get replaced with Lemmas \ref{LemmaBotMax}, \ref{LemmaBotMin} and \ref{LemmaAvoidMax} in Steps 6 and 7. \\

{\bf \raggedleft Step 1.} In this step we state a simple reduction of (\ref{Red1V1}). Afterwards we define two sequences $p_w^v$ for $v \in \{1,2\}$, which will play an important role in our arguments.

 Firstly, we claim that for any $\epsilon_3 > 0$ and $v \in \{1, 2\}$ we have
\begin{equation}\label{Red2V1}
\mathbb{P}_v \left( \mathcal{L}^v_{2}(t_1) < x_1 \right) - \epsilon_3 \leq \liminf_{w \rightarrow \infty} p_w.
\end{equation}
It is clear that if (\ref{Red2V1}) is true then (\ref{Red1V1}) would follow. We thus focus on establishing (\ref{Red2V1}) and fix $\epsilon_3 > 0$ in the sequel.

We know by Lemma \ref{LNoAtoms} that $\mathbb{P}_v(\mathcal{L}^v_{1}(t_1) = x_1) = 0.$
Consequently, we can find $\epsilon_2 > 0$ (depending on $\epsilon_3$) such that 
\begin{equation}\label{CloseToBarrierV1}
\mathbb{P}_v(\mathcal{L}^v_{1}(t_1) \in [x_1 -\epsilon_2, x_1 + \epsilon_2]) < \epsilon_3/8.
\end{equation}
In addition, by possibly making $\epsilon_2$ smaller we can also ensure that
\begin{equation}\label{SecondCurveBoundV1}
\mathbb{P}_v \left( \mathcal{L}^v_{2}(t_1) < x_1 \right)  - \mathbb{P}_v \left( \mathcal{L}^v_{2}(t_1)< x_1 - 2\epsilon_2 \right) < \epsilon_3/8. 
\end{equation}
This fixes our choice of $\epsilon_2$. 

For a function $f \in C([a,b])$ we define the {\em modulus of continuity} by
\begin{equation*}
w(f,\delta) = \sup_{\substack{x,y \in [a,b]\\ |x-y| \leq \delta}} |f(x) - f(y)|.
\end{equation*}
Since $\mathcal{L}^v_1$ and $\mathcal{L}^v_2$ are continuous on $[a,b]$ almost surely we conclude that there exits $W_0^{-1} > \epsilon_1 > 0$ (depending on $\epsilon_3$ and $\epsilon_2$) such that 
\begin{equation}\label{MOCBoundV1}
\mbox{ if $ E_v  = \{ w( \mathcal{L}^v_i, \epsilon_1) > \epsilon_2/10 \mbox{ for some $i \in \{ 1, 2 \}$} \}$ then $\mathbb{P}_v \left( E_v  \right)  < \epsilon_3/8$. }
\end{equation}
 For $v \in \{1,2\}$ we define the event
$$F_v = \{  \mathcal{L}^v_{2}(x) < x_1 -  \epsilon_2 \mbox{ for $x \in [t_1- \epsilon_1, t_1 + \epsilon_1]$}    \}.$$

Define sequences $p_w^v$ for $v \in \{1, 2\}$ through
\begin{equation*}
p_w^v = \mathbb{E} \left[ {\bf 1}_{E^c_v} \cdot {\bf 1}_{F_v} \cdot \frac{ \mathbb{P}_{avoid}^{v,w}\left(\mathcal{Q}( t_1) \leq x_1 \right)}{F(x_1; a_w,b_w,\mathcal{L}^v_1(a_w),\mathcal{L}^{v}_1(b_w))}   \right].
\end{equation*}
 We claim that 
\begin{equation}\label{Red3V1}
\mathbb{P}_v \left( F_v  \right) - 3\epsilon_3/4\leq \liminf_{w \rightarrow \infty} p^v_w.
\end{equation}
We will prove (\ref{Red3V1}) in the steps below. For now we assume its validity and prove (\ref{Red2V1}). Observe that by definition we have $p_w \geq p_w^v$ and so by (\ref{Red3V1}) we have
$$\mathbb{P}_v \left( F_v \right) - 3\epsilon_3/4 \leq \liminf_{w \rightarrow \infty} p_w,$$
In addition, by the definition of $\epsilon_1$ we know that 
\begin{align*}
\mathbb{P}_v \left( F_v \right) &= \mathbb{P}_v \left( F_v \cap E_v \right) + \mathbb{P}_v \left( F_v \cap E^c_v \right) \geq  \mathbb{P}_v \left( F_v \cap E^c_v \right)\\
&\geq \mathbb{P}_v \left( \left\{ \mathcal{L}^v_{2}(t_1) <  x_1 - 2\epsilon_2 \right\} \cap E^c_v \right) \geq \mathbb{P}_v \left( \mathcal{L}^v_{2}(t_1) <  x_1 \right)  - \epsilon_3/4,
\end{align*}
where in the last inequality we used (\ref{SecondCurveBoundV1}) and (\ref{MOCBoundV1}). The last two inequalities imply (\ref{Red2V1}).\\

{\bf \raggedleft Step 2.} Our focus in the remaining steps is to prove (\ref{Red3V1}). We define the events 
$$A_v = \{\mathcal{L}^v_{1}(t_1) \in [x_1 -\epsilon_2/2, x_1 + \epsilon_2/2] \}.$$ 
We claim that $\mathbb{P}_v$-almost surely for all $w$ sufficiently large we have 
\begin{equation}\label{Red4V1}
{\bf 1}_{E^c_v \cap F_v \cap A_v^c} \cdot\frac{ \mathbb{P}_{avoid}^{v,w}\left(\mathcal{Q}( t_1) \leq x_1 \right)}{F(x_1; a_w,b_w,\mathcal{L}^v_1(a_w),\mathcal{L}^v_1(b_w))} \geq {\bf 1}_{E^c_v \cap F_v \cap A_v^c} \cdot \left( 1- \epsilon_3/4 \right).
\end{equation}
We will prove (\ref{Red4V1}) in the steps below. For now we assume its validity and conclude the proof of (\ref{Red3V1}). In view of (\ref{Red4V1}) we know that 
$$\liminf_{w \rightarrow \infty} p^v_w \geq \mathbb{P}_v \left(E_v^c \cap F_v \cap A^c_v\right) - \epsilon_3/4 \geq \mathbb{P}_v(F_v) -  \mathbb{P}_v(E_v) - \mathbb{P}_v(A_v) - \epsilon_3/4 \geq  \mathbb{P}_v(F_v) - \epsilon_3/2,$$
where in the last inequality we used (\ref{CloseToBarrierV1}) and (\ref{MOCBoundV1}). The above clearly implies (\ref{Red3V1}).\\

{\bf \raggedleft Step 3.} We claim that $\mathbb{P}_v$-almost surely
\begin{equation*}
\lim_{w \rightarrow \infty} {\bf 1}_{E^c_v \cap F_v \cap A_v^c} \cdot\frac{ \mathbb{P}_{avoid}^{v,w}\left(\mathcal{Q}( t_1) \leq x_1 \right)}{F(x_1; a_w,b_w,\mathcal{L}^v_1(a_w),\mathcal{L}^{v}_1(b_w))} = {\bf 1}_{E^c_v \cap F_v \cap A_v^c},
\end{equation*}
which clearly implies (\ref{Red4V1}). In view of (\ref{DominationV1}) we know that the right side above is greater than or equal to each term on the left. Consequently, it suffices to show that $\mathbb{P}_v$-almost surely
\begin{equation}\label{Red6InfV1}
\liminf_{w \rightarrow \infty} {\bf 1}_{E^c_v \cap F_v \cap A_v^c} \cdot\frac{ \mathbb{P}_{avoid}^{v,w}\left(\mathcal{Q}( t_1) \leq x_1 \right)}{F(x_1; a_w,b_w,\mathcal{L}^v_1(a_w),\mathcal{L}^{v}_1(b_w))}\geq {\bf 1}_{E^c_v \cap F_v \cap A_v^c}.
\end{equation}
Let $\omega \in E^c_v \cap F_v \cap A_v^c$ be fixed. Then $\omega \in A_v^c$ and so $\mathcal{L}_{1}^v(t_1) > x_1 + \epsilon_2/2$ or $\mathcal{L}_{1}^v(t_1) < x_1 - \epsilon_2/2$, which we treat separately. We will handle the case when $\mathcal{L}_1^v(t_1) < x_1 - \epsilon_2/2$ in this step, and postpone the other case to the next step. See also Figure \ref{S4_11}. 
\begin{figure}[ht]
\begin{center}
  \includegraphics[scale = 0.8]{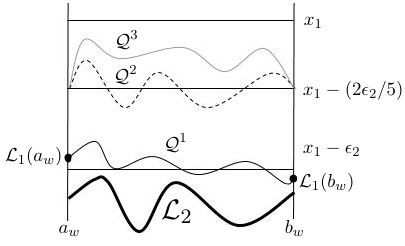}
  \captionsetup{width= 0.9\linewidth}
  \vspace{-2mm}
  \caption{The figure represents schematically the situation when $\omega \in E^c_v \cap F_v \cap A_v^c$ and $\mathcal{L}_1^v(t_1) < x_1 - \epsilon_2/2$. We remark that to make the picture comprehensible we have distorted it and in fact one has that $b_w - a_w$ is much smaller than $\epsilon_2$. In addition, to ease the notation we have removed the dependence on $v \in \{1,2\}$. The curve $\mathcal{Q}^1$ is a Brownian bridge between the points $\mathcal{L}_1(a_w)$ and $\mathcal{L}_1(b_w)$ conditioned on staying above $\mathcal{L}_2$ and proving (\ref{Red6InfV1}) in the case we consider in Step 3 boils down to showing that $\mathbb{P} (\mathcal{Q}^1(t_1) \leq x_1)$ converges to $1$ as $w \rightarrow \infty$ (recall that $t_1 = (b_w + a_w)/2$). This reduction is the first line of (\ref{MEst0V1}). What one observes further is that if $\mathcal{L}_1^v(t_1) < x_1 - \epsilon_2/2$ and $\omega \in E^c_v$ then $\mathcal{L}^v_1(a_w) \leq x_1 - (2\epsilon_2/5)$ and $\mathcal{L}^v_1(b_w) \leq x_1 - (2\epsilon_2/5)$. If we thus construct a Brownian bridge $\mathcal{Q}^2$ starting and ending at $x_1 - (2\epsilon_2/5)$ and conditioned to stay above $\mathcal{L}_2$, then this bridge can be coupled with $\mathcal{Q}^1$ in view of Lemma \ref{MCLxy} so that it sits above it. Finally, on the event $F_v$ the curve $\mathcal{L}^v_2$ lies below the line $x_1 - \epsilon_2$. This means that we can construct a bridge $\mathcal{Q}^3$ with the same starting and ending points as $\mathcal{Q}^2$ but conditioned to stay above the line $x_1 - \epsilon_2$, and this bridge can be coupled with $\mathcal{Q}^2$ in view of Lemma \ref{MCLfg} so that it sits above it. Since each construction pushes the curves upwards, the probability $\mathbb{P}(\mathcal{Q}^3(t_1) \leq x_1)$ is a lower bound for $\mathbb{P}(\mathcal{Q}^1(t_1) \leq x_1)$ and hence it suffices to show that the former is going to $1$ as $w\rightarrow \infty$. This reduction is the content of (\ref{MEst0V1}), where $\mathcal{Q}$ is used to denote all three of the above random curves the distinction being obvious from the notation used for $\mathbb{P}$. Showing that $\mathbb{P}(\mathcal{Q}^3(t_1) \leq x_1)$ converges to $1$ can be seen as follows. The curve $\mathcal{Q}^3$ is a Brownian bridge that is pinned at level $x_1 - 2\epsilon_2/5$, and is conditioned to stay above $x_1 - \epsilon_2$. The order of its typical fluctuation is $w^{-1/2}$ and this is much lower than $\epsilon_2$, which effectively means that $\mathcal{Q}^3$ does not feel its conditioning on staying above $x_1 - \epsilon_2$ (as this level is extremely low in the scale of the fluctuations) and behaves like a regular Brownian bridge. But a regular Brownian bridge started very low is very likely to have a low midpoint as well. The latter heuristic can be justified with simple exact computations, which are done in (\ref{MEst1V1}).
}
  \label{S4_11}
  \end{center}
\end{figure}

Suppose that $\omega \in E^c_v \cap F_v \cap A_v^c$ is such that $\mathcal{L}_{1}^v(t_1)< x_1 - \epsilon_2/2$ and let $W_1 \geq W_0$ be sufficiently large so that $W_1^{-1} < \epsilon_1$. Then for $w \geq W_1$ we have
\begin{equation}\label{MEst0V1}
\begin{split}
&\frac{ \mathbb{P}_{avoid}^{v,w}\left(\mathcal{Q}( t_1) \leq x_1 \right)}{F(x_1; a_w,b_w,\mathcal{L}^v_1(a_w),\mathcal{L}^{v}_1(b_w))}  \geq  \mathbb{P}_{avoid}^{a_w, b_w, \mathcal{L}^v_1(a_w),\mathcal{L}^{v}_1(b_w), \infty, \mathcal{L}^v_2[a_w, b_w]}\left(\mathcal{Q}( t_1) \leq x_1 \right)  \geq \\
&  \geq \mathbb{P}_{avoid}^{a_w, b_w, x_1 -  2\epsilon_2/5,x_1 -  2\epsilon_2/5, \infty, \mathcal{L}^v_2[a_w, b_w]}\left(\mathcal{Q}( t_1) \leq x_1 \right) \geq \\
& \geq \mathbb{P}_{avoid}^{a_w, b_w,x_1 -  2\epsilon_2/5,x_1 -  2\epsilon_2/5, \infty,x_1 - \epsilon_2}\left(\mathcal{Q}( t_1) \leq x_1 \right).
\end{split}
\end{equation}
In the first inequality we used that $F \in (0, 1]$ and the definition of $ \mathbb{P}_{avoid}^{v,w}$. To see the second inequality we note that since $\omega \in E^c_v$ we know that 
$$|\mathcal{L}^v_{1}(a_w) - \mathcal{L}^v_{1}(t_1)| \leq \epsilon_2/10 \mbox{ and }|\mathcal{L}^v_{1}(b_w) - \mathcal{L}^v_{1}(t_1)| \leq \epsilon_2/10,$$
which implies that 
$$\mathcal{L}^v_{1}(a_w)  \leq x_1 -  2\epsilon_2/5  \mbox{ and } \mathcal{L}^v_{1}(b_w) \leq x_1 -  2\epsilon_2/5 .$$
The above inequalities and Lemma \ref{MCLxy} imply the second inequality in (\ref{MEst0V1}). In deriving the third inequality we used that on $F_v$ the curve $ \mathcal{L}^v_2[a_w, b_w]$ is upper bounded by $x_1 - \epsilon_2$ and Lemma \ref{MCLfg}.\\

We consequently observe that
\begin{equation}\label{MEst1V1}
\begin{split}
&  \mathbb{P}_{avoid}^{a_w, b_w,x_1 - 2\epsilon_2/5, x_1 - 2\epsilon_2/5, \infty,x_1 - \epsilon_2}\left(\mathcal{Q}( t_1) \leq x_1 \right) = \\
&= \frac{ \mathbb{P}_{free}^{a_w, b_w,x_1 - 2\epsilon_2/5, x_1 - 2\epsilon_2/5 }\left(\mathcal{Q}( t_1) \leq x_1 \mbox{ and } \inf_{x \in [a_w, b_w]}\mathcal{Q}(x) \geq x_1 - \epsilon_2 \right)}{ \mathbb{P}_{free}^{a_w, b_w,x_1 - 2\epsilon_2/5, x_1 - 2\epsilon_2/5 }\left( \inf_{x \in [a_w, b_w]}\mathcal{Q}(x) \geq x_1 - \epsilon_2 \right)} \\
&\geq 1 - \frac{ \mathbb{P}_{free}^{a_w, b_w,x_1 - 2\epsilon_2/5, x_1 - 2\epsilon_2/5 }\left(\mathcal{Q}( t_1) > x_1  \right)}{ \mathbb{P}_{free}^{a_w, b_w,x_1 - 2\epsilon_2/5, x_1 - 2\epsilon_2/5 }\left( \inf_{x \in [a_w, b_w]}\mathcal{Q}(x) \geq x_1 - \epsilon_2 \right)} \\
&= 1 - \frac{ \mathbb{P}_{free}^{-1, 1,0, 0 }\left(\mathcal{Q}(0) > 2q \sqrt{w}  \right)}{ \mathbb{P}_{free}^{-1, 1,0, 0 }\left( \inf_{x \in [-1,1]}\mathcal{Q}(x) \geq - 3q \sqrt{w} \right)} = 1 - \frac{1 - \Phi( 2q \sqrt{2w}) }{1-  \exp(-9q^2w)}.
\end{split}
\end{equation}
where in the first equality we used Definition \ref{DefAvoidingLaw}, in the next to last equality follows from a simple change of variables (here $q = \epsilon_2/5$), cf. Lemma \ref{LemmaAffine} for $k = 1$. In the last equality, the denominators are equal by Lemma \ref{LemmaBBmax} and the numerators are equal, since $\mathcal{Q}(0)$ is normally distributed with mean $0$ and variance $1/2$ (recall that $\Phi$ was the cdf of a standard Gaussian random variable). Combining (\ref{MEst0V1}) and (\ref{MEst1V1}) we conclude (\ref{Red6InfV1}) when $\omega \in E^c_v \cap F_v \cap A_v^c$ is such that $\mathcal{L}_{1}^v(t_1)< x_1 - \epsilon_2/2$.\\

{\bf \raggedleft Step 4.}  Suppose that $\omega \in E^c_v \cap F_v \cap A_v^c$ is such that $\mathcal{L}_1^v(t_1) > x_1 + \epsilon_2/2$ and let $W_1 \geq W_0$ be sufficiently large so that $W_1^{-1} < \epsilon_1$. See also Figure \ref{S4_12}.

\begin{figure}[ht]
\begin{center}
  \includegraphics[scale = 0.8]{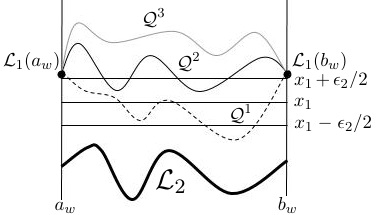}
  \captionsetup{width= 0.9\linewidth}
  \vspace{-2mm}
  \caption{The figure represents schematically the situation when $\omega \in E^c_v \cap F_v \cap A_v^c$ and $\mathcal{L}_1^v(t_1) > x_1 + \epsilon_2/2$. We remark that to make the picture comprehensible we have distorted it and in fact one has that $b_w - a_w$ is much smaller than $\epsilon_2$. In addition, to ease the notation we have removed the dependence on $v \in \{1,2\}$. The curve $\mathcal{Q}^1$ is a Brownian bridge between the points $\mathcal{L}_1(a_w)$ and $\mathcal{L}_1(b_w)$ and $\mathcal{Q}^2$ is a Brownian bridge between the same points but conditioned on staying above $\mathcal{L}_2$. Proving (\ref{Red6InfV1}) in the case we consider in Step 4 boils down to showing that $\mathbb{P} (\mathcal{Q}^2(t_1) \leq x_1) / \mathbb{P} (\mathcal{Q}^1(t_1) \leq x_1)$ is lower bounded by $1$ as $w \rightarrow \infty$ (recall that $t_1 = (b_w + a_w)/2$). On the event $F_v$ the curve $\mathcal{L}^v_2$ lies below the line $x_1 - \epsilon_2$. This means that we can construct a bridge $\mathcal{Q}^3$ with the same starting and ending points as $\mathcal{Q}^2$ but conditioned to stay above the line $x_1 - \epsilon_2$, and this bridge can be coupled with $\mathcal{Q}^2$ in view of Lemma \ref{MCLfg} so that it sits above it. Since this construction pushes the curve upwards, the probability $\mathbb{P}(\mathcal{Q}^3(t_1) \leq x_1)$ is a lower bound for $\mathbb{P}(\mathcal{Q}^2(t_1) \leq x_1)$ and hence it suffices to show that $\mathbb{P} (\mathcal{Q}^3(t_1) \leq x_1) / \mathbb{P} (\mathcal{Q}^1(t_1) \leq x_1)$ is lower bounded by $1$ as $w \rightarrow \infty$. This reduction is the content of (\ref{MSP1}). The curve $\mathcal{Q}^3$ is a Brownian bridge that is pinned at level $x_1 + \epsilon_2/2$, and is conditioned to stay above $x_1 - \epsilon_2$. The order of its typical fluctuation is $w^{-1/2}$ and this is much lower than $\epsilon_2$, which effectively means that $\mathcal{Q}^3$ does not feel its conditioning on staying above $x_1 - \epsilon_2$ (as this level is extremely low in the scale of the fluctuations) and behaves like $\mathcal{Q}^1$. Of course, $\mathbb{P} (\mathcal{Q}^3(t_1) \leq x_1)$ and $ \mathbb{P} (\mathcal{Q}^1(t_1) \leq x_1)$ are tail probabilities but one can still show that their ratio is close to $1$. This is done in (\ref{S48V1}) and the equations that follow it.
}
  \label{S4_12}
  \end{center}
\end{figure}

For $w \geq W_1$ we have
\begin{equation}\label{MSP1}
\frac{ \mathbb{P}_{avoid}^{v,w}\left(\mathcal{Q}( t_1) \leq x_1 \right)}{F(x_1; a_w,b_w,\mathcal{L}_1^v(a_w), \mathcal{L}^v_1(b_w))} \geq  \frac{\mathbb{P}_{avoid}^{a_w, b_w, \mathcal{L}_1^v(a_w), \mathcal{L}^v_1(b_w), \infty,x_1 - \epsilon_2}\left(\mathcal{Q}( t_1) \leq x_1 \right)}{F(x_1; a_w,b_w,\mathcal{L}_1^v(a_w), \mathcal{L}^v_1(b_w))}, 
\end{equation}
where we used that on $F_v$ the curve $ \mathcal{L}^v_2[a_w, b_w]$ is upper bounded by $x_1 - \epsilon_2$ and  Lemma \ref{MCLfg}.
We next notice that by Definition \ref{DefAvoidingLaw} the numerator on the right side equals
$$\frac{\mathbb{P}^{a_w, b_w, \mathcal{L}_1^v(a_w), \mathcal{L}^v_1(b_w)}_{free}\left(\inf_{x\in[a_w, b_w]} \mathcal{Q}(x) \geq x_1 - \epsilon_2 \mbox{ and } \mathcal{Q}(t_1) \leq x_1 \right)}{\mathbb{P}^{a_w, b_w, \mathcal{L}_1^v(a_w), \mathcal{L}^v_1(b_w)}_{free}\left(\inf_{x\in[a_w, b_w]} \mathcal{Q}(x) \geq x_1 - \epsilon_2 \right)}. $$
Combining the last two statements, and performing a change of variables (we use Lemma \ref{LemmaAffine} for $k = 1$) we conclude that
\begin{equation}\label{S48V1}
\begin{split}
&\frac{ \mathbb{P}_{avoid}^{v,w}\left(Q( t_1) \leq x_1 \right)}{F(x_1; a_w,b_w,\mathcal{L}_1^v(a_w), \mathcal{L}^v_1(b_w))} \geq \frac{\tilde{\mathbb{P}}_{v,w}(\inf_{x \in [-1,1]}\tilde{Q}(x) \geq -\epsilon_2 \cdot \sqrt{w} \mbox{ and }\tilde{Q}(0) \leq 0 ) }{\tilde{\mathbb{P}}_{v,w}(\tilde{Q}(0) \leq 0) \cdot \tilde{\mathbb{P}}_{v,w}(\inf_{x \in [-1,1]}\tilde{Q}(x) \geq -\epsilon_2 \cdot \sqrt{w} ) } \geq \\
& \geq \frac{\tilde{\mathbb{P}}_{v,w}(\inf_{x \in [-1,1]}\tilde{Q}(x) \geq -\epsilon_2  \sqrt{w} \mbox{ and }\tilde{Q}(0) \leq 0 ) }{\tilde{\mathbb{P}}_{v,w}(\tilde{Q}(0) \leq 0) } \geq 1 -  \frac{\tilde{\mathbb{P}}_{v,w}(\inf_{x \in [-1,1]}\tilde{Q}(x) < -\epsilon_2  \sqrt{w})}{\tilde{\mathbb{P}}_{v,w}(\tilde{Q}(0) \leq 0) }  ,
\end{split}
\end{equation}
where $\tilde{\mathbb{P}}_{v,w}$ denotes $\mathbb{P}_{free}^{-1,1,A,B}$ with $A_{v,w} = [\mathcal{L}^v_1(a_w) - x_1] \cdot \sqrt{w}$ and $B_{v,w}= [\mathcal{L}^v_1(b_w) - x_1] \cdot \sqrt{w} $ and $\tilde{Q}$ is a $\tilde{\mathbb{P}}_{v,w}$-distributed Brownian bridge.

We now have by the Gaussianity of $\tilde{Q}(0)$ that 
$$\tilde{\mathbb{P}}_{v,w}(\tilde{Q}(0) \leq 0) = 1 - \Phi \left( \frac{A_{v,w} + B_{v,w} }{\sqrt{2}}\right). $$
 Also we have since $\omega \in E^c_v$ that $A_{v,w} \geq (2\epsilon_2/5) \sqrt{w}, B_{v,w} \geq  (2\epsilon_2/5) \sqrt{w}$. Consequently,
$$\tilde{\mathbb{P}}_{v,w}\left(\inf_{x \in [-1,1]}\tilde{Q}(x) < -\epsilon_2 \cdot \sqrt{w}\right) = \exp \left( - [\epsilon_2\sqrt{w} + A_{v,w}][ \epsilon_2 \sqrt{w} + B_{v.w}] \right),$$
where the last equality follows from Lemma \ref{LemmaBBmax}. Let $M_{v,w} = \max(A_{v,w}, B_{v,w})$ and $m_{v,w} = \min(A_{v,w}, B_{v,w})$. Since $\omega \in E^c_v$ we know that $M_{v,w} - m_{v,w} \leq (\epsilon_2/5 ) \sqrt{w}$. Consequently, the above two equalities imply that 
\begin{align*}
\tilde{\mathbb{P}}_{v,w}(\tilde{Q}(0) \leq 0) &\geq \frac{\exp \left( - M_{v,w}^2  \right)}{c_0 \sqrt{2\pi} [1 + \sqrt{2} M_{v,w}]}, \\
\tilde{\mathbb{P}}_{v,w}\left(\inf_{x \in [-1,1]}\tilde{Q}(x) < -\epsilon_2 \cdot \sqrt{w}\right) &\leq \exp \left( - [M_{v,w} + 5 q \sqrt{w}] [M_{v,w} + 4q \sqrt{w}] \right),
\end{align*}
where in the first inequality we used  Lemma \ref{LemmaI1} (here $c_0$ is as in this lemma and is a universal constant), and also $q = \epsilon_2/5.$ Combining the last two inequalities with (\ref{S48V1}) we see that 
$$\frac{ \mathbb{P}_{avoid}^{v,w}\left(Q( t_1) \leq x_1 \right)}{F(x_1; a_w,b_w,\mathcal{L}_1^v(a_w), \mathcal{L}^v_1(b_w))} \geq1 - c_0 [1 + \sqrt{2} M_{v,w}] \cdot \exp \left(- [M_{v,w} + 5 q \sqrt{w}] [M_{v,w} + 4q \sqrt{w}] + M_{v,w}^2 \right) . $$
Since $M_{v,w} \geq 2q \sqrt{w}$ we see that the above expression converges to $1$ as $w \rightarrow \infty$, which proves (\ref{Red6InfV1}) when $\omega \in E^c_v \cap F_v \cap A_v^c$ is such that $\mathcal{L}_{1}^v(t_1)> x_1 + \epsilon_2/2$. This concludes the proof of (\ref{Red6InfV1}) and hence the proposition in this basic case.

%
%
\subsection{Proof of Proposition \ref{PropMain}}\label{Section4.2} Here we present the proof of Proposition \ref{PropMain}. We assume the same notation as in Sections \ref{Section2} and \ref{Section3}. We proceed by induction on $N$ with base case $N = 1$, being obvious. Suppose that we know the result for $N - 1$ and wish to prove it for $N$. 
Suppose that $k \in \mathbb{N}$ and $a = t_0 < t_1 < \cdots < t_k < t_{k+1} = b$, $n_1, \dots, n_k \in \llbracket 1, N \rrbracket$ and $y_1, \dots, y_k \in \mathbb{R}$ are all given. We wish to prove that 
\begin{equation}\label{FDEPMr1}
\mathbb{P}_1 \left( \mathcal{L}^1_{n_1}(t_1) \leq y_1, \dots,\mathcal{L}^1_{n_k}(t_k) \leq y_k  \right) =\mathbb{P}_2 \left( \mathcal{L}^2_{n_1}(t_1) \leq y_1, \dots,\mathcal{L}^2_{n_k}(t_k) \leq y_k \right).
\end{equation}

For clarity we split the proof into several steps. \\

{\bf \raggedleft Step 1.} Let $x_1, \dots, x_k \in \mathbb{R}$ be fixed. We claim that there exists a sequence $\{ p_w\}_{w = 1}^\infty$ with $p_w \in [0,1]$ such that for $v \in \{1,2\}$ we have
\begin{equation}\label{Red1}
\begin{split}
&\mathbb{P}_v \left( \mathcal{L}^v_{n_1}(t_1) < x_1, \dots,\mathcal{L}^v_{n_k}(t_k) < x_k  \right) \leq \liminf_{w \rightarrow \infty} p_w \\
&\leq \limsup_{w \rightarrow \infty} p_w \leq \mathbb{P}_v \left( \mathcal{L}^v_{n_1}(t_1) \leq x_1, \dots,\mathcal{L}^v_{n_k}(t_k) \leq x_k  \right).
\end{split}
\end{equation}
We will prove (\ref{Red1}) in the steps below. For now we assume its validity and finish the proof of (\ref{FDEPMr1}). For $x_1, \dots, x_k \in \mathbb{R}$ and $v \in \{1, 2\}$ we let
$$F_v(x_1, \dots, x_k)=\mathbb{P}_v \left( \mathcal{L}^v_{n_1}(t_1) \leq x_1, \dots,\mathcal{L}^v_{n_k}(t_k) \leq x_k  \right).$$
We also define for $v \in \{1,2\}$ and $r \in \mathbb{R}$
$$G_v(r)= F_v(y_1 + r, y_2 + r, \cdots ,y_k +r).$$
Observe that by basic properties of probability measures we know that $G_1$ and $G_2$ are increasing right-continuous functions. Moreover, if $G_1$ and $G_2$ are both continuous at a point $r$ then from (\ref{Red1}) applied to $x_i = y_i + r$ for $i = 1,\dots, k$ we know that $G_1(r) = G_2(r)$. The latter and Lemma \ref{MonEq} imply that $G_1 = G_2$. In particular, $G_1(0) = G_2(0)$, which is precisely (\ref{FDEPMr1}). \\

{\bf \raggedleft Step 2.} In this step we define the sequence $p_w$ that satisfies (\ref{Red1}). We also introduce some notation that will be used in the rest of the proof.

 Given points $s,t,r \in \mathbb{R}$ with $s < t$ and $\vec{x}, \vec{y} \in W_k^{\circ}$ we define 
\begin{equation*}
F(r; s,t,\vec{x},\vec{y})= \mathbb{P}^{s,t,\vec{x},\vec{y}, \infty, -\infty}_{avoid}(\mathcal{Q}_{N-1}((s +t)/2) \leq r),
\end{equation*}
where $(\mathcal{Q}_{1}, \dots, \mathcal{Q}_{N-1})$ is $ \mathbb{P}^{s,t,\vec{x},\vec{y}, \infty, -\infty}_{avoid}$-distributed. Observe that for fixed $s,t,r$ the function $F(r; s,t,\vec{x},\vec{y})$ is a measurable function of $\vec{x},\vec{y}$ as follows from Lemma \ref{LemmaMeasExp}. For $M \in \mathbb{N}$ we denote 
$$G_M(r; s,t,\vec{x},\vec{y}) = \min \left(M, \frac{1}{F(r; s,t,\vec{x},\vec{y})} \right),$$
and note that $G_M$ is a non-negative bounded measurable function. Let $S = \{ s \in \{1, \dots, k\} : n_s = N\}$.  For $w \in \mathbb{N}$ and $s \in S$ we define $a_s^w = t_s - w^{-1}$ and $b_s^w = t_s + w^{-1}$. We also fix $W_0 \in \mathbb{N}$ sufficiently large so that $w \geq W_0$ implies $2w^{-1} \leq \min_{1 \leq i \leq k+1} (t_i - t_{i-1})$.  

By induction hypothesis, we know that (\ref{FDEPMr1}) holds provided $n_1, \dots, n_k \in \llbracket 1, N-1 \rrbracket$. The latter and Proposition \ref{PropFD} imply that $\pi_{[a,b]}^{\llbracket 1, N-1 \rrbracket}(\mathcal{L}^1)$ under $\mathbb{P}_1$ and $\pi_{[a,b]}^{\llbracket 1, N-1 \rrbracket}(\mathcal{L}^2)$ under $\mathbb{P}_2$ have the same distribution as $\llbracket 1, N-1 \rrbracket$-indexed line ensembles on $[a,b]$. We conclude that for $w \geq W_0$ 
\begin{equation}\label{Trunc1}
\mathbb{E} \left[ H^M_w(\mathcal{L}^1; \vec{t}, \vec{n}, \vec{x}) \right] =  \mathbb{E} \left[ H^M_w(\mathcal{L}^2; \vec{t}, \vec{n}, \vec{x}) \right],
\end{equation}
with
\begin{equation*}
H^M_w(\mathcal{L}^v; \vec{t}, \vec{n}, \vec{x}) = \prod_{s \in S^c} {\bf 1}\{\mathcal{L}_{n_s}^v(t_s) \leq x_s \} \cdot \prod_{s \in S} {\bf 1}\{ \mathcal{L}_{N-1}^v(t_s) \leq x_s\} G_M\left(x_{s}; a^w_s, b^w_s, \vec{x}^{s,v,w}, \vec{y}^{s,v,w} \right),
\end{equation*}
where $\vec{x}^{s,v,w} = (\mathcal{L}_1^v(a_s^w), \dots, \mathcal{L}_{N-1}^v(a_s^w))$ and $\vec{y}^{s,v,w} = (\mathcal{L}_1^v(b_s^w), \dots, \mathcal{L}_{N-1}^v(b_s^w))$ for $v = 1,2$. Some of the notation we defined above is illustrated in Figure \ref{S4_2}.
\begin{figure}[ht]
\begin{center}
  \includegraphics[scale = 0.8]{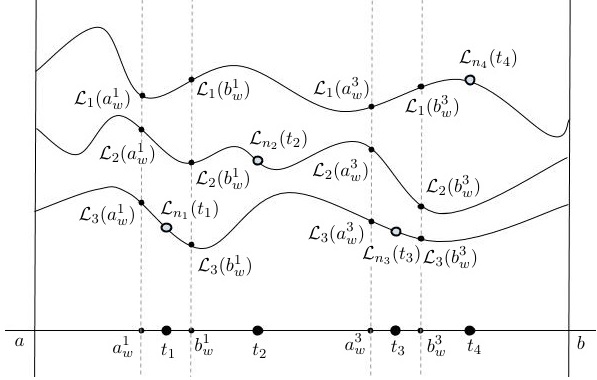}
  \vspace{-2mm}
  \caption{The figure schemiatically represents $\mathcal{L}^v$ where we have suppressed $v$ from the notation. In the figure, $N = 3$, $S = \{1,3\}$ and $n_1 = 3, n_2 = 2, n_3 = 3, n_4 = 1$. }
  \label{S4_2}
  \end{center}
\end{figure}

For $s \in S$ define $\mathcal{F}_{ext}^{s,w} = \mathcal{F}_{ext} (\llbracket 1, N-1 \rrbracket \times (a^w_s,b^w_s))$ as in Definition \ref{DefPBGP} and observe that by the tower property for conditional expectations and the partial Brownian Gibbs property 
\begin{equation}\label{Trunc2}
\begin{split}
&\mathbb{E} \left[H^M_w(\mathcal{L}^v; \vec{t}, \vec{n}, \vec{x})  \right] = \mathbb{E} \left[ \mathbb{E} \left[ \cdots \mathbb{E}\left[ H^M_w(\mathcal{L}^1; \vec{t}, \vec{n}, \vec{x})\Big{|} \mathcal{F}_{ext}^{s_1,w}  \right] \cdots \Big{|} \mathcal{F}_{ext}^{s_m,w} \right] \right] \\
&=\mathbb{E} \left[ \prod_{s \in S^c} {\bf 1}\{\mathcal{L}_{n_s}^v(t_s) \leq x_s \} \prod_{s \in S} \mathbb{P}_{avoid}^{s,v,w}\left(\mathcal{Q}_{N-1}( t_s) \leq x_s\right) \cdot G_M\left(x_{s}; a^w_s, b^w_s, \vec{x}^{s,v,w}, \vec{y}^{s,v,w}\right)  \right],
\end{split}
\end{equation}
where $v \in \{1, 2\}$ and we have written $\mathbb{P}_{avoid}^{s,v,w}$ in place of $\mathbb{P}_{avoid}^{a^w_s, b^w_s,\vec{x}^{s,v,w},\vec{y}^{s,v,w}, \infty, \mathcal{L}_N^v[a^w_s,b^w_s]} $ to simplify the expression, also $(\mathcal{Q}_{1}, \dots, \mathcal{Q}_{N-1})$ is $\mathbb{P}_{avoid}^{s,v,w}$-distributed. In addition, $s_1, \dots, s_m$ is an enumeration of the elements of $S$ and in deriving the above expression, we also used Lemma \ref{LemmaMeasExp}, which implies that 
$\mathbb{P}_{avoid}^{s,v,w}\left(\mathcal{Q}_{N-1}( t_s) \leq x_s\right)$
is measurable with respect to the $\sigma$-algebra
$$ \sigma \left\{ \mathcal{L}^v_i(s) : \mbox{  $i \in \llbracket 1, N-1 \rrbracket$ and $s \in \{a_s^w,b_s^w\}$, or $s \in [a_s^w,b_s^w]$ and $i = N$} \right\}.$$

Combining (\ref{Trunc1}) and (\ref{Trunc2}) and taking the limit as $M \rightarrow \infty$ (using the monotone convergence theorem) we conclude that for any $w \geq W_0$ we have
\begin{equation}\label{Obs2}
\begin{split}
p_w &= \mathbb{E} \left[ \prod_{s \in S}\frac{ \mathbb{P}_{avoid}^{s,1,w}\left(\mathcal{Q}_{N-1}( t_s) \leq x_s \right)}{F(x_s; a^w_s,b^w_s,\vec{x}^{s,1,w},\vec{y}^{s,1,w})} \cdot \prod_{s \in S^c} {\bf 1}\{\mathcal{L}_{n_s}^1(t_s) \leq x_s \}  \right] \\
&= \mathbb{E} \left[ \prod_{s \in S}\frac{ \mathbb{P}_{avoid}^{s,2,w}\left(\mathcal{Q}_{N-1}( t_s) \leq x_s \right)}{F(x_s; a^w_s,b^w_s,\vec{x}^{s,2,w},\vec{y}^{s,2,w})} \cdot \prod_{s \in S^c} {\bf 1}\{\mathcal{L}_{n_s}^2(t_s) \leq x_s \}  \right].
\end{split}
\end{equation}
Equation (\ref{Obs2}) defines the sequence $p_w$ and we show that it satisfies (\ref{Red1}) in the steps below.\\

{\bf \raggedleft Step 3.} In this step we prove the second line in (\ref{Red1}). By Lemma \ref{MCLfg} applied to $a = a_s^w$, $b = b_s^w$, $\vec{x} = \vec{x}^{s,v,w}$, $\vec{y} = \vec{y}^{s,v,w}$, $g^t = \mathcal{L}^v_N[a_s^w, b_s^w]$ and $g^b = -\infty$ we know that 
\begin{equation}\label{Domination}
 \mathbb{P}_{avoid}^{s,v,w} \hspace{-0.5mm}\left(\hspace{-0.5mm}\mathcal{Q}_{N-1}( t_s) \hspace{-0.5mm} \leq \hspace{-0.5mm} x_s\hspace{-0.5mm} \right)\hspace{-0.5mm} \leq  \hspace{-0.5mm} \mathbb{P}_{avoid}^{a^w_s, b^w_s, \vec{x}^{s,v,w},\vec{y}^{s,v,w}\hspace{-1mm} , \infty, -\infty}\hspace{-0.5mm} \hspace{-0.5mm} \left( \hspace{-0.5mm} \mathcal{Q}_{N-1}( t_s) \hspace{-0.5mm} \leq \hspace{-0.5mm} x_s \right) = F(x_s; a^w_s,b^w_s,\vec{x}^{s,v,w}\hspace{-0.5mm},\vec{y}^{s,v,w}).
\end{equation} 
In addition, we observe that on the event $\{ \mathcal{L}^v_N(t_s) > x_s \}$ we have
$ \mathbb{P}_{avoid}^{s,v,w}\left(\mathcal{Q}_{N-1}( t_s) \leq x_s \right) = 0$.
The latter two statements imply that for any $w \geq W_0$ and $v \in \{1,2\}$ we have that 
$$p_w \leq \mathbb{E} \left[ \prod_{s \in S}{\bf 1}\{\mathcal{L}_N^v(t_s) \leq x_s \}  \cdot \prod_{s \in S^c} {\bf 1}\{\mathcal{L}_{n_s}^v(t_s) \leq x_s \}  \right],$$
which clearly implies the second line in (\ref{Red1}). \\

{\bf \raggedleft Step 4.} In this step we state a simple reduction of the first line of (\ref{Red1}). Afterwards we define two sequences $p_w^v$ for $v \in \{1,2\}$, which will play an important role in our arguments.

 Firstly, we claim that for any $\epsilon_3 > 0$ and $v \in \{1, 2\}$ we have
\begin{equation}\label{Red2}
\mathbb{P}_v \left( \mathcal{L}^v_{n_1}(t_1) < x_1, \dots,\mathcal{L}^v_{n_k}(t_k) < x_k  \right) - \epsilon_3 \leq \liminf_{w \rightarrow \infty} p_w.
\end{equation}
It is clear that if (\ref{Red2}) is true then the first line of (\ref{Red1}) would follow. We thus focus on establishing (\ref{Red2}) and fix $\epsilon_3 > 0$ in the sequel.

We know by Lemma \ref{LNoAtoms} that for any $s \in S$ we have 
$\mathbb{P}_v(\mathcal{L}^v_{N-1}(t_s) = x_s) = 0$.
Consequently, we can find $\epsilon_2 > 0$ (depending on $\epsilon_3$) such that 
\begin{equation}\label{CloseToBarrier}
\sum_{s \in S} \mathbb{P}_v(\mathcal{L}^v_{N-1}(t_s) \in [x_s -\epsilon_2, x_s + \epsilon_2]) < \epsilon_3/8.
\end{equation}
In addition, by possibly making $\epsilon_2$ smaller we can also ensure that
\begin{equation}\label{SecondCurveBound}
\mathbb{P}_v \left( \mathcal{L}^v_{n_1}(t_1) \hspace{-0.5mm}  < \hspace{-0.5mm} x_1, \dots,\mathcal{L}^v_{n_k}(t_k)  \hspace{-0.5mm} < \hspace{-0.5mm}x_k  \right)  - \mathbb{P}_v \left( \mathcal{L}^v_{n_1}(t_1) \hspace{-0.5mm} < \hspace{-0.5mm} x_1 - 2\epsilon_2, \dots,\mathcal{L}^v_{n_k}(t_k) \hspace{-0.5mm} < \hspace{-0.5mm}x_k -2\epsilon_2 \right) < \epsilon_3/8. 
\end{equation}
This fixes our choice of $\epsilon_2$. 

For a function $f \in C([a,b])$ we define the {\em modulus of continuity} by
\begin{equation*}
w(f,\delta) = \sup_{\substack{x,y \in [a,b]\\ |x-y| \leq \delta}} |f(x) - f(y)|.
\end{equation*}
Since $\mathcal{L}^v_1, \dots, \mathcal{L}^v_N$ are continuous on $[a,b]$ almost surely we conclude that there exits $W_0^{-1} > \epsilon_1 > 0$ (depending on $\epsilon_3$ and $\epsilon_2$) such that 
\begin{equation}\label{MOCBound}
\mbox{ if $ E_v  = \{ w( \mathcal{L}^v_i, \epsilon_1) > \epsilon_2/10 \mbox{ for some $i \in \llbracket 1, N \rrbracket$} \}$ then $\mathbb{P}_v \left( E_v  \right)  < \epsilon_3/8$. }
\end{equation}
 For $v \in \{1,2\}$ we define the event
$$F_v = \{  \mathcal{L}^v_{n_i}(x) < x_i -  \epsilon_2 \mbox{ for $x \in [t_i- \epsilon_1, t_i + \epsilon_1]$ for $i = 1,\dots, k$}    \}.$$

Define sequences $p_w^v$ for $v \in \{1, 2\}$ through
\begin{equation*}
p_w^v = \mathbb{E} \left[ {\bf 1}_{E^c_v} \cdot {\bf 1}_{F_v} \cdot \prod_{s \in S}\frac{ \mathbb{P}_{avoid}^{s,v,w}\left(\mathcal{Q}_{N-1}( t_s) \leq x_s \right)}{F(x_s; a^w_s,b^w_s,\vec{x}^{s,v,w},\vec{y}^{s,v,w})}   \right].
\end{equation*}
 We claim that 
\begin{equation}\label{Red3}
\mathbb{P}_v \left( F_v  \right) - 3\epsilon_3/4\leq \liminf_{w \rightarrow \infty} p^v_w.
\end{equation}
We will prove (\ref{Red3}) in the steps below. For now we assume its validity and prove (\ref{Red2}). Observe that by definition we have $p_w \geq p_w^v$ and so by (\ref{Red3}) we have
$$\mathbb{P}_v \left( F_v \right) - 3\epsilon_3/4 \leq \liminf_{w \rightarrow \infty} p_w,$$
In addition, by the definition of $\epsilon_1$ we know that 
\begin{align*}
\mathbb{P}_v \left( F_v \right) &= \mathbb{P}_v \left( F_v \cap E_v \right) + \mathbb{P}_v \left( F_v \cap E^c_v \right) \geq  \mathbb{P}_v \left( F_v \cap E^c_v \right)\\
&\geq \mathbb{P}_v \left( \left\{ \mathcal{L}^v_{n_1}(t_1) \hspace{-0.5mm} < \hspace{-0.5mm} x_1 - 2\epsilon_2, \dots,\mathcal{L}^v_{n_k}(t_k) \hspace{-0.5mm} < \hspace{-0.5mm}x_k -2\epsilon_2 \right\} \cap E^c_v \right) \\
&\geq \mathbb{P}_v \left( \mathcal{L}^v_{n_1}(t_1) \hspace{-0.5mm}  < \hspace{-0.5mm} x_1, \dots,\mathcal{L}^v_{n_k}(t_k)  \hspace{-0.5mm} < \hspace{-0.5mm}x_k  \right)  - \epsilon_3/4,
\end{align*}
where in the last inequality we used (\ref{SecondCurveBound}) and (\ref{MOCBound}). The last two inequalities imply (\ref{Red2}).\\

{\bf \raggedleft Step 5.} Our focus in the remaining steps is to prove (\ref{Red3}). We define the events 
$$A_v = \{\mathcal{L}^v_{N-1}(t_s) \in [x_s -\epsilon_2/2, x_s + \epsilon_2/2] \mbox{ for some $s \in S$}\}.$$ 
We claim that $\mathbb{P}_v$-almost surely we have for all $w$ sufficiently large that 
\begin{equation}\label{Red4}
{\bf 1}_{E^c_v \cap F_v \cap A_v^c} \cdot \prod_{s \in S}\frac{ \mathbb{P}_{avoid}^{s,v,w}\left(\mathcal{Q}_{N-1}( t_s) \leq x_s \right)}{F(x_s; a^w_s,b^w_s,\vec{x}^{s,v,w},\vec{y}^{s,v,w})} \geq {\bf 1}_{E^c_v \cap F_v \cap A_v^c} \cdot \left( 1- \epsilon_3/4 \right).
\end{equation}
We will prove (\ref{Red4}) in the steps below. For now we assume its validity and conclude the proof of (\ref{Red3}). In view of (\ref{Red4}) we know that 
$$\liminf_{w \rightarrow \infty} p^v_w \geq \mathbb{P}_v \left(E_v^c \cap F_v \cap A^c_v\right) - \epsilon_3/4 \geq \mathbb{P}_v(F_v) -  \mathbb{P}_v(E_v) - \mathbb{P}_v(A_v) - \epsilon_3/4 \geq  \mathbb{P}_v(F_v) - \epsilon_3/2,$$
where in the last inequality we used (\ref{CloseToBarrier}) and (\ref{MOCBound}). The above clearly implies (\ref{Red3}).\\

{\bf \raggedleft Step 6.} We claim that for each $s \in S$ we have $\mathbb{P}_v$-almost surely
\begin{equation*}
\lim_{w \rightarrow \infty} {\bf 1}_{E^c_v \cap F_v \cap A_v^c} \cdot \frac{ \mathbb{P}_{avoid}^{s,v,w}\left(\mathcal{Q}_{N-1}( t_s) \leq x_s \right)}{F(x_s; a^w_s,b^w_s,\vec{x}^{s,v,w},\vec{y}^{s,v,w})} = {\bf 1}_{E^c_v \cap F_v \cap A_v^c},
\end{equation*}
which clearly implies (\ref{Red4}). In view of (\ref{Domination}) we know that the right side above is greater than or equal to each term on the left. Consequently, it suffices to show that $\mathbb{P}_v$-almost surely
\begin{equation}\label{Red6Inf}
\liminf_{w \rightarrow \infty} {\bf 1}_{E^c_v \cap F_v \cap A_v^c} \cdot \frac{ \mathbb{P}_{avoid}^{s,v,w}\left(\mathcal{Q}_{N-1}( t_s) \leq x_s \right)}{F(x_s; a^w_s,b^w_s,\vec{x}^{s,v,w},\vec{y}^{s,v,w})} \geq {\bf 1}_{E^c_v \cap F_v \cap A_v^c}.
\end{equation}
Let $\omega \in E^c_v \cap F_v \cap A_v^c$ be fixed. Then $\omega \in A_v^c$ and so $\mathcal{L}_{N-1}^v(t_s) > x_s + \epsilon_2/2$ or $\mathcal{L}_{N-1}^v(t_s) < x_s - \epsilon_2/2$, which we treat separately. We will handle the case when $\mathcal{L}_1^v(t_s) < x_s - \epsilon_2/2$ in this step, and postpone the other case to the next step.\\

Suppose that $\omega \in E^c_v \cap F_v \cap A_v^c$ is such that $\mathcal{L}_{N-1}^v(t_s) < x_s - \epsilon_2/2$ and let $W_1 \geq W_0$ be sufficiently large so that $W_1^{-1} < \epsilon_1$. Then for $w \geq W_1$ we have
\begin{equation}\label{MEst0}
\begin{split}
&\frac{ \mathbb{P}_{avoid}^{s,v,w}\left(\mathcal{Q}_{N-1}( t_s) \leq x_s \right)}{F(x_s; a^w_s,b^w_s,\vec{x}^{s,v,w},\vec{y}^{s,v,w})} \geq  \mathbb{P}_{avoid}^{a_s^w, b_s^w, \vec{x}^{s,v,w},\vec{y}^{s,v,w}, \infty, \mathcal{L}^v_N[a_s^w, b_s^w]}\left(\mathcal{Q}_{N-1}( t_s) \leq x_s \right),
\end{split}
\end{equation}
where we used that $F \in (0, 1]$ and the definition of $\mathbb{P}_{avoid}^{s,v,w}$.
Since $\omega \in E^c_v$ we know 
$$|\mathcal{L}^v_{N-1}(a_s^w) - \mathcal{L}^v_{N-1}(t_s)| \leq \epsilon_2/10 \mbox{ and }|\mathcal{L}^v_{N-1}(b_s^w) - \mathcal{L}^v_{N-1}(t_s)|  \leq \epsilon_2/10,$$
which implies that 
$$l_s^w =  x_s -  2\epsilon_2/5 - \mathcal{L}^v_{N-1}(a_s^w) \geq 0 \mbox{ and } r_{s}^w = x_s -  2\epsilon_2/5 -\mathcal{L}^v_{N-1}(b_s^w) \geq 0.$$
Let $\vec{\lambda}^{s,v,w}$ be defined as $\vec{\lambda}^{s,v,w}_i = \vec{x}^{s,v,w}_i + l_s^w $ and $\vec{\rho}^{s,v,w}$ be defined as $\vec{\rho}^{s,v,w}_i = \vec{y}^{s,v,w}_i + r_s^w $ for $i = 1, \dots, N-1$. In particular, we obtain 
\begin{equation}\label{MEst1}
\begin{split}
&  \mathbb{P}_{avoid}^{a_s^w, b_s^w, \vec{x}^{s,v,w},\vec{y}^{s,v,w}, \infty, \mathcal{L}^v_N[a_s^w, b_s^w]}\left(\mathcal{Q}_{N-1}( t_s) \leq x_s \right) \geq  \\
& \mathbb{P}_{avoid}^{a_s^w, b_s^w,\vec{\lambda}^{s,v,w},\vec{\rho}^{s,v,w}, \infty,\mathcal{L}^v_N[a_s^w, b_s^w]}\left(\mathcal{Q}_{N-1}( t_s) \leq x_s \right) \geq   \\
&  \mathbb{P}_{avoid}^{a_s^w, b_s^w,\vec{\lambda}^{s,v,w},\vec{\rho}^{s,v,w}, \infty,x_s - \epsilon_2}\left(\mathcal{Q}_{N-1}( t_s) \leq x_s \right) = \\
& \frac{ \mathbb{P}_{avoid}^{a_s^w, b_s^w,\vec{\lambda}^{s,v,w},\vec{\rho}^{s,v,w}, \infty, - \infty }\left(\mathcal{Q}_{N-1}( t_s) \leq x_s \mbox{ and } \inf_{x \in [a_s^w, b_s^w]}\mathcal{Q}_{N-1}(x) \geq x_s - \epsilon_2 \right)}{ \mathbb{P}_{avoid}^{a_s^w, b_s^w,\vec{\lambda}^{s,v,w},\vec{\rho}^{s,v,w}, \infty, - \infty }\left( \inf_{x \in [a_s^w, b_s^w]}\mathcal{Q}_{N-1}(x) \geq x_s - \epsilon_2 \right)} \\
&\geq 1 - \frac{ \mathbb{P}_{avoid}^{a_s^w, b_s^w,\vec{\lambda}^{s,v,w},\vec{\rho}^{s,v,w}, \infty, - \infty }\left(\mathcal{Q}_{N-1}( t_s) > x_s  \right)}{ \mathbb{P}_{avoid}^{a_s^w, b_s^w,\vec{\lambda}^{s,v,w},\vec{\rho}^{s,v,w}, \infty, - \infty }\left( \inf_{x \in [a_s^w, b_s^w]}\mathcal{Q}_{N-1}(x) \geq x_s - \epsilon_2 \right)}
\end{split}
\end{equation}
where in the first inequality we used Lemma \ref{MCLxy} and in the second one we used Lemma \ref{MCLfg}. The equality in going from the third to the fourth line uses Definition \ref{DefAvoidingLaw} twice. It follows from Lemma \ref{LemmaAvoidMax} applied to $a = a_s^w, b = b_s^w, \vec{x} = \vec{\lambda}^{s,v,w}, \vec{y} = \vec{\rho}^{s,v,w}, k = N-1$ and $r = r_{1,w} = \frac{3\epsilon_2 w^{1/2}}{10} - k + 1$ that if $w$ is sufficiently large (so that $r \geq 0$) we have 
\begin{equation}\label{MEst2}
\mathbb{P}_{avoid}^{a_s^w, b_s^w,\vec{\lambda}^{s,v,w},\vec{\rho}^{s,v,w}, \infty, - \infty }\left( \inf_{x \in [a_s^w, b_s^w]}\mathcal{Q}_{N-1}(x) \geq x_s - \epsilon_2 \right) \geq 1 - (1 - 2e^{-1})^{-N+1} \cdot e^{-4r_{1,w}^2}.
\end{equation}
In addition, it follows from Lemma \ref{LemmaBotMax} applied to $a = a_s^w, b = b_s^w, \vec{x} = \vec{\lambda}^{s,v,w}, \vec{y} = \vec{\rho}^{s,v,w}, k = N-1$ and $r = r_{2,w} = \frac{\sqrt{2} \epsilon_2 w^{1/2}}{5} $ that if $w$ is sufficiently large (so that $r \geq 0$) we have 
 \begin{equation}\label{MEst3}
\mathbb{P}_{avoid}^{a_s^w, b_s^w,\vec{\lambda}^{s,v,w},\vec{\rho}^{s,v,w}, \infty, - \infty }\left(\mathcal{Q}_{N-1}( t_s) > x_s  \right) \leq \frac{c_0 e^{-2r^2_{2,w}}}{\sqrt{2\pi} [1 + 2 r_{2,w}]}.
\end{equation}
Since $r_{1,w}$ and $r_{2,w}$ both converge to $\infty$ as $w \rightarrow \infty$, we see that (\ref{MEst0}), (\ref{MEst1}), (\ref{MEst2}) and (\ref{MEst3}) together imply (\ref{Red6Inf})  when $\omega \in E^c_v \cap F_v \cap A_v^c$ is such that $\mathcal{L}_{N-1}^v(t_s) < x_s - \epsilon_2/2$.\\

{\bf \raggedleft Step 7.} Suppose that $\omega \in E^c_v \cap F_v \cap A_v^c$ is such that $\mathcal{L}_{N-1}^v(t_s) > x_s + \epsilon_2/2$ and let $W_1 \geq W_0$ be sufficiently large so that $W_1^{-1} < \epsilon_1$. Then for $w \geq W_1$ we have
\begin{equation*}
\frac{ \mathbb{P}_{avoid}^{s,v,w}\left(\mathcal{Q}_{N-1}( t_s) \leq x_s \right)}{F(x_s; a^w_s,b^w_s,\vec{x}^{s,v,w},\vec{y}^{s,v,w})} \geq  \frac{\mathbb{P}_{avoid}^{a_s^w, b_s^w, \vec{x}^{s,v,w},\vec{y}^{s,v,w}, \infty,x_s - \epsilon_2}\left(\mathcal{Q}_{N-1}( t_s) \leq x_s \right)}{F(x_s; a^w_s,b^w_s,\vec{x}^{s,v,w},\vec{y}^{s,v,w})}, 
\end{equation*}
where we used that on $F_v$ the curve $ \mathcal{L}^v_N[a_s^w, b_s^w]$ is upper bounded by $x_s - \epsilon_2$ and  Lemma \ref{MCLfg}.
We next notice that 
\begin{align*}
&\mathbb{P}_{avoid}^{a_s^w, b_s^w, \vec{x}^{s,v,w},\vec{y}^{s,v,w}, \infty,x_s - \epsilon_2}\left(\mathcal{Q}_{N-1}( t_s) \leq x_s \right) \\
&=\frac{ \mathbb{P}_{avoid}^{a_s^w, b_s^w,\vec{x}^{s,v,w},\vec{y}^{s,v,w}, \infty, - \infty }\left(\mathcal{Q}_{N-1}( t_s) \leq x_s \mbox{ and } \inf_{x \in [a_s^w, b_s^w]}\mathcal{Q}_{N-1}(x) \geq x_s - \epsilon_2 \right)}{ \mathbb{P}_{avoid}^{a_s^w, b_s^w,\vec{x}^{s,v,w},\vec{y}^{s,v,w}, \infty, - \infty }\left( \inf_{x \in [a_s^w, b_s^w]}\mathcal{Q}_{N-1}(x) \geq x_s - \epsilon_2 \right)}.
\end{align*}
Combining the last two statements, and performing a change of variables we conclude that
\begin{equation}\label{MEst4}
\begin{split}
&\frac{ \mathbb{P}_{avoid}^{s,v,w}\left(\mathcal{Q}_{N-1}( t_s) \leq x_s \right)}{F(x_s; a^w_s,b^w_s,\vec{x}^{s,v,w},\vec{y}^{s,v,w})} \geq \frac{\tilde{\mathbb{P}}_{v,w}(\inf_{x \in [-1,1]}\tilde{\mathcal{Q}}_{N-1}(x) \geq -\epsilon_2 \cdot \sqrt{w} \mbox{ and }\tilde{\mathcal{Q}}_{N-1}(0) \leq 0 ) }{\tilde{\mathbb{P}}_{v,w}(\tilde{\mathcal{Q}}_{N-1}(0) \leq 0) \cdot \tilde{\mathbb{P}}_{v,w}(\inf_{x \in [-1,1]}\tilde{\mathcal{Q}}_{N-1}(x) \geq -\epsilon_2 \cdot \sqrt{w} ) } \\
&\geq \frac{\tilde{\mathbb{P}}_{v,w}(\inf_{x \in [-1,1]}\tilde{\mathcal{Q}}_{N-1}(x) \geq -\epsilon_2 \cdot \sqrt{w} \mbox{ and }\tilde{\mathcal{Q}}_{N-1}(0) \leq 0 ) }{\tilde{\mathbb{P}}_{v,w}(\tilde{\mathcal{Q}}_{N-1}(0) \leq 0)  } \\
&\geq 1 -  \frac{\tilde{\mathbb{P}}_{v,w}(\inf_{x \in [-1,1]}\tilde{\mathcal{Q}}_{N-1}(x) \leq -\epsilon_2 \cdot \sqrt{w})}{\tilde{\mathbb{P}}_{v,w}(\tilde{\mathcal{Q}}_{N-1}(0) \leq 0) }  ,
\end{split}
\end{equation}
where $\tilde{P}_{v,w} = \mathbb{P}_{avoid}^{-1,1,\vec{A}, \vec{B}, -\infty, \infty}$ with $\vec{A}_{v,w} = [\vec{x}^{s,v,w} - x_s \cdot \vec{1}] \cdot \sqrt{w}$ and $\vec{B}_{v,w}= [\vec{y}^{s,v,w} - x_s \cdot \vec{1}] \cdot \sqrt{w} $ and $\vec{1}$ is the vector in $\mathbb{R}^{N-1}$ with all entries equal to $1$. The change of variables we used above comes from Lemma \ref{LemmaAffine} applied to $r = x_s$, $u = t_s$, $c = \sqrt{w}$, $a = - 1$, $b = 1$, $\vec{x} =  \vec{A}_{v,w}$ and $\vec{y} =  \vec{B}_{v,w}$. 

Put $\vec{A}_{v,w} = (A^{v,w}_1, \dots, A^{v,w}_{N-1})$ and $\vec{B}_{v,w}= (B^{v,w}_1, \dots, B^{v,w}_{N-1})$. We also let $M_{v,w} = \max(A^{v,w}_{N-1}, B^{v,w}_{N-1})$ and $m_{v,w} = \min (A^{v,w}_{N-1}, B^{v,w}_{N-1}).$  Since $\omega \in E_v^c$ and by assumption $\mathcal{L}_{N-1}^v(t_s) > x_s + \epsilon_2/2$, we know that $m_{v,w} \geq \sqrt{w} \cdot (2\epsilon_2/5)$ and $M_{v,w} - m_{v,w} \leq \sqrt{w} (\epsilon_2/5)$.

It follows from Lemma \ref{LemmaAvoidMax} applied to $a = -1, b = 1, \vec{x} = \vec{A}_{v,w} , \vec{y} = \vec{B}_{v,w}, k = N-1$ and $r = r_{1,w} = \frac{\epsilon_2 \sqrt{w} + m_{v,w} }{2} - k + 1$ that if $w$ is sufficiently large (so that $r \geq 0$) we have 
\begin{equation}\label{MEst5}
\tilde{\mathbb{P}}_{v,w}\left(\inf_{x \in [-1,1]}\tilde{\mathcal{Q}}_{N-1}(x) \leq -\epsilon_2 \cdot \sqrt{w}\right)  \leq  (1 - 2e^{-1})^{-N+1} \cdot e^{-4r_{1,w}^2}.
\end{equation}
In addition, it follows from Lemma \ref{LemmaBotMin} applied to $a = -1, b = 1, \vec{x} = \vec{A}_{v,w} , \vec{y} = \vec{B}_{v,w}, k = N-1$ and $r = r_{2,w} = \frac{M_{v,w}}{\sqrt{2}} $ that 
 \begin{equation}\label{MEst6}
\tilde{\mathbb{P}}_{v,w}(\tilde{\mathcal{Q}}_{N-1}(0) \leq 0)  \geq  \frac{c_0 e^{-2r_{2,w}^2}}{\sqrt{2\pi} [1 + 2r_{2,w}]}.
\end{equation}
Combining (\ref{MEst4}), (\ref{MEst5}) and (\ref{MEst6}) we see that for all $w$ sufficiently large we have
\begin{equation}\label{MEst7}
\begin{split}
&\frac{ \mathbb{P}_{avoid}^{s,v,w}\left(\mathcal{Q}_{N-1}( t_s) \leq x_s \right)}{F(x_s; a^w_s,b^w_s,\vec{x}^{s,v,w},\vec{y}^{s,v,w})}  \geq 1 - [1- 2e^{-1}]^{-N+1}e^{-4r_{1,w}^2 + 2 r_{2,w}^2} \cdot \frac{\sqrt{2\pi} [1 + 2 r_{2,w}]}{c_0} \\
&= 1 - \frac{\sqrt{2\pi}[1- 2e^{-1}]^{-N+1}}{c_0} \cdot [1 + \sqrt{2} M_{v,w}]e^{-[\epsilon_2 \sqrt{w} + m_{v,w} - 2N + 4]^2 + M_{v,w}^2}  \\
&\geq  1 - \frac{\sqrt{2\pi}[1- 2e^{-1}]^{-N+1}}{c_0} \cdot [1 + \sqrt{2} M_{v,w}]e^{-[(4\epsilon_2/5) \sqrt{w} + M_{v,w} - 2N + 4]^2 + M_{v,w}^2}  \\
&\geq 1 -   [1 + \sqrt{2} M_{v,w}]e^{-[(\epsilon_2/2) \sqrt{w} + M_{v,w}]^2 + M_{v,w}^2} \geq 1 - [1 + \sqrt{2} M_{v,w}]e^{-\epsilon_2 M_{v,w}},
\end{split}
\end{equation}
where the first equality used the definition of $r_{1,w}$ and $r_{2,w}$; in  going from the second to the third line we used that $M_{v,w} - m_{v,w} \leq \sqrt{w} (\epsilon_2/5)$, and the inequalities at the end of the second and third lines hold for all large enough $w$. Since $M_{v,w} \geq \sqrt{w} (2\epsilon_2/5)$ we know it converges to infinity as $w \rightarrow \infty$ and so we see that (\ref{MEst7}) implies (\ref{Red6Inf}) when $\omega \in E^c_v \cap F_v \cap A_v^c$ is such that $\mathcal{L}_{N-1}^v(t_s) > x_s + \epsilon_2/2$.. This concludes the proof of (\ref{Red6Inf}) and hence the proposition.

%
%
\subsection{Proof of Corollary \ref{CorMain2}}\label{Section4.3} In this section we give the proof of Corollary \ref{CorMain2}. We will use the same notation as in the statement of the corollary and Section \ref{Section4.2} above.\\

The proof is by contradiction and we assume that for every
$k\in \mathbb{N}$,  $t_1 < t_2 < \cdots < t_k$ with $t_i \in \Lambda$ and $x_1, \dots, x_k \in \mathbb{R}$ we have
\begin{equation}\label{S43E1}
\mathbb{P}_1 \left( \mathcal{L}^1_1(t_1) \leq x_1, \dots,\mathcal{L}^1_1(t_k) \leq x_k  \right) = \mathbb{P}_2 \left( \mathcal{L}^2_1(t_1) \leq x_1, \dots,\mathcal{L}^2_1(t_k) \leq x_k  \right).
\end{equation}
We know that the projection of $\mathcal{L}^2$ to the top $N_1$ curves is a $\Sigma_1$-indexed line ensemble on $\Lambda$ that satisfies the partial Brownian Gibbs property, cf. Remark \ref{RPBGP}. By our assumption above we have that this line ensemble under $\mathbb{P}_2$ has the same top curve distribution as $\mathcal{L}^1$ under $\mathbb{P}_1$ and so by Theorem \ref{ThmMain} we conclude that for any $a < b$ with $a, b \in \Lambda$ we have that $\pi^{\llbracket 1, N_1 \rrbracket}_{[a,b]}(\mathcal{L}^1)$ under $\mathbb{P}_1$ has the same distribution as $\pi^{\llbracket 1, N_1 \rrbracket}_{[a,b]}(\mathcal{L}^2)$ under $\mathbb{P}_2$ as $\Sigma_1$-indexed line ensembles. This allows us to repeat the arguments in Step 2 of Section \ref{Section4.2} and we let $p_w$ be as in (\ref{Obs2}) for the case $k = 1$, $t_1 = (b+a)/2$, $ N- 1 =  N_1$, $S = \{1\}$, $x_1 \in \mathbb{R}$ and $W_0$ is sufficiently large so that $ a^w = t_1 - 1/w \in [a,b]$ and $b^w = t_1 + 1/w \in [a,b]$ for $w \geq W_0$. In particular we have
\begin{equation}\label{Obs2S43}
\begin{split}
p_w = \mathbb{E} \left[ \frac{ \mathbb{P}_{avoid}^{1,1,w}\left(\mathcal{Q}_{N_1}( t_1) \leq x_1 \right)}{F(x_1; a^w,b^w,\vec{x}^{1,1,w},\vec{y}^{1,1,w})}   \right]= \mathbb{E} \left[ \frac{ \mathbb{P}_{avoid}^{1,2,w}\left(\mathcal{Q}_{N_1}( t_1) \leq x_1 \right)}{F(x_1; a^w,b^w,\vec{x}^{s,2,w},\vec{y}^{1,2,w})}\right],
\end{split}
\end{equation}
where in the left expectation $\mathcal{L}_{N_1 + 1}^1[ a^w, b^w] = - \infty$ (here we used that $\mathcal{L}^1$ satisfies the Brownian Gibbs rather than the partial Brownian Gibbs property). In particular, we have by definition 
$$ \mathbb{P}_{avoid}^{1,1,w}\left(\mathcal{Q}_{N_1}( t_1) \leq x_1 \right) = F(x_1; a^w,b^w,\vec{x}^{1,1,w},\vec{y}^{1,1,w})$$
and so $p_w = 1$. On the other hand, by repeating the arguments in Step 3 of Section \ref{Section4.2} we have the second line of (\ref{Red1}), namely that 
$$\limsup_{w \rightarrow \infty} p_w \leq \mathbb{P}_2 \left( \mathcal{L}^2_{N_1+1}(t_1) \leq x_1 \right).$$
This shows that $ \mathbb{P}_2 \left( \mathcal{L}^2_{N_1+1}(t_1) \leq x_1 \right) = 1$ for all $x_1 \in \mathbb{R}$, which is our desired contradiction. Hence (\ref{S43E1}) cannot hold for every
$k\in \mathbb{N}$,  $t_1 < t_2 < \cdots < t_k$ with $t_i \in \Lambda$ and $x_1, \dots, x_k \in \mathbb{R}$, which is what we wanted to prove.

%% file: Section5.tex
\section{Appendix}\label{Section5}
In this section we prove the three lemmas stated in Section \ref{Section2.3}. Our approach goes through proving analogous results for non-intersecting symmetric random walks and taking scaling limits. We first isolate some preliminary results in Section \ref{Section5.1}. The proof of Lemma \ref{AvoidIsGibbs} is given in Section~\ref{Section5.2} and the ones for Lemmas \ref{MCLxy} and \ref{MCLfg} are given in Section \ref{Section5.3}.

%
\subsection{Preliminaries}\label{Section5.1} Let $X_i$ be i.i.d. random variables such that $\mathbb{P}(X_1 = - 1) = \mathbb{P}(X_1 = 0) = \mathbb{P}(X_1 = 1) = 1/3$. In addition, we let $S_N = X_1 + \cdots + X_N$ and for $z \in \llbracket -N, N\rrbracket$ we let $S^{(N,z)} = \{S_m^{(N,z)}\}_{m = 0}^N$ denote the process $\{S_m \}_{m = 0}^N$ with law conditioned so that $S_N = z$. We extend the definition of $S^{(N,z)}_t$ to non-integer values of $t$ by linear interpolation. 

We have the following theorem, which is a special case of \cite[Theorem 2.6]{DW} when $p = 0$.
\begin{theorem}\label{DWThm} There exist constants $0 < C,c,\alpha < \infty$ such that for every positive integer $N$ there is a probability space on which are defined a standard Brownian bridge $\tilde{B}(t)$ and a family of processes $S^{(N,z)}$ for $z \in \llbracket -N, N \rrbracket$ such that 
\begin{equation*}
\mathbb{E} \left[ e^{c \Delta(N,z)} \right] \leq C e^{\alpha (\log N)} e^{z^2/N},
\end{equation*}
where $\Delta(n,z) = \sup_{0 \leq t \leq N} \left| \sqrt{2N/3} \cdot \tilde{B}(t/N) + \frac{t}{N}z - S_t^{(N,z)} \right|$.
\end{theorem}

We summarize some useful notation in the following definition.

\begin{definition}\label{Grids}
Fix $a,b \in \mathbb{R}$ with $b > a$ and a scaling parameter $n \in \N$. With the latter data we define two quantities $\Delta_n^t = (b-a) / n^2$ and $\Delta_n^x = \sqrt{3 \Delta^t_n / 2}$. Furthermore, we introduce two grids $\R_{n} = (\Delta_n^x) \cdot \Z$ and $\Lambda_{n^2} = \{a + m \cdot \Delta_n^t : m \in \mathbb{Z}\}$. Given $u, v \in \Lambda_{n^2}$ with $u < v$ and $x,y \in \mathbb{R}_n$ with $|x - y| \leq \frac{\Delta_n^x}{\Delta_n^t} \cdot (v -u)$, we define the $C([u,v])$-valued random variable
$$Y(t) = x + \Delta_n^x \cdot S^{((v -u)/\Delta_n^t,  (y-x)/ \Delta_n^x)}_{(t -u)/\Delta_n^t}  \mbox{ for $t \in [u,v]$}.$$
As defined $Y(t)$ is a continuous function on $[u,v]$ such that $Y(u) = x$ and $Y(v) = y$. We denote the law of $Y$ by $\mathbb{P}^{u,v,x,y}_{free, n}$.
\end{definition}

The following result roughly states that the laws $\mathbb{P}^{u,v,x,y}_{free, n}$ weakly converge to the law of a Brownian bridge as $n \rightarrow \infty$  if the quantities $u,v,x,y$ converge.
\begin{lemma}\label{ConvToBridge} Let $x,y, a', b' \in \mathbb{R}$ with $a' < b'$. In addition, let $a < b$ and for $n \in \mathbb{N}$ let $x_n, y_n \in \mathbb{R}_n$ and $a_n, b_n \in \Lambda_{n^2}$ with $a_n \leq a'$, $b_n \geq b'$ and $|x_n - y_n| \leq \frac{\Delta_n^x}{\Delta_n^t} \cdot [b_n -a_n]$ (here we used the notation from Definition \ref{Grids}). Suppose $a_n \rightarrow a'$, $b_n \rightarrow b'$, $x_n \rightarrow x$ and $y_n \rightarrow y$ as $n \rightarrow \infty$. Let $Y^n$ be a random variable with law $\mathbb{P}_{free,n}^{a_n,b_n,x_n,y_n}$ and let $Z^n$ be a $C([a',b'])$-valued random variable, defined through $Z^n(t)= Y^n(t)$ for $t \in [a',b']$. Then the law of $Z^n$ converges weakly to $\mathbb{P}_{free}^{a',b', x,y}$ as $n \rightarrow \infty$.
\end{lemma}
\begin{proof} Let $z_n = [\Delta_n^x]^{-1} \cdot (y_n - x_n)$ and note that $z_n \in \llbracket -N, N \rrbracket$, where $N = \frac{[b_n - a_n] }{\Delta_n^t}.$ Let $\tilde{B}$ be a standard Brownian bridge and define random $C([a',b'])$-valued random variables $B^n$ and $B$ through
\begin{align*}
B^n(t) &= \sqrt{b_n-a_n} \cdot  \tilde{B}\left( \frac{t - a_n}{b_n-a_n} \right) + \frac{t -a_n}{b_n-a_n} \cdot y_n + \frac{b_n-t}{b_n-a_n} \cdot x_n,\\
{B}(t) &=  \sqrt{b-a} \cdot \tilde{B}\left( \frac{t - a}{b-a} \right) + \frac{t -a}{b-a} \cdot y + \frac{b-t}{b-a} \cdot x.
\end{align*}
Clearly, $B$ has law $\mathbb{P}_{free}^{a,b, x,y}$ and $B^n \implies {B}$ as $n \rightarrow \infty$. It follows from \cite[Theorem 3.1]{Billing}, that to conclude that $Z^n \implies B$ as $n \rightarrow \infty$ it suffices to show that we can construct a sequence of probability spaces that support $Y^n, B^n$ so that
\begin{equation}\label{Suff1}
\rho( Y^n, B^n) \implies 0 \mbox{ as $n \rightarrow \infty$, where $\rho(f,g) = \sup_{x \in [a_n,b_n]} |f(x) - g(x)|$.} 
\end{equation}

From Theorem \ref{DWThm} we can construct a probability space that supports $Y^n$ and $B^n$ (for each fixed $n \in \mathbb{N}$) such that 
\begin{equation}\label{DWeqv2}
\mathbb{E} \left[ e^{c \tilde{\Delta}(N,x_n,y_n)} \right] \leq C e^{\alpha (\log N)} e^{z_n^2/N},
\end{equation}
where
$$ \tilde{\Delta}(N,x_n,y_n)=  [\Delta_n^x]^{-1} \cdot  \rho (B^n , Y^n).$$
Let $\epsilon > 0$ be given. Since $y_n - x_n \rightarrow y-x$ as $n \rightarrow \infty$, we know that we can find $N_1 \in \mathbb{N}$ and $C_1 > 0$ such that if $n \geq N_1$ we have $|z_n| \leq C_1 \cdot \sqrt{N}$.  Using (\ref{DWeqv2}) and Chebyshev's inequality we see that for $n \geq N_1$ we have
$$\mathbb{P}( \rho(B^n, Y^n) > \epsilon) \leq e^{-c \epsilon [\Delta_n^x]^{-1} } \cdot  C e^{\alpha (\log N)} e^{C_1^2},$$
which converges to $0$ as $n\rightarrow \infty$. The latter implies (\ref{Suff1}) and concludes the proof of the lemma.
\end{proof}

We next introduce the multi-line generalization of Definition \ref{Grids}.
\begin{definition}\label{Grids2} Continue with the same notation as in Definition \ref{Grids} and fix $k \in \mathbb{N}$. Suppose that $S^{(N,z), i}$ for $i = 1,\dots, k$ and $z \in \llbracket -N, N \rrbracket$ are $k$ independent processes with the same law as $S^{(N,z)}$. In addition, let $\vec{x}, \vec{y} \in \R_n^k$ be such that $|x_i  -y_i| \leq \frac{\Delta_n^x}{\Delta_n^t}\cdot [v-u]$. With this data we define the $\llbracket 1, k \rrbracket$-indexed line ensemble on $[u,v]$ through
\begin{equation}\label{DiscreteLE}
\mathcal{Y}^n(i,t) = x_i + \Delta_n^x \cdot S^{((v -u)/\Delta_n^t,  (y_i-x_i)/ \Delta_n^x), i}_{(t -u)/\Delta_n^t}  \mbox{ for $t \in [u,v]$ and $i \in \llbracket 1, k \rrbracket$}.
\end{equation}
We call the law of the resulting $\llbracket 1, k \rrbracket$-indexed line ensemble $\mathbb{P}_{free, n}^{a,b,\vec x,\vec y}$. 

Suppose that $\vec{x}, \vec{y} \in \R_n^k \cap \weyl_k$. By analogy with Definition \ref{DefAvoidingLaw}, given continuous functions $f: [u,v] \rightarrow (-\infty, \infty]$ and $g:[u,v] \rightarrow [-\infty,  \infty)$, we define the probability measure $\mathbb{P}_{avoid, n}^{u,v, \vec{x}, \vec{y}, f, g}$ to be the distribution of $\mathcal{Y}^n$ from (\ref{DiscreteLE}), conditioned on the event
$$E_n = \{ f(r) > \mathcal{Y}^n(1,r) > \mathcal{Y}^n(2,r) > \cdots > \mathcal{Y}^n(k,r) > g(r) \mbox{ for $r \in [u,v]$} \}.$$
This measure is well-defined if the set of trajectories satisfying the latter conditions is non-empty. 
\end{definition}

We need the following convergence result for non-intersecting random walk bridges. 
\begin{lemma}\label{lem:RW}
Fix $k \in \N$ and $a,b \in \mathbb{R}$ with $a < b$ and assume the same notation as in Definition \ref{Grids2}. Suppose that $f : [a,b] \to (-\infty, +\infty], g : [a,b] \to [-\infty, +\infty)$ are continuous functions such that $f(t) > g(t)$ for $t \in [a,b]$. Let $a', b' \in [a,b]$ be such that $a' < b'$ and suppose that $\vec{x}, \vec{y} \in \weyl_k$ are such that $f(a') > x_1, f(b') > y_1$, $g(a') < x_k$, $g(b') < y_k$. Suppose further that $\vec{x}^n, \vec{y}^n \in \weyl_k \cap\R_{n}^k$ are such that $\lim_{n \to \infty} \vec{x}^n = \vec x$, $\lim_{n \to \infty} \vec{y}^n = \vec y$; and $f_n: [a,b] \to (-\infty, +\infty], g_n : [a,b] \to [-\infty, +\infty)$ are sequences of continuous functions such that $f_n \rightarrow f$ and $g_n \rightarrow g$ uniformly as $n \rightarrow \infty$ on $[a,b]$ (if $f = \infty$ the latter means that $f_n = \infty$ for all large enough $n$ and similarly if $g = -\infty$ the latter means $g_n = -\infty$ for large enough $n$). Finally, suppose that $a_n,b_n \in \Lambda_{n^2}$ are such that $a_n \leq a'$, $b_n \geq b'$ and $a_n$ is maximal while $b_n$ is minimal subject to these conditions. 
Then there exists $N_0 \in \mathbb{N}$ such that $\mathbb{P}_{avoid, n}^{a_n,b_n, \vec{x}^n, \vec{y}^n, f^n, g^n}$ are well-defined for $n \geq N_0$. Moreover, if $\mathcal{Y}^n$ are $\llbracket 1, k \rrbracket$-indexed line ensembles with laws $\mathbb{P}_{avoid, n}^{a_n,b_n, \vec{x}^n, \vec{y}^n, f^n, g^n}$ and $\mathcal{Z}^n$ are the $\llbracket 1, k \rrbracket$-indexed line ensembles on $[a',b']$ defined through
\begin{equation}\label{ProjWalk}
\mathcal{Z}^n(i,t) = \mathcal{Y}^n(i,t) \mbox{ for $n \geq N_0$, $i \in \llbracket 1, k \rrbracket$, $t \in [a',b']$}
\end{equation}
then $\mathcal{Z}^n$ converge weakly to $\mathbb{P}_{avoid}^{a',b', \vec{x}, \vec{y}, f, g}$ as $n \to \infty$.
\end{lemma}
\begin{proof}
Observe that we can find $\epsilon > 0$ and continuous functions $h_1, \dots, h_k :[a',b'] \rightarrow \mathbb{R}$ (all depending on $\vec{x}, \vec{y}, f, g, a',b'$)  such that $h_i(a') = x_i$, $h_i(b') =y_i$ for $i  =1 ,\dots, k$, such that the following holds. If $u_i : [a',b'] \rightarrow \mathbb{R}$ are continuous and $\rho(u_i, h_i) =\sup_{x \in [a',b']} |u_i(x) - h_i(x)| < \epsilon$ then 
$$f(x) - \epsilon > u_1(x) + \epsilon > u_1(x) - \epsilon > u_2(x) + \epsilon > \cdots > u_k(x) + \epsilon > u_k(x) - \epsilon > g(x) \mbox{ for all $x \in [a',b']$}.$$
Observe that by Lemma \ref{Spread} we know that 
$$\mathbb{P}_{free}^{a',b',\vec{x}, \vec{y}} \left( \rho(\mathcal{Q}_i, h_i) < \epsilon \mbox{ for all $i = 1, \dots, k$}  \right) > 0,$$
where $(\mathcal{Q}_1, \dots, \mathcal{Q}_k)$ above are the random curves that are $\mathbb{P}_{free}^{a',b',\vec{x}, \vec{y}}$-distributed.

Since $\vec{x}^n - \vec{y}^n \rightarrow \vec{x} - \vec{y}$ we know that there exists $N_1 \in \mathbb{N}$ such that if $n \geq N_1$ we have $|x^n_i - y^n_i| \leq \frac{\Delta_n^x}{\Delta_n^t} \cdot [b_n -a_n]$. For $n \geq N_1$ we know from Lemma \ref{ConvToBridge} that if $\mathcal{Y}^n$ has law $\mathbb{P}_{free, n}^{a_n,b_n,\vec x^n,\vec y^n}$ and $\mathcal{Z}^n$ is as in (\ref{ProjWalk}) then $\mathcal{Z}^n$ converges weakly to $\mathbb{P}_{free}^{a,b,\vec{x}, \vec{y}}$ as $n \rightarrow \infty$. Consequently, there exists $N_2$ such that for $n \geq \max(N_1, N_2)$ we have
$$\mathbb{P}_{free,n}^{a_n,b_n,\vec{x}^n, \vec{y}^n} \left( \rho(\mathcal{Z}^n_i, h_i) < \epsilon \mbox{ for all $i = 1, \dots, k$}  \right) > 0.$$
Suppose further that $N_3$ is sufficiently large so that $\sup_{x \in [a, b]}|f_n(x) - f(x)| < \epsilon/4$ and $\sup_{x \in [a, b]}|g_n(x) - g(x)| < \epsilon/4$. If $f = \infty$ or $g = -\infty$ (or both) we choose $N_3$ sufficiently large so that $f_n = \infty$ or $g_n = -\infty$ (or both). We also let $N_4$ be sufficiently large so that if $n \geq N_4$ and $|x-y| \leq\Delta^t_n$ then $|f(x) - f(y)| < \epsilon/4$ and $|g(x) - g(y)| < \epsilon/4$ (if $f = \infty$ we ignore the first condition and if $g = -\infty$ we ignore the second condition). Finally, we let $N_5$ be sufficiently large so that $n \geq N_5$ implies $\Delta^x_n < \epsilon/4$. Overall, if $n \geq N_0 = \max(N_1, N_2, N_3,N_4, N_5)$ we see that 
\begin{align*}
\{ f_n(r) > \mathcal{Y}^n(1,r) &> \mathcal{Y}^n(2,r) > \cdots > \mathcal{Y}^n(k,r) > g_n(r) \mbox{ for $r \in [a_n,b_n]$} \} \\
&\subset \{ \rho(\mathcal{Z}^n_i, h_i) < \epsilon \mbox{ for all $i = 1, \dots, k$} \}.
\end{align*}
The above implies that $\mathbb{P}_{avoid, n}^{a_n,b_n, \vec{x}^n, \vec{y}^n, f^n, g^n}$ is well-defined as long as $n \geq N_0$, which proves the first part of the lemma.\\

Let $\Lambda' = [a',b']$ and $\Sigma = \llbracket 1, k\rrbracket$, we need to show that for any bounded continuous function $F : C(\Sigma \times \Lambda') \to \R$ we have
\begin{equation}\label{RWLimit}
\lim_{n \to \infty} \E\left[F(\mathcal{Z}^{n})\right] = \E\left[F(\mathcal{Q})\right],
\end{equation}
where $\mathcal{Q}$ is a $\Sigma$-indexed line ensembles, whose distribution is $\mathbb{P}_{avoid}^{a',b', \vec{x}, \vec{y}, f, g}$. 

We define the functions $H_{f,g} : C(\Sigma \times \Lambda') \to \R$ and $H_{f,g}^n: C(\Sigma \times [a_n, b_n]) \to \R$ as
\begin{equation*}
\begin{split}
&H_{f,g}(\CL) = {\bf 1} \{ f(r) > \CL_1(r) > \CL_2(r) > \cdots > \CL_k(r) > g(r) \mbox{ for $r \in [a',b']$}\},\\
&H^n_{f,g}(\CL) = {\bf 1} \{ f(r) > \CL_1(r) > \CL_2(r) > \cdots > \CL_k(r) > g(r) \mbox{ for $r \in [a_n,b_n]$}\}.
\end{split}
\end{equation*}
Using these functions, we can write for $n \geq N_0$
\begin{equation}\label{RW}
\E[F( \mathcal{Z}^n)] = \frac{\E\left[F(\pi_{[a',b']}(\CL^{n})) H_{f_n, g_n}(\CL^{n}) \right]}{\E\left[H_{f_n, g_n}(\CL^{n}) \right]},
\end{equation}
where $\CL^{n}$ is a line ensemble of independent random walk bridges with distribution $\mathbb{P}_{free, n}^{a_n,b_n,\vec x^n,\vec y^n}$. Also if $\mathcal{L} \in C(\Sigma \times [a_n, b_n]) $ we define $\pi_{[a',b']}(\mathcal{L})$ to be the element in $C(\Sigma \times [a', b']) $ defined through
$$ \pi_{[a',b']}(\mathcal{L})(i,x) = \mathcal{L}(i,x)  \mbox{ for $i = 1,\dots, k$ and $x \in [a',b']$}.$$
 We remark that the choice of $N_0$ makes the denominator in (\ref{RW}) strictly positive.

By Definition \ref{DefAvoidingLaw} we also have
\begin{equation}\label{BB}
\E[F(\mathcal{Q})] = \frac{\E\left[F(\CL) H_{f, g}(\CL) \right]}{\E\left[H_{f, g}(\CL) \right]},
\end{equation}
where $\CL$ is a line ensemble of independent random walk bridges with distribution $\mathbb{P}_{free}^{a,b,\vec x,\vec y}$. In view of (\ref{RW}) and (\ref{BB}) we see that to prove (\ref{RWLimit}) it suffices to prove that for any bounded continuous function $F : C(\Sigma \times \Lambda') \to \R$ we have
\begin{equation}\label{RWLimitRed}
\lim_{n \to \infty} \E\left[F(\pi_{[a',b']}(\CL^{n}))H_{f^n, g^n}(\CL^{n})  \right] = \E\left[F(\CL)  H_{f, g}(\CL)\right].
\end{equation}

By Lemma \ref{ConvToBridge} we know that $\pi_{[a',b']}(\mathcal{L}^n) \implies \mathcal{L}$ as $n \rightarrow \infty.$ In addition, using that $C([a,b])$ with the uniform topology is separable, see e.g. \cite[Example 1.3, pp 11]{Billing}, we know that $C(\Sigma \times \Lambda') $ is also separable. In particular we can apply the Skorohod Representation Theorem, see \cite[Theorem 6.7]{Billing}, from which we conclude that there exists a probability space $(\Omega, \mathcal{F}, \mathbb{P})$, which supports $C(\Sigma \times [a_n,b_n])$-valued random variables $\mathcal{L}^n$ and a $C(\Sigma \times \Lambda')$-valued random variable $\mathcal{L}$ such that $\pi_{[a',b']}(\mathcal{L}^n) \rightarrow \mathcal{L}$ for every $\omega \in \Omega$ and such that under $\mathbb{P}$ the law of $\mathcal{L}^n$ is $\mathbb{P}_{free, n}^{a_n,b_n,\vec x^n,\vec y^n}$, while under $\mathbb{P}$ the law of $\mathcal{L}$ is $\mathbb{P}_{free}^{a',b',\vec x,\vec y}$. Here we implicitly used the maximality of $a_n$ and the minimality of $b_n$, which imply that $\mathcal{L}^n$ is completely determined from $ \pi_{[a',b']}(\mathcal{L}^n) $.

It follows from the continuity of $F$ that on the event 
$$E_1 = \{\omega: f(r) > \CL_1(\omega) (r) > \CL_2(\omega) (r) > \cdots > \CL_k(\omega)(r) > g(r) \mbox{ for $ r \in [a',b']$} \}$$
we have $F(\pi_{[a',b']}(\mathcal{L}^n))H_{f^n, g^n}(\CL^{n})  \rightarrow F(\CL)$. In addition, on the event
$$E_2 = \{\omega:  \CL_i(\omega) (r) < \CL_{i+1}(\omega) (r) \mbox{ for some $i \in \llbracket 0, k \rrbracket$ and $r \in [a',b']$ with $\mathcal{L}_0 = f$, $\mathcal{L}_{k+1} = g$} \}$$
we have that $F(\pi_{[a',b']}(\mathcal{L}^n))H_{f^n, g^n}(\CL^{n})  \rightarrow 0$. By Lemma \ref{NoTouch} we know that $\mathbb{P}(E_1 \cup E_2) = 1$ and so $\mathbb{P}$-almost surely we have $F(\pi_{[a',b']}(\mathcal{L}^n))H_{f^n, g^n}(\CL^{n}) \rightarrow F(\CL)  H_{f, g}(\CL)$. By the bounded convergence theorem, we conclude (\ref{RWLimitRed}), which finishes the proof of the lemma.
\end{proof}
 
%
\subsection{Proof of Lemma \ref{AvoidIsGibbs} }\label{Section5.2}
We assume the same notation as in Lemma \ref{AvoidIsGibbs} and Definition \ref{DefBGP}. Put $\Lambda = [a,b]$ and $\Sigma = \llbracket 1,k \rrbracket$. We fix a set $K = \{k_1, k_1 + 1, \dots, k_2 \} \subset \llbracket 1, k\rrbracket$ and $a', b' \in [a,b]$ with $a' < b'$. Furthermore, we take a  bounded Borel-measurable function $F: C(K \times [a,b]) \rightarrow \mathbb{R}$. Our goal in this section is to prove that $\mathbb{P}$-almost surely
\begin{equation}\label{BGPTowerS5}
\mathbb{E} \left[ F\left(\mathcal{L}|_{K \times [a',b']} \right)  {\big \vert} \mathcal{F}_{ext} (K \times (a',b'))  \right] =\mathbb{E}_{avoid}^{a',b', \vec{x}, \vec{y}, f, g} \bigl[ F(\tilde{\mathcal{Q}}) \bigr],
\end{equation}
where
$$\mathcal{F}_{ext} (K \times (a',b')) = \sigma \left \{ \mathcal{L}_i(s): (i,s) \in D_{K,a',b'}^c \right\}$$
is the $\sigma$-algebra generated by the variables in the brackets above, $ \mathcal{L}|_{K \times [a',b']}$ denotes the restriction of $\mathcal{L}$ to the set $K \times [a',b']$, $\vec{x} = (\mathcal{L}_{k_1}(a'), \dots, \mathcal{L}_{k_2}(a'))$, $\vec{y} = (\mathcal{L}_{k_1}(b'), \dots, \mathcal{L}_{k_2}(b'))$, $f = \mathcal{L}_{k_1 - 1}[a',b']$ with the convention that $f = \infty$ if $k_1 - 1 \not \in \Sigma$, $g = \mathcal{L}_{k_2 +1}[a',b']$ with the convention that $g = -\infty$ if $k_2+1 \not \in \Sigma$. 

We split the proof of (\ref{BGPTowerS5}) in two steps for the sake of clarity.\\

{\bf \raggedleft Step 1.} Let $m \in \mathbb{N}$, $n_1, \dots, n_m \in \Sigma$, $t_1, \dots, t_m \in [a,b]$ and $f_1, \dots, f_m : \mathbb{R} \rightarrow \mathbb{R}$ be bounded continuous functions. We let $S = \{i \in \llbracket 1, m\rrbracket: n_i \in K \mbox{ and } t_i \in [a',b']\}$. We claim that 
\begin{equation}\label{S51Red1}
\mathbb{E} \left[ \prod_{i = 1}^m f_i(\mathcal{L}(n_i, t_i)) \right] = \mathbb{E} \left[ \prod_{s \in S^c} f_s(\mathcal{L}(n_s, t_s)) \cdot\mathbb{E}_{avoid}^{a',b', \vec{x}, \vec{y}, f, g} \left[ \prod_{s \in S} f_s(\tilde{\mathcal{Q}}(n_s, t_s)) \right] \right]. 
\end{equation}
We show (\ref{S51Red1}) in the step below. Here we assume its validity and conclude the proof of the lemma.\\

Define the functions 
$$h_n(x;r) = \begin{cases} 0, &\mbox{ if $x > r + n^{-1}$,} \\ 1 - n (x-r), &\mbox{ if $x \in [r, r+n^{-1}]$, } \\ 1, &\mbox{ if $x < r$.} \end{cases}$$
Let us fix $m_1, m_2 \in \mathbb{N}$, $n^1_1, \dots, n^1_{m_1}, n^2_1, \dots, n^2_{m_2} \in \Sigma$, $t^1_1, \dots, t^1_{m_1}, t^2_1, \dots, t^2_{m_2} \in [a,b]$ so that $(n^1_i, t^1_i) \not \in K \times [a', b']$ for $i = 1, \dots, m_1$ and $(n^2_i, t^2_i) \in K \times [a', b']$ for $i = 1, \dots, m_2$. It follows from (\ref{S51Red1}) that
\begin{equation*}
\begin{split}
&\mathbb{E} \hspace{-1mm} \left[ \prod_{i = 1}^{m_1} h_n(\mathcal{L}(n^1_i, t^1_i); a_i)  \prod_{i = 1}^{m_2} h_n(\mathcal{L}(n^2_i, t^2_i); b_i)  \hspace{-0.5mm} \right] \hspace{-0.5mm}= \hspace{-0.5mm} \mathbb{E} \left[ \hspace{-0.5mm} \prod_{i = 1}^{m_1} h_n(\mathcal{L}(n^1_i, t^1_i); a_i)  \mathbb{E}_{avoid}^{a',b', \vec{x}, \vec{y}, f, g} \hspace{-1.5mm} \left[ \hspace{-0.5mm}\prod_{i = 1}^{m_2} h_n(\tilde{\mathcal{Q}}(n^2_i, t^2_i); b_i)  \hspace{-0.5mm} \right]  \hspace{-0.5mm}\right] \hspace{-0.5mm}. 
\end{split}
\end{equation*}
Taking the limit as $n \rightarrow \infty$ we conclude by the bounded convergence theorem that
\begin{equation*}
\begin{split}
&\mathbb{E} \hspace{-1mm} \left[ \prod_{i = 1}^{m_1} \hat{h}(\mathcal{L}(n^1_i, t^1_i); a_i)  \prod_{i = 1}^{m_2} \hat{h}(\mathcal{L}(n^2_i, t^2_i); b_i)  \hspace{-0.5mm} \right] \hspace{-0.5mm}= \hspace{-0.5mm} \mathbb{E} \left[ \hspace{-0.5mm} \prod_{i = 1}^{m_1}\hat{h}(\mathcal{L}(n^1_i, t^1_i); a_i)  \mathbb{E}_{avoid}^{a',b', \vec{x}, \vec{y}, f, g} \hspace{0mm} \left[ \hspace{-0.5mm}\prod_{i = 1}^{m_2} \hat{h}(\tilde{\mathcal{Q}}(n^2_i, t^2_i); b_i)  \hspace{-0.5mm} \right]  \hspace{-0.5mm}\right] \hspace{-0.5mm}. 
\end{split}
\end{equation*}
where $\hat{h}(x;a) = {\bf 1}\{ x \leq a\}$. Let $\mathcal{H}$ denote the space of bounded Borel-measurable functions $H: C(K \times [a,b]) \rightarrow \mathbb{R}$ such that 
\begin{equation}\label{S51H}
\begin{split}
&\mathbb{E} \hspace{-1mm} \left[ \prod_{i = 1}^{m_1} \hat{h}(\mathcal{L}(n^1_i, t^1_i); a_i)  H\left(\mathcal{L}|_{K \times [a',b']}\right)   \right] \hspace{-0.5mm}= \hspace{-0.5mm} \mathbb{E} \left[  \prod_{i = 1}^{m_1}\hat{h}(\mathcal{L}(n^1_i, t^1_i); a_i)  \mathbb{E}_{avoid}^{a',b', \vec{x}, \vec{y}, f, g} \hspace{0mm} \left[ \hspace{-0.5mm} H(\tilde{\mathcal{Q}})  \right]  \right] \hspace{-0.5mm}. 
\end{split}
\end{equation}
Our work so far shows that ${\bf 1}_{A} \in \mathcal{H}$ for any set $A \in \mathcal{A}$, where $\mathcal{A}$ is the $\pi$-system of sets of the form
$$ \{ h \in C( K \times [a',b']) : h(n^2_i, t^2_i) \leq b_i \mbox{ for $i = 1, \dots, m_2$}\}.$$
It is clear that $\mathcal{H}$ is closed under linear combinations (by linearity of the expectation). Furthermore, if $H_n \in \mathcal{H}$ is an increasing sequence of non-negative measurable functions that increase to a bounded function $H$ then $H \in \mathcal{H}$ by the monotone convergence theorem. By the monotone class theorem, see e.g. \cite[Theorem 5.2.2]{Durrett}, we have that $\mathcal{H}$ contains all bounded measurable functions with respect to $\sigma(\mathcal{A})$, and the latter is $\mathcal{C}_{K}$ in view of Lemma \ref{FinDim}. In particular $F \in \mathcal{H}$. \\

Let $\mathcal{B}$ denote the collection of sets $B \in \mathcal{F}_{ext} (K \times (a',b'))$ such that 
\begin{equation}\label{S51B}
\begin{split}
&\mathbb{E} \hspace{-1mm} \left[ {\bf 1}_B \cdot  F\left(\mathcal{L}|_{K \times [a',b']}\right)   \right] \hspace{-0.5mm}= \hspace{-0.5mm} \mathbb{E} \left[  {\bf 1}_B \cdot \mathbb{E}_{avoid}^{a',b', \vec{x}, \vec{y}, f, g} \hspace{0mm} \left[ \hspace{-0.5mm} F(\tilde{\mathcal{Q}})  \right]  \right] \hspace{-0.5mm}. 
\end{split}
\end{equation}
The bounded convergence theorem implies that $\mathcal{B}$ is a $\lambda$-system, and (\ref{S51H}) being true for all bounded $\mathcal{C}_{K}$-measurable functions $H$ implies that $\mathcal{B}$ contains the $\pi$-system $P$ of sets of the form
$$ \{ h \in C( \Sigma \times [a,b]) : h(n_i, t_i) \leq a_i \mbox{ for $i = 1, \dots, m$, where $(n_i, t_i)  \in D_{K,a',b'}^c$}\}.$$
By the $\pi-\lambda$ Theorem, see \cite[Theorem 2.1.6]{Durrett}, we see that $\mathcal{B}$ contains $\sigma(P)$, which is precisely $ \mathcal{F}_{ext} (K \times (a',b'))$. We conclude that (\ref{S51B}) holds for all $B \in \mathcal{F}_{ext} (K \times (a',b'))$. Since by Lemma \ref{LemmaMeasExp} we know that $ \mathbb{E}_{avoid}^{a',b', \vec{x}, \vec{y}, f, g} \hspace{0mm} \bigl[ F(\tilde{\mathcal{Q}})  \bigr] $ is $\mathcal{F}_{ext} (K \times (a',b'))$-measurable, we conclude (\ref{BGPTowerS5}) by the defining properties of conditional expectations.\\

{\bf \raggedleft Step 2.} In this step we prove (\ref{S51Red1}). Following the notation from Definitions \ref{Grids} and \ref{Grids2} we let $\vec{x}^n, \vec{y}^n \in \mathbb{R}^k_n \cap \weyl_k$ be such that $|x^n_i - y_i^n| \leq \frac{\Delta_n^x}{\Delta_n^t} [b-a]$ and $\vec{x}^n \rightarrow \vec{x}$ and $\vec{y}^n \rightarrow \vec{y}$. It follows from Lemma~\ref{lem:RW} applied to $a' = a $, $b' = b$, $f = f_n = \infty$, $g = g_n = -\infty$ that the $\llbracket 1, k \rrbracket$-indexed line ensembles $\mathcal{Y}^n$ whose laws are $\mathbb{P}_{avoid, n}^{a,b, \vec{x}^n, \vec{y}^n, \infty, -\infty}$ converge weakly to $\mathbb{P}_{avoid}^{a,b, \vec{x}, \vec{y}, \infty, -\infty}$ as $n \to \infty$. In particular, we conclude that 
\begin{equation}\label{S51Red2}
\mathbb{E} \left[ \prod_{i = 1}^m f_i(\mathcal{L}(n_i, t_i)) \right] = \lim_{n \rightarrow \infty} \mathbb{E} \left[\prod_{i = 1}^m f_i(\mathcal{Y}^n(n_i, t_i)) \right]. 
\end{equation}
Using that $C([a,b])$ with the uniform topology is separable, see e.g. \cite[Example 1.3, p. 11]{Billing}, we know that $C(\Sigma \times [a,b])$ is also separable. In particular we can apply the Skorohod Representation Theorem, see \cite[Theorem 6.7]{Billing}, from which we conclude that there exists a probability space $(\Omega, \mathcal{F}, \mathbb{P})$, which supports $C(\Sigma \times [a,b])$-valued random variables $\mathcal{Y}^n$ and $\mathcal{L}$ such that $\mathcal{Y}^n \rightarrow \mathcal{L}$ for every $\omega \in \Omega$ and such that under $\mathbb{P}$ the law of $\mathcal{Y}^n$ is $\mathbb{P}_{avoid, n}^{a,b,\vec x^n,\vec y^n\infty, -\infty}$, while under $\mathbb{P}$ the law of $\mathcal{L}$ is $\mathbb{P}_{avoid}^{a,b,\vec x,\vec y,\infty, -\infty}$. 

We now let $a_n,b_n \in \Lambda_{n^2}$ be such that $a_n \leq a'$, $b_n \geq b'$ and $a_n$ is maximal while $b_n$ is minimal subject to these conditions. We also fix $N_0$ sufficiently large so that $n \geq N_0$ implies that $t_s < a_n$ or $t_s > b_n$ for $s \in S^c$ such that $n_s \in K$. Let $\vec{X}^n, \vec{Y}^n$ be defined through 
$$\vec{X}^n_i = \mathcal{Y}^n(k_1 + i - 1, a_n) \mbox{ and } \vec{Y}^n_i = \mathcal{Y}^n(k_1 + i - 1, b_n) \mbox{ for $i \in \llbracket 1, k_2 - k_1 + 1 \rrbracket $}.$$
Since $\mathcal{Y}^n$ is uniformly distributed on all (finitely many) avoiding trajectories from $\vec{x}^n$ to $\vec{y}^n$, we conclude that the restriction of $\mathcal{Y}^n$ to $K \times [a_n, b_n]$ is precisely uniformly distributed on all (finitely many) avoiding trajectories from $\vec{X}^n$ to $\vec{Y}^n$, conditioned on staying below $f_n = \mathcal{Y}^n_{k_1-1}$ and above $g_n = \mathcal{Y}^n_{k_2+1}$ with the usual convention that $f_n = \infty$ if $k_1 = 1$ and $g_n = -\infty$ if $k_2 = k$. The latter observation allows us to deduce that
\begin{equation}\label{S51Red3}
\mathbb{E} \left[\prod_{i = 1}^m f_i(\mathcal{Y}^n(n_i, t_i)) \right] = \mathbb{E} \left[\prod_{s \in S^c}f_s(\mathcal{Y}^n(n_s, t_s))  \cdot \mathbb{E}_{avoid,n} \left[ \prod_{s \in S} f_s(\mathcal{Z}^n(n_s -k_1 + 1, t_s)) \right] \right],
\end{equation}
where we have written $\mathbb{E}_{avoid,n}$ in place of $\mathbb{E}^{a_n, b_n, \vec{X}^n, \vec{Y}^n,f_n, g_n}_{avoid,n} $ to ease the notation.\\

In view of our Skorohod embedding space $(\Omega, \mathcal{F}, \mathbb{P})$ we know that almost surely $f_n \rightarrow f$ on $[a,b]$, where $f(x) = \mathcal{L}(k_1 - 1;x)$ if $k_1 \geq 2$ or $f = \infty$ if $k_1 = 1$. Analogously $g_n \rightarrow g$ on $[a,b]$, where $g(x) = \mathcal{L}(k_2 + 1;x)$ if $k_2\leq k-1$ or $g = -\infty$ if $k_2 = k$. In addition, $\vec{X}^n \rightarrow \vec{X}$ and $\vec{Y}^n \rightarrow \vec{Y}$ where 
$$\vec{X}_i = \mathcal{L}(k_1 + i - 1, a') \mbox{ and } \vec{Y}_i = \mathcal{L}^n(k_1 + i - 1, b') \mbox{ for $i \in \llbracket 1, k_2 - k_1 + 1 \rrbracket $}.$$
Furthermore, by Definition \ref{DefAvoidingLaw} and Lemma \ref{NoTouch} we know that $\mathbb{P}$-almost surely $\vec{X}, \vec{Y} \in \weyl_k.$ Consequently, from Lemma \ref{lem:RW} we conclude that $\mathbb{P}$-almost surely
\begin{equation}\label{S51Red4}
\lim_{n\rightarrow \infty} \mathbb{E}_{avoid,n} \left[ \prod_{s \in S} f_s(\mathcal{Z}^n(n_s -k_1 + 1, t_s)) \right] = \mathbb{E}_{avoid}^{a',b', \vec{x}, \vec{y}, f, g} \left[ \prod_{s \in S} f_s(\tilde{\mathcal{Q}}(n_s, t_s)) \right].
\end{equation}
Finally, the continuity of $f_i$ and and the $\omega$-wise convergence of $\mathcal{Y}^n$ to $\mathcal{L}$ implies that for every $\omega \in \Omega$ we have
\begin{equation}\label{S51Red5}
\lim_{n\rightarrow \infty}\prod_{s \in S^c}f_s(\mathcal{Y}^n(n_s, t_s))  = \prod_{s \in S^c}f_s(\mathcal{L}(n_s, t_s)) .
\end{equation}
Equation (\ref{S51Red1}) is now a consequence of (\ref{S51Red2}), (\ref{S51Red3}), (\ref{S51Red4}) and (\ref{S51Red5}) after an application of the bounded convergence theorem.

%
\subsection{Proofs of Lemmas \ref{MCLxy} and \ref{MCLfg} }\label{Section5.3} The main result of this section is as follows.

\begin{lemma}\label{MCLxyfg} Assume the same notation as in Definition \ref{DefAvoidingLaw}. Fix $k \in \mathbb{N}$, $a < b$ and and two continuous functions $g^t, g^b: [a,b] \rightarrow \mathbb{R} \cup \{- \infty \}$ such that $g^t(x) \geq g^b(x)$ for all $x \in [a,b]$. We also fix $\vec{x}, \vec{y}, \vec{x}', \vec{y}' \in \mathbb{R}_{>}^k$ such that $g^b(a) < x_k$, $g^b(b) < y_k$, $g^t(a) < x_k'$, $g^t(b) < y_k'$  and $x_i \leq x_i'$, $y_i \leq y_i'$ for $i = 1,\dots, k$. Then there exists a probability space $(\Omega, \mathcal{F}, \mathbb{P})$, which supports two $\llbracket 1, k \rrbracket$-indexed line ensembles $\mathcal{L}^t$ and $\mathcal{L}^b$ on $[a,b]$ such that the law of $\mathcal{L}^{t}$ {\big (}resp. $\mathcal{L}^b${\big )} under $\mathbb{P}$ is $\mathbb{P}_{avoid}^{a,b, \vec{x}', \vec{y}', \infty, g^t}$ {\big (}resp. $\mathbb{P}_{avoid}^{a,b, \vec{x}, \vec{y}, \infty, g^b}${\big )} and such that $\mathbb{P}$-almost surely we have $\mathcal{L}_i^t(x) \geq \mathcal{L}^b_i(x)$ for all $i = 1,\dots, k$ and $x \in [a,b]$.
\end{lemma}

It is clear that Lemmas \ref{MCLxy} and \ref{MCLfg} both follow from Lemma \ref{MCLxyfg}. The reason we keep the statements of the two lemmas separate earlier in the paper is that it makes their application a bit more transparent in the main body of text.

\begin{proof}[Proof of Lemma \ref{MCLxyfg}] We assume the same notation as in Lemma \ref{MCLxyfg} and also Definition \ref{Grids2}. Specifically, we fix $\Sigma = \llbracket 1, k\rrbracket$ and $\Lambda = [a,b]$. For clarity we split the proof into three steps.\\

{\bf \raggedleft Step 1.} We choose any sequences $\vec{x}^n, \vec{y}^n, \vec{u}^n, \vec{v}^n \in \weyl_k \cap\R_{n}^k$ such that for each $n \in \mathbb{N}$ we have $x_i^n \leq u_i^n$, $y_i^n \leq v_i^n$ for $i = 1, \dots,k$ and also such that $\lim_{n \to \infty} \vec{x}^n = \vec x$, $\lim_{n \to \infty} \vec{y}^n = \vec y$, $\lim_{n \rightarrow \infty}\vec{u}^n = \vec{x}'$ and $\lim_{n\rightarrow \infty} \vec{v}^n = \vec{y}'$. It follows from Lemma \ref{lem:RW} applied to $a' =a$, $b' = b$ that there exists $N_0 \in \mathbb{N}$ such that if $n \geq N_0$ we have that $\mathbb{P}_{avoid, n}^{a,b, \vec{x}^n, \vec{y}^n, \infty, g^b}$ and $\mathbb{P}_{avoid, n}^{a,b, \vec{u}^n, \vec{u}^n, \infty, g^t}$ are well-defined. 

We claim that we can construct sequences of probability spaces $(\Omega_n, \mathcal{F}_n, \mathbb{P}_n)$ for $n \geq N_0$ that support $\llbracket 1, k \rrbracket$-indexed line ensembles $\mathcal{Y}^n$ and $\mathcal{Z}^n$, whose laws are $\mathbb{P}_{avoid, n}^{a,b, \vec{x}^n, \vec{y}^n, \infty, g^b}$ and  $\mathbb{P}_{avoid, n}^{a,b, \vec{u}^n, \vec{v}^n, \infty, g^t}$ respectively, such that for each $\omega \in \Omega_n$ we have
\begin{equation}\label{S52Red1}
\mathcal{Y}^n(\omega)(i,x) \leq \mathcal{Z}^n(\omega)(i,x) \mbox{ for $i = 1, \dots, k$ and $x \in [a,b]$}.
\end{equation}
We show (\ref{S52Red1}) in the next step. Here we assume its validity and conclude the proof of the lemma.\\

It follows from Lemma \ref{lem:RW} that $\mathcal{Y}^n$ converge weakly to $\mathbb{P}_{avoid}^{a,b, \vec{x}, \vec{y}, \infty, g^b}$ and $\mathcal{Z}^n$ converge weakly to $\mathbb{P}_{avoid}^{a,b, \vec{x}', \vec{y}', \infty, g^t}$ as $n \to \infty$. In particular, the latter sequences of measures are relatively compact, which by the separability and completeness of $C(\llbracket 1, k \rrbracket \times [a,b])$ implies that these sequences are tight, cf. \cite[Theorem 5.2]{Billing}. In particular, the sequence of random variables $(\mathcal{Y}^n, \mathcal{Z}^n)$ on $(\Omega_n, \mathcal{F}_n, \mathbb{P}_n)$ (viewed as $C(\llbracket 1, k \rrbracket \times [a,b]) \times C(\llbracket 1, k \rrbracket \times [a,b])$-valued random variables with the product topology and corresponding Borel $\sigma$-algebra) are also tight. 

By Prohorov's theorem, see \cite[Theorem 5.1]{Billing}, we conclude that the sequence of laws of $(\mathcal{Y}^n, \mathcal{Z}^n)$ is relatively compact. Let $n_m$ be a subsequence such that $(\mathcal{Y}^{n_m}, \mathcal{Z}^{n_m})$ converge weakly. By the Skorohod Representation Theorem, see \cite[Theorem 6.7]{Billing}, we conclude that there exists a probability space $(\Omega, \mathcal{F}, \mathbb{P})$, which supports $C(\Sigma \times [a,b])$-valued random variables $\mathcal{Y}^{n_m}, \mathcal{Z}^{n_m}$ and $\mathcal{Y}, \mathcal{Z}$ such that
\begin{enumerate}
\item  $\mathcal{Y}^{n_m} \rightarrow \mathcal{Y}$ for every $\omega \in \Omega$ as $m \rightarrow \infty$; 
\item $\mathcal{Z}^{n_m} \rightarrow \mathcal{Z}$  for every $\omega \in \Omega$ as $m \rightarrow \infty$; 
\item  under $\mathbb{P}$ the law of $\mathcal{Y}^{n_m}$ is $\mathbb{P}_{avoid, n}^{a,b,\vec x^{n_m},\vec y^{n_m}\infty, g^b}$;
\item   under $\mathbb{P}$ the law of $\mathcal{Z}^{n_m}$ is $\mathbb{P}_{avoid, n}^{a,b,\vec u^{n_m},\vec v^{n_m}\infty, g^t}$;
\item  $\mathbb{P}$-almost surely we have $\mathcal{Y}^{n_m}(i,x) \leq \mathcal{Z}^{n_m}(i,x) \mbox{ for $m \geq 1$, $i = 1, \dots, k$ and $x \in [a,b]$}$.
\end{enumerate}
Since $\mathcal{Y}^n$ converge weakly to $\mathbb{P}_{avoid}^{a,b, \vec{x}, \vec{y}, \infty, g^b}$ and $\mathcal{Z}^n$ converge weakly to $\mathbb{P}_{avoid}^{a,b, \vec{x}', \vec{y}', \infty, g^t}$ we conclude that under $\mathbb{P}$ the variables $\mathcal{Y}$ and $\mathcal{Z}$ have laws $\mathbb{P}_{avoid}^{a,b, \vec{x}, \vec{y}, \infty, g^b}$ and $\mathbb{P}_{avoid}^{a,b, \vec{x}', \vec{y}', \infty, g^t}$ respectively. Also conditions (1), (2) and (5) above imply that $\mathbb{P}$-almost surely we have 
$$\mathcal{Y}(i,x) \leq \mathcal{Z}(i,x) \mbox{ for $i = 1, \dots, k$ and $x \in [a,b]$}.$$
Consequently, taking the above probability space $(\Omega, \mathcal{F}, \mathbb{P})$ and setting $(\mathcal{L}^t, \mathcal{L}^b) =(\mathcal{Y}, \mathcal{Z})$ we obtain the statement of the lemma.\\

{\bf \raggedleft Step 2.} In this step we prove (\ref{S52Red1}). Our approach will closely follow the one in \cite[Section 6]{CorHamA}.

Let $Y_n$ and $Z_n$ denote the (finite) sets of possible elements in $C(\llbracket 1, k \rrbracket \times [a,b])$ that the line ensembles $\mathcal{Y}^n$ and $\mathcal{Z}^n$ can take with positive probability. We will construct a continuous time Markov chain $(A_t, B_t)$ taking values in $Y_n \times  Z_n$, such that:
\begin{enumerate}
\item $A_t$ and $B_t$ are each Markov in their own filtration;
\item $A_t$ is irreducible and has invariant measure $\mathbb{P}_{avoid, n}^{a,b, \vec{x}^n, \vec{y}^n, \infty, g^b}$;
\item $B_t$ is irreducible and has invariant measure $\mathbb{P}_{avoid, n}^{a,b, \vec{u}^n, \vec{v}^n, \infty, g^t}$;
\item for every $t \geq 0$ we have $A_t(i,x) \leq B_t(i,x)$ for $i \in \llbracket 1, k \rrbracket$ and $x \in [a,b]$.
\end{enumerate}
We will construct the Markov chain $(A_t, B_t)$ in the next step. Here we assume we have such a construction and conclude the proof of (\ref{S52Red1}).\\

From \cite[Theorems 3.5.3 and 3.6.3]{Norris} we know that $A_N$ weakly converges to $\mathbb{P}_{avoid, n}^{a,b, \vec{x}^n, \vec{y}^n, \infty, g^b}$ and $B_N$ weakly converges to $\mathbb{P}_{avoid, n}^{a,b, \vec{u}^n, \vec{v}^n, \infty, g^t}$ as $N \rightarrow \infty$. In particular, we see that $A_N, B_N$ are tight and then so is the sequence $(A_N,B_N)$. By Prohorov's theorem, see \cite[Theorem 5.1]{Billing}, we conclude that the sequence of laws of $(A_N, B_N)$ is relatively compact. Let $N_m$ be a subsequence such that  $(A_{N_m}, B_{N_m})$ converge weakly. By the Skorohod Representation Theorem, see \cite[Theorem 6.7]{Billing}, we conclude that there exists a probability space $(\Omega, \mathcal{F}, \mathbb{P})$, which supports $C(\Sigma \times [a,b])$-valued random variables $\mathcal{A}_m, \mathcal{B}_m$ and $\mathcal{A}, \mathcal{B}$ such that
\begin{itemize}
\item  $\mathcal{A}_{m} \rightarrow \mathcal{A}$ for every $\omega \in \Omega$ as $m \rightarrow \infty$; 
\item $\mathcal{B}_m \rightarrow \mathcal{B}$  for every $\omega \in \Omega$ as $m \rightarrow \infty$; 
\item  under $\mathbb{P}$ the law of $(\mathcal{A}_m, \mathcal{B}_m)$ is the same as that of $(A_{N_m}, B_{N_m})$.
\end{itemize}
The weak convergence of $A_N, B_N$ implies that $\mathcal{A}$ has law $\mathbb{P}_{avoid, n}^{a,b, \vec{x}^n, \vec{y}^n, \infty, g^b}$ and $\mathcal{B}$ has law $\mathbb{P}_{avoid, n}^{a,b, \vec{u}^n, \vec{v}^n, \infty, g^t}$. Furthermore, the fourth condition in the beginning of the step shows that $\mathcal{A}(i,x) \leq \mathcal{B}(i,x)$ for $i \in \llbracket 1, k \rrbracket$ and $x \in [a,b]$. Consequently, we can take $(\Omega_n, \mathcal{F}_n, \mathbb{P}_n)$ to be the above space $(\Omega, \mathcal{F}, \mathbb{P})$ and set $(\mathcal{Y}^n, \mathcal{Z}^n)= (\mathcal{A}, \mathcal{B})$. This proves (\ref{S52Red1}).\\

{\bf \raggedleft Step 3.} In this final step we construct the chain $(A_t, B_t)$, satisfying the four conditions in the beginning of Step 2. We first describe the initial state of the Markov chain $(A_0, B_0)$. Notice that if $y \in Y_n$ there is a natural way to encode $y(i,x)$ for $i \in \llbracket 1, k \rrbracket$ by a list of $n^2$ symbols $\{-1, 0, 1\}$, where the $j$-th symbol is precisely 
$$\frac{y(i, a + j \cdot \Delta^t_n) - y(i, a + (j-1) \cdot \Delta^t_n)}{\Delta_n^x}.$$
 We define the lexicographic ordering on the set of all such lists of symbols (where of course $1 > 0 > -1$). If we look at $y(i,x)$, we see that there is a maximal sequence of $n^2$ symbols, which consists of 
$\left\lfloor \frac{1}{2} \cdot \left( \frac{y_i^n - x_i^n}{ \Delta_x}+ n^2 \right)\right\rfloor$
symbols $1$, followed by a $0$ if 
$ \frac{1}{2} \cdot \left( \frac{y_i^n - x_i^n}{ \Delta_x}+ n^2 \right)\not \in \mathbb{Z},$
 followed by 
$\left\lfloor \frac{1}{2} \cdot \left( \frac{x_i^n - y_i^n}{ \Delta_x}+ n^2 \right) \right\rfloor$
 symbols $-1$. We call the curve corresponding to this list $y^{max}(i,\cdot)$. One directly checks that $y^{max}=(y^{\max}(1, \cdot), \cdots, y^{\max}(k, \cdot)) \in Y_n.$ In showing the last statement, we implicitly used that $n \geq N_0$ so that $\mathbb{P}_{avoid, n}^{a,b, \vec{x}^n, \vec{y}^n, \infty, g^b}$ is well-defined. 

 We analogously define $z^{max} \in Z_n$ by replacing everywhere above $x_i^n, y_i^n$ with $u_i^n, v_i^n$ respectively. Again one needs to use that $n \geq N_0$ so that $\mathbb{P}_{avoid, n}^{a,b, \vec{u}^n, \vec{v}^n, \infty, g^t}$ is well-defined. One further checks directly that $y^{max}(i,x) \leq z^{max}(i,x)$ for all $i \in \llbracket 1, k \rrbracket$ and $x\in [a,b]$. The state $(y^{max}, z^{max})$ is the initial state of our chain.\\

We next describe the dynamics. For each point $r \in \Lambda_{n^2} \cap (a,b)$, each $i \in \llbracket 1,k \rrbracket$ and each $\delta \in \{-1, 0, 1\}$ we have an independent Poisson clock, ringing with rate $1$. When the clock corresponding to $(r,i,\delta)$ rings at time $T$, we update $(A_{T-},B_{T-})$ as follows. We erase the part of $A_{T-}(i,x)$ (resp. $B_{T-}(i,x)$) for $x \in [r - \Delta_n^t, r+\Delta_n^t]$ and replace that piece with two linear pieces connecting the points $\left(r-\Delta_n^t, A_{T-}(i, r- \Delta_n^t)\right)$ and $\left(r+\Delta_n^t, A_{T-}(i, r+ \Delta_n^t)\right)$ with $\left(r, A_{T-}(i, r) + \delta \cdot \Delta_n^x\right)$ (resp. $\left(r-\Delta_n^t, B_{T-}(i, r- \Delta_n^t)\right)$ and $\left(r+\Delta_n^t, B_{T-}(i, r+ \Delta_n^t)\right)$ with $\left(r, B_{T-}(i, r) + \delta \cdot \Delta_n^x\right)$). If the resulting $C(\Sigma \times [a,b])$-valued element is in $Y_n$ (resp. $Z_n$) we set $A_T$ (resp. $B_T$) to it. Otherwise, we set $A_T$ (resp. $B_T$) to $A_{T-}$ (resp. $B_{T-}$). This defines the dynamics.

It is clear from the above definition that $(A_t, B_t)$ is a Markov chain and that $A_t$ and $B_t$ are individually Markov in their own filtration. Moreover, one directly verifies that the uniform measure on $Y_n$ (resp. $Z_n$) is invariant under the above dynamics. The latter observations show that conditions (1), (2) and (3) in the beginning of Step 2 all hold. We are thus left with verifying condition (4). By construction, we know that $A_t(i,x) \leq B_t(i,x)$ for all $i \in \llbracket 1, k\rrbracket$ and $x \in [a,b]$ when $t =0$. What remains to be seen is that the update rule, explained in the previous paragraph, maintains this property for all $t \geq 0$. \\

For the sake of contradiction, suppose that $A_{T-} \in Y_n$, $B_{T-}\in Z_n$ are such that $A_{T-}(i,x) \leq B_{T-}(i,x)$ for all $i \in \llbracket 1, k\rrbracket$ and $x \in [a,b]$, but that after the $(r,i, \delta)$-clock has rung at time $T$ we no longer have that $A_{T}(i,x) \leq B_{T}(i,x)$ for all $i \in \llbracket 1, k\rrbracket$ and $x \in [a,b]$. By the formulation of the dynamics, the latter implies that $A_T(i,r) > B_T(i,r)$, and is only possible if $A_{T-}(i,r) = B_{T-}(i,r)$. In particular, we distinguish two cases: (C1) $\delta = 1$ and $A_{T}(i,r) = A_{T-}(i,r) + \Delta_n^x$, while $B_{T}(i,r) = B_{T-}(i,r)$ or (C2) $\delta = -1$ and $A_{T}(i,r) = A_{T-}(i,r)$, while $B_{T}(i,r) = B_{T-}(i,r) - \Delta_n^x$. 

In the case (C1), the fact that $B_{T}(i,r) = B_{T-}(i,r)$, means that the $C(\Sigma \times [a,b])$-valued element obtained from $B_{T-}$ by erasing the part of $B_{T-}(i,x)$ for $x \in [r - \Delta_n^t, r+\Delta_n^t]$ and replacing it with two linear pieces connecting $\left(r-\Delta_n^t, B_{T-}(i, r- \Delta_n^t)\right)$ and $\left(r+\Delta_n^t, B_{T-}(i, r+ \Delta_n^t)\right)$ with $\left(r, B_{T-}(i, r) +  \Delta_n^x\right)$ is not in $Z_n$. This means that 
$$B_{T}(i,r) + \Delta^x_n \geq \max\left( B_{T}(i -1,r), B_{T}(i, r- \Delta_n^t) + 2 \Delta^x_n,  B_{T}(i, r+ \Delta_n^t) + 2 \Delta^x_n\right).$$
Here the convention is $B_T(0,x) = \infty$. But then since 
$$A_{T-}(i,r) = B_{T-}(i,r), \mbox{ and $A_{T-}(i,x) \leq B_{T-}(i,x)$ for all $i \in \llbracket 1, k\rrbracket$ and $x \in [a,b]$},$$
we conclude that 
$$A_{T-}(i,r) + \Delta^x_n \geq \max\left( A_{T}(i -1,r), A_{T}(i, r- \Delta_n^t) + 2 \Delta^x_n,  A_{T}(i, r+ \Delta_n^t) + 2 \Delta^x_n\right),$$
again with the convention $A_T(0, x) = \infty$. The latter contradicts the fact that $A_{T}(i,r) = A_{T-}(i,r) + \Delta_n^x$ since it implies $A_T \not \in Y_n$.

In the case (C2), the fact that $A_{T}(i,r) = A_{T-}(i,r)$, means that the $C(\Sigma \times [a,b])$-valued element obtained from $A_{T-}$ by erasing the part of $A_{T-}(i,x)$ for $x \in [r - \Delta_n^t, r+\Delta_n^t]$ and replacing it with two linear pieces connecting $\left(r-\Delta_n^t, A_{T-}(i, r- \Delta_n^t)\right)$ and $\left(r+\Delta_n^t, A_{T-}(i, r+ \Delta_n^t)\right)$ with $\left(r, A_{T-}(i, r) -  \Delta_n^x\right)$ is not in $Y_n$. This means that 
$$A_{T}(i,r) - \Delta^x_n \leq \min\left( A_{T}(i +1,r), A_{T}(i, r- \Delta_n^t) - 2 \Delta^x_n,  A_{T}(i, r+ \Delta_n^t) - 2 \Delta^x_n\right),$$
where $A_{T}(k+1,x) = g^b(x)$. But then since 
$$A_{T-}(i,r) = B_{T-}(i,r), \mbox{ and $A_{T-}(i,x) \leq B_{T-}(i,x)$ for all $i \in \llbracket 1, k\rrbracket$ and $x \in [a,b]$},$$
we conclude that 
$$B_{T-}(i,r) - \Delta^x_n \leq \min\left( B_{T}(i -1,r), B_{T}(i, r- \Delta_n^t) - 2 \Delta^x_n,  B_{T}(i, r+ \Delta_n^t) - 2 \Delta^x_n\right),$$
where $B_{T}(k+1,x) = g^t(x)$ and we used that $g^t(x) \geq g^b(x)$ for all $x \in [a,b]$. The latter; however, contradicts the fact that $B_{T}(i,r) = B_{T-}(i,r) - \Delta_n^x$, as it implies that $ B_T \not \in Z_n$. Overall, we see that we reach a contradiction in both cases. This means that $(A_t, B_t)$ satisfies all four conditions in Step 2, which concludes the proof of the lemma.
\end{proof}

%% file: PD.bbl
\providecommand{\bysame}{\leavevmode\hbox to3em{\hrulefill}\thinspace}
\providecommand{\MR}{\relax\ifhmode\unskip\space\fi MR }
\providecommand{\MRhref}[2]{%
  \href{http://www.ams.org/mathscinet-getitem?mr=#1}{#2}
}
\providecommand{\href}[2]{#2}
\begin{thebibliography}{10}

\bibitem{Bar01}
Y.~Baryshnikov, \emph{Gues and queues}, Probab. Theory Relat. Fields
  \textbf{119} (2001), 256--274.

\bibitem{Billing}
P.~Billingsley, \emph{Convergence of {P}robability {M}easures, 2nd ed}, John
  Wiley and Sons, New York, 1999.

\bibitem{BorCor}
A.~Borodin and I.~Corwin, \emph{Macdonald processes}, Probab. Theory Relat.
  Fields \textbf{158} (2014), 225--400.

\bibitem{BCFV}
A.~Borodin, I.~Corwin, P.~Ferrari, and B.~Vet{\H o}, \emph{Height fluctuations
  for the stationary {K}{P}{Z} equation}, Math. Phys. Anal. Geom. \textbf{18}
  (2015), 20, https://doi.org/10.1007/s11040-015-9189-2.

\bibitem{BCF}
A.~Borodin, I.~Corwin, and P.~L. Ferrari, \emph{Anisotropic $(2+1)$d growth and
  {G}aussian limits of $q$-{W}hittaker processes}, Probab. Theory Relat. Fields
  \textbf{172} (2018), 245--321.

\bibitem{CU2}
I.~Corwin, \emph{The {K}ardar-{P}arisi-{Z}hang equation and universality
  class}, Random Matrices: Theory Appl. \textbf{1} (2012).

\bibitem{CD}
I.~Corwin and E.~Dimitrov, \emph{Transversal fluctuations of the {A}{S}{E}{P},
  {S}tochastic six vertex model, and {H}all-{L}ittlewood {G}ibbsian line
  ensembles}, Commun. Math. Phys. \textbf{363} (2018), 435--501.

\bibitem{CorHamA}
I.~Corwin and A.~Hammond, \emph{Brownian {G}ibbs property for {A}iry line
  ensembles}, Invent. Math. \textbf{195} (2014), 441--508.

\bibitem{CorHamK}
\bysame, \emph{K{P}{Z} line ensemble}, Probab. Theory Relat. Fields
  \textbf{166} (2016), 67--185.

\bibitem{COSZ}
I.~Corwin, N.~O'Connell, T.~Sepp{\"a}l{\"a}inen, and N.~Zygouras,
  \emph{Tropical combinatorics and {W}hittaker functions}, Duke Math. J.
  \textbf{163} (2014), 513--563.

\bibitem{CorwinSun}
I.~Corwin and X.~Sun, \emph{Ergodicity of the {A}iry line ensemble}, Electron.
  Commun. Probab. \textbf{19} (2014), no.~49, 1--11.

\bibitem{DNV}
D.~Dauvergne, M.~Nica, and B.~Vir{\' a}g, \emph{Uniform convergence to the
  {A}iry line ensemble},  (2019), Preprint: arXiv:1907.10160.

\bibitem{Def10}
M.~Defosseux, \emph{Orbit measures, random matrix theory and interlaced
  determinantal processes}, Ann. Inst. Henri Poincar{\' e} Probab. Stat.
  \textbf{46} (2010), 209--249.

\bibitem{ED}
E.~Dimitrov, \emph{K{P}{Z} and {A}iry limits of {H}all-{L}ittlewood random
  plane partitions}, Ann. Inst. H. Poincar{\' e} Probab. Statist. \textbf{54}
  (2018), 640--693.

\bibitem{ED2}
\bysame, \emph{Six-vertex models and the {GUE}-corners process}, Int. Math.
  Res. Notices (2018), doi.org/10.1093/imrn/rny072.

\bibitem{DW}
E.~Dimitrov and X.~Wu, \emph{{KMT} coupling for random walk bridges},  (2019),
  Preprint: arXiv:1905.13691.

\bibitem{Durrett}
R.~Durrett, \emph{Probability: theory and examples, {F}ourth edition},
  Cambridge University Press, Cambridge, 2010.

\bibitem{Dys62}
F.J. Dyson, \emph{A {B}rownian motion model for the eigenvalues of a random
  matrix}, J. Math. Phys. \textbf{3} (1962), 1191--1198.

\bibitem{EK08}
P.~Eichelsbacjer and W.~K{\"o}nig, \emph{Ordered random walks}, Electron. J.
  Probab. \textbf{13} (2008), 1307--1336.

\bibitem{GorinAlt}
V.~Gorin, \emph{From alternating sign matrices to the {G}aussian unitary
  ensemble}, Comm. Math. Phys. \textbf{332} (2014), 437--447.

\bibitem{Ham4}
A.~Hammond, \emph{Brownian regularity for the {A}iry line ensemble, and
  multi-polymer watermelons in {B}rownian last passage percolation},  (2016),
  Preprint: arXiv:1609.029171.

\bibitem{Ham1}
\bysame, \emph{Exponents governing the rarity of disjoint polymers in
  {B}rownian last passage percolations},  (2017), Preprint: arXiv:1709.04110.

\bibitem{Ham2}
\bysame, \emph{Modulus of continuity of polymer weight profiles in {B}rownian
  last passage percolation},  (2017), Preprint: arXiv:1709.04115.

\bibitem{Ham3}
\bysame, \emph{A patchwork quilt sewn from {B}rownian fabric: regularity of
  polymer weight profiles in {B}rownian last passage percolation}, Forum Math.
  Pi \textbf{7} (2019), Preprint: arXiv:1709.04113.

\bibitem{JN06}
K.~Johansson and E.~Nordenstam, \emph{Eigenvalues of {G}{U}{E} minors},
  Electron. J. Probab. \textbf{11} (2006), 1342--1371.

\bibitem{KS}
I.~Karatzas and S.~Shreve, \emph{Brownian motion and stochastic calculus},
  Volume 113 of Graduate Texts in Mathematics, Springer, 1988.

\bibitem{KQ}
A.~Krishnan and J.~Quastel, \emph{{T}racy-{W}idom fluctuations for
  perturbations of the log-gamma polymer in intermediate disorder}, Ann. Appl.
  Probab. \textbf{28} (2018), no.~6, 3736--3764.

\bibitem{MZ}
D.~Mason and H.~Zhou, \emph{Quantile coupling inequalities and their
  applications}, Probab. Surveys \textbf{9} (2012), 439--479.

\bibitem{Munkres}
J.~Munkres, \emph{Topology, 2nd ed}, Prentice Hall, Inc., Upper Saddle River,
  NJ, 2003.

\bibitem{NZ}
V.-L. Nguyen and N.~Zygouras, \emph{Variants of geometric {R}{S}{K}, geometric
  {P}{N}{G} and the multipoint distribution of the log-gamma polymer}, Int.
  Math. Res. Notices (2016).

\bibitem{Nor09}
E.~Nordenstam, \emph{Interlaced particles in tilings and random matrices},
  Doctoral Thesis, KTH (2009).

\bibitem{Norris}
J.R. Norris, \emph{Markov {C}hains}, Cambridge University Press, New York, NY,
  1997.

\bibitem{OCY}
N.~O'Connell and M.~Yor, \emph{{B}rownian analogues of {B}urke's theorem},
  Stoch. Proc. Appl. \textbf{96} (2001), 285--304.

\bibitem{OR}
A.~Okounkov and N.~Reshetikhin, \emph{The birth of a random matrix}, Mosc.
  Math. J. \textbf{6} (2006), 553--566.

\bibitem{OV}
G.~Olshanski and A.~Vershik, \emph{Ergodic unitarily invariant measures on the
  space of infinite {H}ermitian matrices}, Amer. Math. Soc. Transl. Ser. 2
  \textbf{175} (1996), 137--175.

\bibitem{Spohn}
M.~Pr{\" a}hofer and H.~Spohn, \emph{Scale invariance of the {P}{N}{G}
  {D}roplet and the {A}iry process}, J. Stat. Phys. \textbf{108} (2002),
  1071--1106.

\bibitem{Rudin}
W.~Rudin, \emph{Principles of mathematical analsyis, 3rd ed.}, New York:
  McGraw-Hill, 1964.

\bibitem{Sep12}
T.~Sepp{\"a}l{\"a}inen, \emph{Scaling for a one-dimensional directed polymer
  with boundary}, Ann. Probab. \textbf{40} (2012), 19--73.

\bibitem{Stein3}
E.~Stein and R.~Shakarchi, \emph{Real analysis}, Princeton University Press,
  Princeton, 2005.

\bibitem{Wu19}
X.~Wu, \emph{Tightness of discrete {G}ibbsian line ensembles with exponential
  interaction {H}amiltonians},  (2019), Preprint: arXiv:1909.00946.

\end{thebibliography}
